\documentclass[12pt]{amsart}
\usepackage[margin=1in]{geometry}
\usepackage[utf8]{inputenc} 
\usepackage{amsmath,amsthm,amssymb,amsfonts, dsfont, graphicx,mathtools,mathrsfs,tikz,physics, stmaryrd, bm, tikz-cd,tkz-euclide,comment, enumerate, appendix, todonotes, hyperref}

\numberwithin{equation}{section}
\newtheorem{theorem}{Theorem}[section]
\newtheorem{proposition}[theorem]{Proposition}
\newtheorem{lemma}[theorem]{Lemma}
\newtheorem{corollary}[theorem]{Corollary}

\theoremstyle{definition}
\newtheorem{remark}[theorem]{Remark}

\newtheorem*{remark*}{Remark} 
\newtheorem*{proposition*}{Proposition}
\newtheorem*{acknowledgement*}{Acknowledgement}

\newtheorem{definition}[theorem]{Definition}
\newcommand\nnorm[1]{\left\lVert#1\right\rVert}
\renewcommand\norm[1]{\left\lvert#1\right\rvert}
\newcommand{\overbar}[1]{\mkern 1.5mu\overline{\mkern-1.5mu#1\mkern-1.5mu}\mkern 1.5mu}

\newcommand{\Z}{\mathbb{Z}}

\newcommand{\Q}{\mathbb{Q}}
\newcommand{\R}{\mathbb{R}}

\newcommand{\C}{\mathbb{C}}

\begin{document}
\title{Systems of bihomogeneous forms of small bidegree}
\author[L. Hochfilzer]{Leonhard Hochfilzer}
\address{Mathematisches Institut, Bunsenstraße 3-5, 37073 Göttingen, Germany}
\email{leonhard.hochfilzer@mathematik.uni-goettingen.de}

\begin{abstract}
We use the circle method to count the number of integer solutions to systems of bihomogeneous equations of bidegree $(1,1)$ and $(2,1)$ of bounded height in lopsided boxes. Previously, adjusting Birch's techniques to the bihomogeneous setting, Schindler showed an asymptotic formula  provided the number of variables grows at least quadratically with the number of equations considered. Based on recent methods by Rydin Myerson we weaken this assumption and show that the number of variables only needs to satisfy a linear bound in terms of the number of equations.	
\end{abstract}

\maketitle
\tableofcontents

\section{Introduction}
Studying the number of rational solutions of bounded height on a system of equations is a fundamental tool in order to understand the distribution of rational points on varieties. A longstanding result by Birch~\cite{birch_forms} establishes an asymptotic formula for the number of integer points of bounded height that are solutions to a system of homogeneous forms of the same degree in a general setting, provided the number of variables is sufficiently big relative to the singular locus of the variety defined by the system of equations. This was recently improved upon by Rydin Myerson~\cite{myerson_quadratic, myerson_cubic} whenever the degree is $2$ or $3$. These results may be used in order to prove Manin's conjecture for certain Fano varieties, which arise as complete intersections in projective space.

Analogous to Birch's result, Schindler studied systems of bihomogeneous forms~\cite{schindler_bihomogeneous}. Using the hyperbola method, Schindler established Manin's conjecture for certain bihomogeneous varieties as a result~\cite{schindler_manin_biprojective}. The aim of this paper is to improve Schindler's result by applying the ideas of Rydin Myerson to the bihomogeneous setting.

Consider a system of bihomogeneous forms $\bm{F}(\bm{x},\bm{y})= \left( F_1(\bm{x},\bm{y}), \hdots, F_R(\bm{x},\bm{y}) \right)$ with integer coefficients in variables $\bm{x} = (x_1, \hdots, x_{n_1})$ and $\bm{y} = (y_1, \hdots, y_{n_2})$. We assume that all of the forms have the same bidegree, which we denote by $(d_1,d_2)$ for nonnegative integers $d_1,d_2$. By this we mean that for any scalars $\lambda, \mu \in \mathbb{C}$ we have
\begin{equation*}
	F_i(\lambda \bm{x}, \mu \bm{y}) = \lambda^{d_1} \mu^{d_2} F_i(\bm{x},\bm{y}),
\end{equation*}
for all $i = 1, \hdots, R$. This system defines a biprojective variety $V \subset \mathbb{P}_{\mathbb{Q}}^{n_1-1} \times \mathbb{P}_{\mathbb{Q}}^{n_2-1}$. One can also interpret the system in the affine variables $(x_1, \hdots, x_{n_1}, y_1, \hdots, y_{n_2})$ and thus $\bm{F}(\bm{x},\bm{y})$ also defines an affine variety which we will denote by $V_0 \subset \mathbb{A}_{\mathbb{Q}}^{n_1+n_2}$.
We are interested in studying the set of integer solutions to this system of bihomogeneous equations.
Consider two boxes $\mathcal{B}_i \subset [-1,1]^{n_i}$ where each edge is of side length at most one and they are all parallel to the coordinate axes. In order to study the questions from an analytic point of view, for $P_1, P_2 > 1$ we define the following counting function
\begin{equation*}
	N(P_1,P_2) = \# \{ (\bm{x}, \bm{y}) \in \mathbb{Z}^{n_1} \times \mathbb{Z}^{n_2} \mid \bm{x}/P_1 \in \mathcal{B}_1, \;  \bm{y}/P_2 \in \mathcal{B}_2, \;  \bm{F}(\bm{x},\bm{y}) = \bm{0} \}.
\end{equation*}
Generalising the work of Birch~\cite{birch_forms}, Schindler~\cite{schindler_bihomogeneous} used the circle method to achieve an asymptotic formula for $N(P_1,P_2)$ as $P_1,P_2 \rightarrow \infty$ provided certain conditions on the number of variables are satisfied, as we shall describe below.
Before we can state Schindler's result, consider the varieties $V_1^*$ and $V_2^*$ in $\mathbb{A}_\mathbb{Q}^{n_1+n_2}$ to be defined by the equations 
\begin{equation*}
	\mathrm{rank}\left(\frac{\partial F_i }{\partial x_j} \right)_{i,j} < R, \quad \text{and} \quad  \mathrm{rank}\left(\frac{\partial F_i }{\partial y_j} \right)_{i,j} < R
\end{equation*}
respectively. Assume that $V_0$ is a complete intersection, which means that $\dim V_0 = n_1+n_2-R$. Write $b = \max\left\{ \frac{\log(P_1)}{\log(P_2)},1\right\}$ and $u = \max\left\{ \frac{\log (P_2)}{\log(P_1)}, 1 \right\}$. If $n_i > R$ and
\begin{equation} \label{eq.cond_schindler}
	n_1+n_2 - \dim V_i^* > 2^{d_1+d_2-2} \max \{R(R+1)(d_1+d_2-1), R(bd_1 + ud_2) \},
\end{equation}	
is satisfied, for $i=1,2$ then Schindler showed the asymptotic formula
\begin{equation} \label{eq.schindler_asymptotic}
	N(P_1,P_2) = \sigma P_1^{n_1-Rd_1} P_2^{n_2-Rd_2} + O\left(P_1^{n_1-Rd_1} P_2^{n_2-Rd_2} \min \{P_1,P_2 \}^{-\delta}\right),
\end{equation}
for some $\delta > 0$ and where $\sigma$ is positive if the system $\bm{F}(\bm{x},\bm{y}) = \bm{0}$ has a smooth $p$-adic zero for all primes $p$, and the variety $V_0$ has a smooth real zero in $\mathcal{B}_1 \times \mathcal{B}_2$. 
In the case when the equations $F_1(\bm{x},\bm{y}), \hdots, F_R(\bm{x}, \bm{y})$ define a smooth complete intersection $V$, and where the bidegree is $(1,1)$ or $(2,1)$ the goal of this paper is to improve the restriction on the number of variables~\eqref{eq.cond_schindler} and still show~\eqref{eq.schindler_asymptotic}.

The result by Schindler generalises a well-known result by Birch~\cite{birch_forms}, which deals with systems of homogeneous equations; Let $\mathcal{B} \subset [-1,1]^n$ be a box containing the origin with side lengths at most $1$ and edges parallel to the coordinate axes. Given homogeneous equations $G_1(\bm{x}), \hdots, G_R(\bm{x})$ with rational coefficients of common degree $d$ define the counting function 
\begin{equation*}
	N(P) = \# \{ \bm{x} \in \mathbb{Z}^n \colon \bm{x}/P \in \mathcal{B}, \; G_1(\bm{x}) = \cdots = G_R(\bm{x}) = 0 \}.
\end{equation*}
Write $V^* \subset \mathbb{A}^n_{\mathbb{Q}}$ for the variety defined by 
\begin{equation*}
		\mathrm{rank} \left( \frac{\partial G_i}{\partial x_j} \right)_{i,j} < R,
\end{equation*}
commonly referred to as the \emph{Birch singular locus}. Assuming that $G_1, \hdots, G_R$ define a complete intersection $X \subset \mathbb{P}^{n-1}_{\mathbb{Q}}$ and that the number of variables satisfies
\begin{equation} \label{eq.assumption_birch_quadr}
		n-\dim V^* > R(R+1)(d-1)2^{d-1},
\end{equation}
then Birch showed
\begin{equation} \label{eq.birch_asymptotic}
		N(P) = \tilde{\sigma} P^{n-dR} + O(P^{n-dR-\varepsilon}),
\end{equation} 
	where $\tilde{\sigma} > 0$  if the system $\bm{G}(\bm{x})$ has a smooth $p$-adic zero for all primes $p$ and the variety $X$ has a smooth real zero in $\mathcal{B}$.

Building on ideas of M\"uller~\cite{mueller2, mueller1} on quadratic Diophantine inequalities, Rydin Myerson improved Birch's theorem. He weakened the assumption on the number of variables in the cases $d=2,3$~\cite{myerson_quadratic, myerson_cubic} whenever $R$ is reasonably large. Assuming that $X \subset \mathbb{P}^{n-1}_{\mathbb{Q}}$ defines a complete intersection, he was able to replace the condition in \eqref{eq.assumption_birch_quadr} by
\begin{equation} \label{eq.myerson_assumption}
	n - \sigma_{\mathbb{R}} > d2^d R,
\end{equation}
where 
\[
\sigma_{\mathbb{R}} = 1+ \max_{\bm{\beta} \in \mathbb{R}^R \setminus \{\bm{0}\}} \dim \mathrm{Sing} \mathbb{V}(\bm{\beta} \cdot \bm{G}),
\]
and where $\mathbb{V}(\bm{\beta} \cdot \bm{G})$ is the pencil defined by $\sum_{i=1}^R \beta_i G(\bm{x})$ in $\mathbb{P}^{n-1}_{\mathbb{Q}}$. We note at this point that several other authors have replaced the Birch singular locus condition with weaker assumptions, such as Schindler~\cite{schindler2015variant} and Dietmann~\cite{dietmann_weyl} who also considered dimensions of pencils, and very recently Yamagishi~\cite{yamagishi2023birch} who replaced the Birch singular locus with a condition regarding the Hessian of the system. Returning to Rydin Myerson's result if $X$ is non-singular then one can show
\[
\sigma_{\mathbb{R}} \leq R-1
\]
and in this case if $n \geq (d 2^d+1)R$ then one obtains the desired asymptotic.
Notably, the work of Rydin Myerson showed the number of variables $n$ thus only has to grow linearly in the number of equations $R$, whereas $R$ appeared quadratically in Birch's work. If $d\geq 4$ he showed that for \textit{generic} systems of forms it suffices to assume \eqref{eq.myerson_assumption} for the asymptotic~\eqref{eq.birch_asymptotic} to hold. Generic here means that the set of coefficients is required to lie in some non-empty Zariski open subset of the parameter space of coefficients of the equations.

Our goal in this paper is to generalise the results obtained by Rydin Myerson to the case of bihomogeneous varieties whenever the bidegree of the forms is $(1,1)$ or $(2,1)$. Those two cases correspond to degrees $2$ and $3$ in the homogeneous case, respectively. We call a bihomogeneous form \emph{bilinear} if the bidegree is $(1,1)$. Given a bilinear form $F_i(\bm{x},\bm{y})$ we may write it as
\begin{equation*}
	F_i(\bm{x},\bm{y}) = \bm{y}^T A_i \bm{x},
\end{equation*}
for some $n_2 \times n_1$-dimensional matrices $A_i$ with rational entries. Given $\bm{\beta} \in \mathbb{R}^R$ write
\begin{equation*}
	A_{\bm{\beta}} = \sum_{i = 1}^R \beta_i A_i.
\end{equation*}
Regarding $A_{\bm{\beta}}$ as a map $\mathbb{R}^{n_1} \rightarrow \mathbb{R}^{n_2}$ and and $A_{\bm{\beta}}^T$ as a map $\mathbb{R}^{n_2} \rightarrow \mathbb{R}^{n_1}$ we define the  quantities
\begin{equation*}
	\sigma_{\mathbb{R}}^{(1)} \coloneqq \max_{\bm{\beta} \in \mathbb{R}^R \setminus \{ \bm{0}\}} \dim \ker (A_{\bm{\beta}}), \quad \text{and} \quad
	\sigma_{\mathbb{R}}^{(2)} \coloneqq \max_{\bm{\beta} \in \mathbb{R}^R \setminus \{ \bm{0}\}} \dim \ker (A_{\bm{\beta}}^T).
\end{equation*}
We state our first theorem for systems of bilinear forms. Since the situation is completely symmetric with respect to the $\bm{x}$ and $\bm{y}$ variables if the forms are bilinear, we may without loss of generality assume $P_1  \geq P_2$ in the counting function, and still obtain the full result.
\begin{theorem} \label{thm.bilinear}
	Let $F_1(\bm{x},\bm{y}), \hdots, F_R(\bm{x}, \bm{y})$ be bilinear forms with integer coefficients such that the biprojective variety $\mathbb{V}(F_1, \hdots, F_R) \subset \mathbb{P}_{\mathbb{Q}}^{n_1-1} \times \mathbb{P}_{\mathbb{Q}}^{n_2-1}$ is a complete intersection.  Let $P_1\geq P_2> 1$, write $b = \frac{\log(P_1)}{\log(P_2)}$ and assume further that 
	\begin{equation} \label{eq.assumption_n_i_-sigma_i}
		n_i - \sigma_{\mathbb{R}}^{(i)} > (2b+2)R
	\end{equation}
	holds for $i = 1,2$. Then there exists some $\delta > 0$ depending at most on $b$, $\bm{F}$, $R$ and $n_i$ such that 
	\begin{equation*}
		N(P_1,P_2) = \sigma P_1^{n_1-R} P_2^{n_2-R} + O(P_1^{n_1-R} P_2^{n_2-R-\delta})
	\end{equation*}
	holds, where $\sigma > 0$ if the system $\bm{F}(\bm{x},\bm{y}) = \bm{0}$ has a smooth $p$-adic zero for all primes $p$ and if the variety $V_0$ has a smooth real zero in $\mathcal{B}_1 \times \mathcal{B}_2$.
	
	Moreover, if we assume $\mathbb{V}(F_1, \hdots, F_R) \subset \mathbb{P}^{n_1-1}_{\mathbb{Q}} \times \mathbb{P}_{\mathbb{Q}}^{n_2-1}$ to be smooth the same conclusions hold if we assume
	\begin{equation*}
		\min \{n_1, n_2 \} > (2b+2)R \quad \text{and} \quad n_1+n_2 > (4b+5)R
	\end{equation*}
	instead of \eqref{eq.assumption_n_i_-sigma_i}.
\end{theorem}

We now move on to systems of forms $F_1(\bm{x},\bm{y}), \hdots, F_R(\bm{x},\bm{y})$ of bidegree $(2,1)$. We may write such a form $F_i(\bm{x},\bm{y})$  as
\begin{equation*}
	F_i(\bm{x},\bm{y}) = \bm{x}^T H_i(\bm{y}) \bm{x},
\end{equation*}
where $H_i(\bm{y})$ is a symmetric $n_1 \times n_1$ matrix whose entries are linear forms in the variables $\bm{y} = (y_1, \hdots, y_{n_2})$. % We say that such a form is \emph{diagonal} if $A_i(\bm{y})$ is diagonal. 
Similarly to above, given $\bm{\beta} \in \mathbb{R}^R$ we write
\begin{equation*}
	H_{\bm{\beta}}(\bm{y}) = \sum_{i = 1}^R\beta_i H_i(\bm{y}).
\end{equation*}
Given $\ell \in\{ 1, \hdots, n_2\}$ write $\bm{e}_\ell \in \R^{n_2}$ for the standard unit basis vectors. Write \[\mathbb{V}(\bm{x}^T H_{\bm{\beta}} (\bm{e}_\ell) \bm{x})_{\ell = 1, \hdots, n_2} = \mathbb{V}(\bm{x}^T H_{\bm{\beta}} (\bm{e}_1) \bm{x}, \hdots, \bm{x}^T H_{\bm{\beta}} (\bm{e}_{n_2}) \bm{x}) \subset \mathbb{P}_{\mathbb{Q}}^{n_1-1}\]
for this intersection of pencils, and
 define
\begin{equation} \label{eq.defn_s_1}
	s^{(1)}_{\mathbb{R}} \coloneqq 1 + \max_{\bm{\beta} \in \mathbb{R}^R \setminus \{ \bm{0} \} } \dim \mathbb{V}(\bm{x}^T H_{\bm{\beta}} (\bm{e}_\ell) \bm{x})_{\ell = 1, \hdots, n_2}.
\end{equation}
Further write $\mathbb{V}(H_{\bm{\beta}}(\bm{y}) \bm{x})$ for the biprojective variety defined by the system of equations
\[
\mathbb{V}(H_{\bm{\beta}}(\bm{y}) \bm{x}) = \mathbb{V}((H_{\bm{\beta}}(\bm{y}) \bm{x})_1, \hdots, (H_{\bm{\beta}}(\bm{y}) \bm{x})_{n_1}) \subset \mathbb{P}_{\mathbb{Q}}^{n_1-1} \times \mathbb{P}_{\mathbb{Q}}^{n_2-1}
\]
and define
\begin{equation} \label{eq.defn_s_2}
	s_{\mathbb{R}}^{(2)} \coloneqq \left\lfloor \frac{\max_{\bm{\beta} \in \mathbb{R}^R \setminus \{ 0\}} \dim \mathbb{V}(H_{\bm{\beta}}(\bm{y}) \bm{x})}{2} \right\rfloor +1, 
\end{equation}
where $\lfloor x \rfloor$ denotes the largest integer $m$ such that $m \leq x$.

\begin{theorem} \label{thm.2,1_different_dimensions}
	Let $F_1(\bm{x},\bm{y}), \hdots, F_R(\bm{x},\bm{y})$ be bihomogeneous forms with integer coefficients of bidegree $(2,1)$ such that the biprojective variety $\mathbb{V}(F_1, \hdots, F_R) \subset \mathbb{P}_{\mathbb{Q}}^{n_1-1} \times \mathbb{P}_{\mathbb{Q}}^{n_2-1}$ is a complete intersection.  Let $P_1,P_2 > 1$ be real numbers. Write $b = \max\left\{\frac{\log(P_1)}{\log(P_2)}, 1 \right\}$	 and $u = \max\left\{\frac{\log(P_2)}{\log(P_1)}, 1 \right\}$. Assume further that	
	\begin{equation} \label{eq.assumption_n_i_(2,1)_introduction}
		n_1 - s_{\mathbb{R}}^{(1)} > (8b+4u)R \quad \text{and} \quad \frac{n_1+n_2}{2} - s_{\mathbb{R}}^{(2)} > (8b+4u)R
	\end{equation}
	is satisfied. Then there exists some $\delta > 0$ depending at most on $b$, $u$, $R$, $n_i$ and $\bm{F}$ such that 
	\begin{equation} \label{eq.asymptotic_bidegree_2_1}
		N(P_1,P_2) = \sigma P_1^{n_1-2R} P_2^{n_2-R} + O(P_1^{n_1-2R} P_2^{n_2-R} \min\{ P_1,P_2\}^{-\delta})
	\end{equation}
	holds, where $\sigma > 0$ if the system $\bm{F}(\bm{x},\bm{y}) = \bm{0}$ has a smooth $p$-adic zero for all primes $p$, and if the variety $V_0$ has a smooth real zero in $\mathcal{B}_1 \times \mathcal{B}_2$.
	
	If we assume that $\mathbb{V}(F_1, \hdots, F_R) \subset \mathbb{P}_{\mathbb{Q}}^{n_1-1} \times \mathbb{P}_{\mathbb{Q}}^{n_2-1}$ is smooth, then the same conclusions hold if we assume
	\begin{equation} \label{eq.assumptions_n_i_(2,1)_smooth_case}
		n_1 > (16b+8u+1)R, \quad \text{and} \quad n_2 > (8b+4u+1)R 
	\end{equation}
	instead of \eqref{eq.assumption_n_i_(2,1)_introduction}.
\end{theorem}

We remark that we preferred to give conditions in terms of the geometry of the variety regarded as a biprojective variety, as opposed to an affine variety. The reason for this is the potential application of this result to proving Manin's conjecture for this variety, which will be addressed in due course.
 
 Compared to the result by Schindler we thus basically remove the assumption that the number of variables needs to  grow at least quadratically in $R$. In particular, if the complete intersection defined by the system is assumed to be smooth, then our results requires fewer variables than Schindler's provided
 \[
 d_1b+d_2u < \frac{R+1}{2}
 \]
 is satisfied, in the cases $(d_1,d_2) = (1,1)$ or $(2,1)$. In particular, if $R$ is large this means our result provides significantly more flexibility in the choice of $u$ and $b$.
  
One cannot hope to achieve the asymptotic formula~\eqref{eq.schindler_asymptotic} in general where a condition of the shape $n_i > R(bd_1+ud_2)$ is not present. To see this note that the counting function satisfies
\[
N(P_1,P_2) \gg P_1^{n_1} + P_2^{n_2},
\]
coming from the solutions when $x_1 = \cdots = x_{n_1}= 0$ and $y_1 = \cdots = y_{n_2}= 0$. The asymptotic formula~\eqref{eq.schindler_asymptotic} thus implies 
\[
P_i^{n_i} \ll P_1^{n_1-d_1R} P_2^{n_2-d_2R},
\]
for $i=1,2$. Noting that $P_1^{u} = P_2$ if $u > 1$ and $P_2^{b}  =P_1$ if $b >1$ and comparing the exponents one necessarily finds $n_i > R(bd_1+ud_2)$.

If the forms are diagonal then one can take boxes $\mathcal{B}_i$, which avoid the coordinate axes in order to remedy this obstruction. In fact this is the approach taken by Blomer and Br\"{u}dern~\cite{blomer_bruedern_hyp} and they proved an asymptotic formula of a system of multihomogeneous equations without a restriction on the number of variables similar to the type described above.

If the forms are not diagonal the problem still persists, even if one were to take boxes avoiding the coordinate axes. In general there may be 'bad' vectors $\bm{y}$ away from the coordinate axes such that
\[
\#\left\{ \bm{x} \in \mathbb{Z}^{n_1} \colon \bm{F}(\bm{x},\bm{y}) = \bm{0}, \norm{\bm{x}} \leq P_1 \right\} \gg P_1^{n_1-a},
\]
where $a < d_1R$ for example. This is in contrast to the diagonal case, where the only vectors $\bm{y}$ where this occurs lie on at least one coordinate axis. It would be interesting to consider a modified counting function where one excludes such vectors $\bm{y}$, and analogously 'bad' vectors $\bm{x}$. In a general setting it seems difficult to control the set of such vectors. In particular, it is not clear how one would deal with the Weyl differencing step if one were to consider such a counting function.

\subsection{Manin's conjecture}
Let $V \subset \mathbb{P}_{\mathbb{Q}}^{n_1-1} \times \mathbb{P}_{\mathbb{Q}}^{n_2-1}$ be a non-singular complete intersection defined by a system of forms $F_i(\bm{x},\bm{y})$, $i = 1, \hdots, R$ of common bidegree $(d_1,d_2)$. Assume $n_i > d_iR$ so that $V$ is a Fano variety, which means that the inverse of the canonical bundle in the Picard group, the \emph{anticanonical bundle}, is very ample. For a field $K$, write $V(K)$ for the set of $K$-rational points of $V$. In the context of Manin's conjecture we define this to be the set of $K$-morphisms
\[
\mathrm{Spec}(K) \rightarrow V_K,
\]
where $V_K$ denotes the base change of $V$ to the field $K$.
 For a subset $U(\Q) \subset V(\Q)$ and $P \geq 1$ consider the counting function
\[
N_U(P) = \# \left\{ (\bm{x}, \bm{y}) \in U(\mathbb{Q}) \colon H(\bm{x},\bm{y}) \leq P \right\},
\]
where $H(\cdot, \cdot)$ is the \emph{anticanonical height} induced by the anticanonical bundle and a choice of global sections. 
In our case one such height may be explicitly given as follows. If $(\bm{x}, \bm{y}) \in U(\mathbb{Q})$ we may pick representatives $\bm{x} \in \mathbb{Z}^{n_1}$, and $\bm{y} \in \mathbb{Z}^{n_2}$ such that $(x_1, \hdots, x_{n_1}) = (y_1, \hdots, y_{n_2}) = 1$ and we define 
\[
H(\bm{x},\bm{y}) = \left( \max_i \lvert x_i \rvert \right)^{n_1-Rd_1}  \left( \max_i \lvert y_i \rvert \right)^{n_2-Rd_2}.
\]
Manin's Conjecture in this context states that, provided $V$ is a Fano variety such that $V(\mathbb{Q}) \subset V$ is Zariski dense, there exists a subset $U(\mathbb{Q}) \subset V(\mathbb{Q})$ where $(V \setminus U)(\mathbb{Q})$ is a \emph{thin} set such that
\[
N_U(P) \sim c P (\log P)^{\rho-1},
\]
where $\rho$ is the Picard rank of the variety $V$ and $c$ is a constant as predicted and interpreted by Peyre~\cite{peyre_constant}. We briefly recall the definition of a thin set, according to Serre~\cite{serre_thin_sets}.  First recall a set $A \subset V(K)$ is of type
\begin{itemize}
	\item[($C_1$)] if $A \subseteq W(K)$, where $W \subsetneq V$ is Zariski closed, 
	\item[($C_2$)] if $A \subseteq \pi(V'(K))$, where $V'$ is irreducible such that $\dim V = \dim V'$, where $\pi \colon V' \rightarrow V$ is a generically finite morphism of degree at least $2$.
\end{itemize}
Now a subset of the $K$-rational points of $V$ is \emph{thin} if it is a finite union of sets of type $(C_1)$ or $(C_2)$. Originally Batyrev--Manin~\cite{batyrev_manin} conjectured that it suffices to assume that $(V \setminus U)$ is Zariski closed, but there have been found various counterexamples to this, the first one being due to Batyrev--Tschinkel~\cite{batyrev_tschinkel_counterexample}.

In~\cite{schindler_manin_biprojective} Schindler showed an asymptotic formula of the shape above, if $V$ is smooth and $d_1,d_2 \geq 2$ and 
\[
n_i > 3 \cdot 2^{d_1+d_2} d_1 d_2 R^3 + R
\]
is satisfied for $i=1,2$. If $R=1$ she moreover verified that the constant obtained agrees with the one predicted by Peyre, and thus proved Manin's conjecture for bihomogeneous hypersurfaces when the conditions above are met. The proof uses the asymptotic~\eqref{eq.schindler_asymptotic} established in~\cite{schindler_bihomogeneous} along with uniform counting results on fibres. That is, for a vector $\bm{y} \in \mathbb{Z}^{n_2}$ one may consider the counting function 
\[
N_{\bm{y}}(P) = \# \left\{ \bm{x} \in \mathbb{Z}^{n_1} \colon \bm{F}(\bm{x},\bm{y}) = \bm{0}, \lvert \bm{x} \rvert \leq P \right\},
\]
and to understand its asymptotic behaviour uniformly means to understand the dependence of $\bm{y}$ on the constant in the error term. Similarly she considered $N_{\bm{x}}(P)$ for 'good' $\bm{x}$ and combined the three resulting estimates to obtain an asymptotic formula for the number of solutions $\widetilde{N}(P_1,P_2)$ to the system $\bm{F}(\bm{x},\bm{y}) = \bm{0}$, where $\lvert \bm{x} \rvert \leq P_1$, $\lvert \bm{y} \rvert \leq P_2$ and $\bm{x},\bm{y}$ are 'good'. Considering only 'good' tuples essentially removes a closed subset from $V$, and thus, after an application of a slight modification of the hyperbola method developed as in~\cite{blomer_bruedern_hyp} she obtained an asymptotic formula for $N_U(P)$ of the desired shape.

In forthcoming work the result established in Theorem~\ref{thm.2,1_different_dimensions} will be used in verifying Manin's Conjecture for $V$, when $(d_1,d_2) = (2,1)$ in fewer variables than would be expected using Schindler's method as described above. Further, since the Picard rank of $V$ is strictly greater than $1$, it would be interesting to consider the \emph{all heights approach} as suggested by Peyre~\cite[Question V.4.8]{peyre_book_all_heights}. As noted by Peyre himself, in the case when a variety has Picard rank $1$, the answer to his Question 4.8 follows provided one can prove Manin's conjecture with respect to the height function induced by the anticanonical bundle.

Schindler's results have been improved upon in a few special cases. Browning and Hu showed Manin's conjecture in the case of smooth biquadratic hypersurfaces in $\mathbb{P}^{n-1}_{\mathbb{Q}} \times \mathbb{P}^{n-1}_{\mathbb{Q}}$ if the number of variables satisfies $n>35$. If the bidegree is $(2,1)$ then Hu showed that $n>25$ suffices in order to obtain Manin's conjecture. Systems of bilinear varieties are flag varieties and thus Manin's conjecture follows from the result for flag varieties, which was proven by Franke, Manin and Tschinkel~\cite{jens_manin_tschinkel_manin_flag} using the theory of Eisenstein series. In the special case when the variety is defined by $\sum_{i=0}^s x_iy_i = 0$ then Robbiani~\cite{robbiani_bilinear} showed how one may use the circle method to establish Manin's conjecture if $s \geq 3$, which was later improved to $s \geq 2$ by Spencer~\cite{spencer_manin_bilinear}. 

\subsection*{Acknowledgements}
The author would like to thank Damaris Schindler for many helpful comments and conversations regarding this project. The author would further like to thank Christian Bernert and Simon Rydin Myerson for helpful conversations.
\subsection*{Conventions}
The symbol $\varepsilon >0$ is an arbitrarily small value, which we may redefine whenever convenient, as is usual in analytic number theory. Given forms $g_\ell$, $\ell = 1, \hdots, k$  we write $\mathbb{V}(g_\ell)_{\ell = 1, \hdots, k}$ or sometimes just $\mathbb{V}(g_\ell)_\ell$ for the intersection $\mathbb{V}(g_1, \hdots, g_k)$. Further, we may sometimes consider a vector of forms $\bm{h} = (h_1, \hdots, h_k)$ and we similarly write $\mathbb{V}(\bm{h})$ for the intersection $\mathbb{V}(h_1, \hdots, h_k)$.

For a real number $x \in \R$ we will write $e(x) = e^{2 \pi i x}$. We will use Vinogradov's notation $O(\cdot)$ and $\ll$. 

We shall repeatedly use the convention that the dimension of the empty set $-1$.
\section{Multilinear forms}
Both Theorem~\ref{thm.bilinear} and Theorem~\ref{thm.2,1_different_dimensions} follow from a more general result. If we have control over the number of 'small' solutions to the associated linearised forms then we can show that the asymptotic~\eqref{eq.schindler_asymptotic} holds. More explicitly, given a bihomogeneous form $F(\bm{x},\bm{y})$ with integer coefficients of bidegree $(d_1,d_2)$ for positive integers $d_1,d_2$, we may write it as
\begin{equation*}
F(\bm{x},\bm{y}) = \sum_{\bm{j}} \sum_{\bm{k}} F_{\bm{j},\bm{k}} x_{j_1} \cdots x_{j_{d_1}} y_{k_1} \cdots y_{k_{d_2}},
\end{equation*}
where the coefficients $F_{\bm{j},\bm{k}} \in \Q$ are symmetric in $\bm{j}$ and $\bm{k}$. 
We define the associated multilinear form
\begin{equation*}
	\Gamma_{F}(\widetilde{\bm{x}}, \widetilde{\bm{y}}) \coloneqq d_1! d_2! \sum_{\bm{j}} \sum_{\bm{k}} F_{\bm{j},\bm{k}} x^{(1)}_{j_1} \cdots x^{(d_1)}_{j_{d_1}} y^{(1)}_{k_1} \cdots y^{(d_2)}_{k_{d_2}},
\end{equation*}
where $\widetilde{\bm{x}} = (\bm{x}^{(1)}, \hdots, \bm{x}^{(d_1)})$ and $\widetilde{\bm{y}} = (\bm{y}^{(1)}, \hdots, \bm{y}^{(d_2)})$ for vectors $\bm{x}^{(i)}$ of $n_1$ variables and vectors $\bm{y}^{(i)}$ of $n_2$ variables. Write further $\widehat{\bm{x}} = (\bm{x}^{(1)}, \hdots, \bm{x}^{(d_1-1)})$ and $\widehat{\bm{y}} = (\bm{y}^{(1)}, \hdots, \bm{y}^{(d_2-1)})$. Given $\bm{\beta} \in \mathbb{R}^R$ we define the auxiliary counting function $N_1^{\mathrm{aux}}(\bm{\beta}; B)$ to be the number of integer vectors satisfying $\widehat{\bm{x}} \in (-B,B)^{(d_1-1)n_1}$ and $\widetilde{\bm{y}} \in (-B,B)^{d_2n_2}$ such that 
\begin{equation*}
	\norm{\Gamma_{\bm{\beta} \cdot \bm{F}}(\widehat{\bm{x}}, \bm{e}_\ell, \widetilde{\bm{y}})} < \nnorm{\bm{\beta} \cdot \bm{F}}_\infty B^{d_1+d_2-2},
\end{equation*}
for $\ell = 1, \hdots, n_1$ where $\nnorm{\bm{\beta} \cdot \bm{F}}_\infty \coloneqq \frac{1}{d_1!d_2!} \max_{\bm{j}, \bm{k}} \norm{\frac{\partial^{d_1+d_2}(\bm{\beta} \cdot \bm{F})}{\partial x_{j_1} \cdots \partial x_{j_{d_1}} \partial y_{k_1} \cdots \partial y_{k_{d_2}}  }}$.  We  define $N_2^{\mathrm{aux}}(\bm{\beta};B)$ analogously.

The technical core of this paper is the following theorem.
\begin{theorem} \label{thm.n_aux_imply_result}
	Assume $n_1,n_2 > (d_1+d_2)R$ and let $\bm{F}(\bm{x},\bm{y}) = (F_1(\bm{x},\bm{y}), \hdots, F_R(\bm{x},\bm{y}))$ be a system of bihomogeneous forms with integer coefficients of common bidegree $(d_1,d_2)$ such that the variety $\mathbb{V}(\bm{F}) \subset \mathbb{P}_{\mathbb{Q}}^{n_1-1} \times \mathbb{P}_{\mathbb{Q}}^{n_2-1}$ is a complete intersection. Let $P_1,P_2 > 1$ and write $b = \max \left\{\log(P_1) /\log (P_2),1 \right\}$ and $u = \max \left\{\log(P_2) /\log (P_1),1 \right\}$. 
	
	Assume there exist $C_0 \geq 1$ and $\mathscr{C} > (bd_1+ud_2)R$ such that for all $\bm{\beta} \in \mathbb{R}^R \setminus \{ \bm{0}\}$ and all $B>0$ we have
	\begin{equation} \label{eq.N_aux_general_condition}
		N_i^{\mathrm{aux}}(\bm{\beta};B) \leq C_0 B^{d_1n_1 + d_2n_2 - n_i - 2^{d_1+d_2-1}\mathscr{C}}
	\end{equation}
	for $i=1,2$.  There exists some $\delta > 0$ depending on $b$, $u$, $C_0$, $R$, $d_i$ and $n_i$ such that
	\begin{equation*} %\label{eq.main_asymptotic}
		N(P_1,P_2) = \sigma P_1^{n_1-d_1R}P_2^{n_2-d_2R} + O \left(P_1^{n_1-d_1R}P_2^{n_2-d_2R} \min\{P_1,P_2 \}^{-\delta} \right).
	\end{equation*}
	The factor $\sigma = \mathfrak{I} \mathfrak{S}$ is the product of the singular integral $\mathfrak{I}$ and the singular series $\mathfrak{S}$, as defined in~\eqref{eq.def_singular_integral} and~\eqref{eq.def_singular_series}, respectively. Moreover, if the system $\bm{F}(\bm{x},\bm{y}) = \bm{0}$ has a non-singular real zero in $\mathcal{B}_1 \times \mathcal{B}_2$ and a non-singular $p$-adic zero for every prime $p$, then $\sigma >0$.
	\end{theorem}
While showing that~\eqref{eq.N_aux_general_condition} holds is rather straightforward when the bidegree is $(1,1)$ it becomes significantly more difficult when the bidegree increases. In fact, in Rydin Myerson's work a similar upper bound on a similar auxiliary counting function needs to be shown. He is successful in doing so when the degree is $2$ or $3$ and the system defines a complete intersection, but for higher degrees he was only able to show this upper bound for generic systems. 
Our strategy is as follows. We will establish Theorem~\ref{thm.n_aux_imply_result} in Section~\ref{sec.weyl_differencing_etc} and Section~\ref{sec.circle_method} and then use this to show Theorem~\ref{thm.bilinear} and Theorem~\ref{thm.2,1_different_dimensions} in Section~\ref{sec.bilinear_proof}  and in Section~\ref{sec.proof.2-1}.

\section{Geometric preliminaries}

The following Lemma is taken from \cite{schindler_manin_biprojective}.
\begin{lemma}[Lemma 2.2 in \cite{schindler_manin_biprojective}] \label{lem.geometry_intersections}
	Let $W$ be a smooth variety that is complete over some algebraically closed field and consider a closed irreducible subvariety $Z \subseteq W$ such that $\dim Z \geq 1$. Given an effective divisor $D$ on $W$ then the dimension of every irreducible component of $D \cap Z$ is at least $\dim Z-1$.	If $D$ is moreover ample we have in addition that $D \cap Z$ is nonempty.
\end{lemma}
In particular the following corollary will be very useful.

\begin{corollary} \label{cor.geometry_intersections}
	Let $V \subseteq \mathbb{P}^{n_1-1}_{\mathbb{C}} \times \mathbb{P}^{n_2-1}_{\mathbb{C}}$ be a closed variety such that $\dim V \geq 1$. Consider $H = \mathbb{V}(f)$ where $f(\bm{x}, \bm{y})$ is a polynomial of bidegree at least $(1,1)$ in the variables $(\bm{x}, \bm{y}) = (x_1, \hdots, x_{n_1}, y_1, \hdots, y_{n_2})$. Then
	\begin{equation*}
		\dim (V \cap H) \geq \dim V - 1, 
	\end{equation*}
	in particular $V \cap H$ is non-empty.
\end{corollary}
\begin{proof}
	Since the bidegree of $f$ is at least $(1,1)$ we have that $H$ defines an effective and ample divisor on $\mathbb{P}^{n_1-1}_{\mathbb{C}} \times \mathbb{P}^{n_2-1}_{\mathbb{C}}$. We apply Lemma \ref{lem.geometry_intersections} with  $W = \mathbb{P}^{n_1-1}_{\mathbb{C}} \times \mathbb{P}^{n_2-1}_{\mathbb{C}}$, $D = H$ and $Z$ any irreducible component of $V$. 
\end{proof}

\begin{lemma} \label{lem.sing_bf_small}
	Let $\bm{F}(\bm{x},\bm{y})$ be a system of $R$ bihomogeneous equations of the same bidegree $(d_1,d_2)$ with $d_1, d_2 \geq 1$. Assume that $\mathbb{V}(\bm{F}) \subset \mathbb{P}_{\mathbb{C}}^{n_1-1} \times \mathbb{P}_{\mathbb{C}}^{n_2-1}$ is a smooth complete intersection. Given $\bm{\beta} \in \mathbb{R}^R \setminus \{ \bm{0}\}$ we have
	\begin{equation*}
		\dim \mathrm{Sing} \mathbb{V}(\bm{\beta} \cdot \bm{F}) \leq R-2,
	\end{equation*}
	where we write $\bm{\beta} \cdot \bm{F} = \sum_i \beta_i F_i$.
\end{lemma}
\begin{proof}
	 The singular locus of $\mathbb{V}(\bm{\beta} \cdot \bm{F})$ is given by
	 \begin{equation*}
	 	\mathrm{Sing} \mathbb{V}(\bm{\beta} \cdot \bm{F}) = \mathbb{V}\left(\frac{\partial (\bm{\beta} \cdot \bm{F})}{\partial x_j} \right)_{j = 1, \hdots, n_1} \cap \mathbb{V}\left(\frac{\partial (\bm{\beta} \cdot \bm{F})}{\partial y_j} \right)_{j = 1, \hdots, n_2}.
	 \end{equation*}
	 Assume without loss of generality $\beta_R \neq 0$ so that $\mathbb{V}(\bm{F}) = \mathbb{V}(F_1, \hdots, F_{R-1}, \bm{\beta} \cdot \bm{F})$. We claim that we have the following inclusion 
	 \begin{equation} \label{eq.singular_loci_containment}
	 	\mathbb{V}(F_1, \hdots, F_{R-1}) \cap \mathrm{Sing} \mathbb{V}(\bm{\beta} \cdot \bm{F}) \subseteq \mathrm{Sing} \mathbb{V}(\bm{F}).
	 \end{equation}
	 	 To see this note first that $\mathbb{V}(F_1, \hdots, F_{R-1}) \cap \mathrm{Sing} \mathbb{V}(\bm{\beta} \cdot \bm{F}) \subseteq \mathbb{V}(\bm{F})$. Further, the Jacobian matrix $J(\bm{F})$ of $\bm{F}$ is given by
	 \begin{equation*}
	 	J(\bm{F}) = \left( \frac{\partial F_i}{\partial z_j} \right)_{ij},
	 \end{equation*}
	 where $i = 1, \hdots, R$ and $z_j$ ranges through $x_1, \hdots, x_{n_1}, y_1, \hdots, y_{n_2}$.
	 Now if the equations
	 \begin{equation*}
	 	\frac{\partial (\bm{\beta} \cdot \bm{F})}{\partial x_j} = \frac{\partial (\bm{\beta} \cdot \bm{F})}{\partial y_j} = 0,
	 \end{equation*}
	 are satisfied then this implies that the rows of $J(\bm{F})$ are linearly dependent. Since $\mathbb{V}(\bm{F})$ is a complete intersection we deduce the claim. 
 
	 Assume now for a contradiction that $\dim \mathrm{Sing} \mathbb{V}(\bm{\beta} \cdot \bm{F}) \geq R-1$ holds. Applying Corollary~\ref{cor.geometry_intersections} $(R-1)$-times with $V = \mathrm{Sing} \mathbb{V} (\bm{\beta} \cdot \bm{F})$, noting that the bidegree of $F_i$ is at least $(1,1)$, we find 
	 \begin{equation*}
	 	\mathbb{V}(F_1, \hdots, F_{R-1}) \cap \mathrm{Sing} \mathbb{V}(\bm{\beta} \cdot \bm{F}) \neq \emptyset.
	 \end{equation*}
	 This contradicts~\eqref{eq.singular_loci_containment} since $\mathrm{Sing} \mathbb{V}(\bm{F})  =\emptyset$ by assumption.
\end{proof}

\begin{lemma} \label{lem.dimensions_varieties_not_too_big}
	Let $n_1 \leq n_2$ be two positive integers. For $i = 1, \hdots, n_2$ let $A_i \in \mathrm{M}_{n_1 \times n_1}(\C)$ be symmetric matrices. Consider the varieties $V_1 \subset \mathbb{P}_\C^{n_1-1}$ %in the variables $[\bm{t}]$ 
	and $V_2 \subset \mathbb{P}_\C^{n_1-1} \times \mathbb{P}_\C^{n_2-1}$ %in the variables $([\bm{x}],[\bm{y}])$ 
	defined by  
	\begin{align*}
		V_1 &= \mathbb{V}(\bm{t}^T A_i \bm{t})_{i = 1, \hdots, n_2} \\
		V_2 &= \mathbb{V}\left(\sum_{i=1}^{n_2}y_i A_i \bm{x}\right).
	\end{align*}
	Then we have
	\begin{equation*} \label{eq:dim_varieties_not_too_big}
		\dim V_2 \leq \dim V_1 + n_2-1.
	\end{equation*}
	In particular, if $V_1 = \emptyset$ then $\dim V_2 \leq n_2-2$.
\end{lemma}

\begin{proof}
	Consider the variety $V_3 \subset \mathbb{P}_\C^{n_1-1} \times \mathbb{P}_\C^{n_1-1}$ defined by
	\[
	V_3 = \mathbb{V}(\bm{z}^T A_i \bm{x})_{i = 1, \hdots, n_2}. 
	\]
	Further for $\bm{x} = (x_1, \hdots, x_{n_1})^T$ consider  
	\[
	A (\bm{x}) = (A_1 \bm{x} \cdots A_{n_2} \bm{x}) \in \mathrm{M}_{n_1 \times n_2}(\C)[x_1, \hdots, x_{n_1}].
	\]
	We may write $V_2 = \mathbb{V}(A(\bm{x})\bm{y})$ and $V_3 = \mathbb{V}(\bm{z}^T A(\bm{x}))$. Our first goal is to relate the dimensions of the varieties above as follows
	\begin{equation} \label{eq.dim_v_2_v_3}
			\dim V_2 \leq \dim V_3 +n_2-n_1.
	\end{equation}
	% If n_1 = n_2 we should get equality here.
	For $r = 0, \hdots, n_1$ define the quasi-projective varieties $D_r \subset \mathbb{P}_\C^{n_1-1}$ given by
	\[
	D_r = \{ \bm{x} \in \mathbb{P}^{n_1-1}_\C \colon \mathrm{rank}(A (\bm{x})) = r \}.
	\]
	These are quasiprojective since they may be written as the intersection of the vanishing of all $(r+1) \times (r+1)$ minors of $A(\bm{x})$ with the complement of the vanishing of all $r \times r$ minors.
	For each $r$ let
	\[
	D_r = \bigcup_{i \in I_r} D_r^{(i)}
	\]
	be a decomposition into finitely many irreducible components. Since $\bigcup_r D_r = \mathbb{P}_\C^{n_1-1}$ we have
	\begin{equation*} \label{eq.dim_v_2_over_union}
	\dim V_2 = \max_{\substack{0 \leq r < n_2 \\ i \in I_r}} \dim ((D_r^{(i)} \times \mathbb{P}_\C^{n_2-1}) \cap V_2).
	\end{equation*}
	Note that $r = n_2$ doesn't play a role here, since the intersection $(D_{n_2}^{(i)} \times \mathbb{P}_\C^{n_2-1}) \cap V_2$ is empty. Similarly we get
	\begin{equation*} \label{eq.dim_v_1_over_union}
		\dim V_3 = \max_{\substack{0 \leq r < n_2 \\ i \in I_r}} \dim ((D_r^{(i)} \times \mathbb{P}_\C^{n_1-1}) \cap V_3).
	\end{equation*}
	%Note that since $n_1 \leq n_2$
	For $0 \leq r < n_2$ and $i \in I_r$ consider now the surjective projection maps 
	\[
	\pi_{2,r,i} \colon (D_r^{(i)} \times \mathbb{P}_\C^{n_2-1}) \cap V_2 \rightarrow D_r^{(i)}, \; (\bm{x}, \bm{y}) \mapsto \bm{x},
	\]
	and
	\[
	\pi_{3,r,i} \colon (D_r^{(i)} \times \mathbb{P}_\C^{n_1-1}) \cap V_3 \rightarrow D_r^{(i)}, \; (\bm{x}, \bm{z}) \mapsto \bm{x},
	\]
	We note that by the way $D_r^{(i)}$ was constructed here, the fibres of both of these projection morphisms have constant dimension for fixed $r$. By the rank-nullity theorem we find that the dimensions of the fibres are related as follows
	\begin{equation} \label{eq.dim_fibres_rank_null}
	\dim \pi_{2,r,i}^{-1}(\bm{x}) = \dim \pi_{3,r,i}^{-1}(\bm{x}) +n_2-n_1.
	\end{equation}
	We claim that the morphism $\pi_{2,r,i}$ is proper. For this note that the structure morphism $\mathbb{P}^{n_1-1}_\C \rightarrow \mathrm{Spec} \, \C$ is proper whence $D_r^{(i)} \times \mathbb{P}_\C^{n_1-1} \rightarrow D_r^{(i)}$ must be proper too, as properness is preserved under base change. As $(D_r^{(i)} \times \mathbb{P}_\C^{n_2-1}) \cap V_2$ is closed inside  $D_r^{(i)} \times \mathbb{P}_\C^{n_1-1}$ the restriction $\pi_{2,r,i}$ must also be proper. By an analogous argument it follows $\pi_{3,r,i}$ is also proper.
	
	Further note that the fibres of $\pi_{2,r,i}$ are irreducible since they define linear subspaces of $(D_r^{(i)} \times \mathbb{P}_\C^{n_2-1}) \cap V_2$, and similarly the fibres of $\pi_{3,r,i}$ are irreducible. Since $D_r^{(i)}$ is irreducible by construction and all the fibres have constant dimension, it follows %from \cite[Exercise 11.4.C]{vakil2017rising} 
	that $(D_r^{(i)} \times \mathbb{P}_\C^{n_2-1}) \cap V_2$ is irreducible. Similarly $(D_r^{(i)} \times \mathbb{P}_\C^{n_1-1}) \cap V_3$ is irreducible. 

%	%The dimensions of the fibres are constant and irreducible so that also $(D_r^{(i)} \times \mathbb{P}^{n_2-1}) \cap V_2$ and $(D_r^{(i)} \times \mathbb{P}^{n_1-1}) \cap V_3$ must be irreducible.
%	% Vakil, Exercise 11.4C assumes the morphism is proper. It suffices to have a closed morphism, because the result still holds, see: https://math.stackexchange.com/questions/242360/irreducible-fibers-of-a-closed-subset-implies-irreducibility and http://www-personal.umich.edu/~mmustata/Note1_09.pdf
%	% Now the morphism is closed because: the structure morphism  P^n -> Spec(k) is proper, in particularly universally closed and hence W_r^{(i)} x P^n -> W_r^{(i)} is closed. Now V_2 is closed so restricting the projection to the intersection with V_2 gives a closed map. It's surjective for other reasons.
	
	% This is Vakil, Exercise 11.4.C. The theorem we use next is Vakil, Theorem 11.4.1. The open subset is just everything, since the fibres have constant dimension
	%Hence all the conditions of \cite[Theorem 11.4.1]{vakil2017rising} are satisfied 
	Hence all the conditions of Chevalley's upper semicontinuity theorem are satisfied~\cite[Th\'eor\`eme 13.1.3]{EGA4},
	so that for any $\bm{x} \in D_r^{(i)}$ we obtain
	\begin{equation} \label{eq.dim_of_fibres}
	\dim \pi_{2,r,i}^{-1}(\bm{x}) = \dim ((D_r^{(i)} \times \mathbb{P}_\C^{n_2-1}) \cap V_2) - \dim D_r^{(i)},
	\end{equation}
	and
	\begin{equation} \label{eq.dim_of_fibres__V3}
	\dim \pi_{3,r,i}^{-1}(\bm{x}) = \dim ((D_r^{(i)} \times \mathbb{P}_\C^{n_1-1}) \cap V_3) - \dim D_r^{(i)}.
	\end{equation}
	Hence~\eqref{eq.dim_of_fibres} and~\eqref{eq.dim_of_fibres__V3} together with~\eqref{eq.dim_fibres_rank_null} yield
	\[
	\dim ((D_r^{(i)} \times \mathbb{P}_\C^{n_2-1}) \cap V_2) = \dim ((D_r^{(i)} \times \mathbb{P}_\C^{n_1-1}) \cap V_3)  + n_2-n_1.
	\]
	Choosing $r$ and $i$ such that $\dim V_2 = \dim ((D_r^{(i)} \times \mathbb{P}_\C^{n_2-1}) \cap V_2)$ the claim \eqref{eq.dim_v_2_v_3} now follows.

	Thus it is enough to find an upper bound for $\dim V_3$. 
	To this end, consider the affine cones $\widetilde{V_1} = \mathbb{V}(\bm{u}^T A_i \bm{u})_{i = 1, \hdots, n_2} \subset \mathbb{A}_{\mathbb{C}}^{n_1}$ and  $\widetilde{V_3} = \mathbb{V}(\bm{x}^T A(\bm{z})) \subset \mathbb{A}_{\mathbb{C}}^{n_1} \times \mathbb{A}_{\mathbb{C}}^{n_1}$. Note in particular, that $\widetilde{V}_1 \neq \emptyset$ even if $V_1 =  \emptyset$. 
	
	Write $\widetilde{\Delta} \subset \mathbb{A}_{\mathbb{C}}^{n_1} \times \mathbb{A}_{\mathbb{C}}^{n_1}$ for the diagonal given by $\mathbb{V}(x_i = z_i)_i$. Then $\widetilde{V_3} \cap \widetilde{\Delta} \cong  \widetilde{V_1} \neq \emptyset$. Thus, the affine dimension theorem~\cite[Proposition 7.1]{hartshorne2013algebraic} yields
	\[
	\dim \widetilde{V_1} \geq \dim \widetilde{V_3} -n_1.
	\] 
	Noting $\dim V_1 +1 \geq \dim \widetilde{V_1}$ and $\dim \widetilde{V_3} \geq \dim V_3 +2$ now gives the desired result. We remind the reader at this point that this is compatible with the convention $\dim \emptyset = -1$.
\end{proof}

\section{The auxiliary inequality} \label{sec.weyl_differencing_etc}
We remind the reader of the notation $e(x) = e^{2 \pi i x}$. For $\bm{\alpha} \in [0,1]^R$ define 
\begin{equation*}
	S(\bm{\alpha},P_1,P_2) = S(\bm{\alpha}) \coloneqq \sum_{\bm{x} \in P_1 \mathcal{B}_1} \sum_{\bm{y} \in P_2 \mathcal{B}_2} e \left( \bm{\alpha} \cdot \bm{F}\left(\bm{x},\bm{y} \right) \right),
\end{equation*}
where the sum ranges over $\bm{x} \in \mathbb{Z}^{n_1}$ such that $\bm{x}/P_1 \in \mathcal{B}_1$ and similarly for $\bm{y}$. Throughout this section we will assume $P_1 \geq P_2$. Note crucially that we have
\begin{equation*}
	N(P_1,P_2) = \int_{[0,1]^R} S(\bm{\alpha}) d \bm{\alpha}.
\end{equation*}
As noted in the introduction we can rewrite the forms as 
\begin{equation*}
	F_i(\bm{x},\bm{y}) = \sum_{\bm{j}} \sum_{\bm{k}} F^{(i)}_{\bm{j},\bm{k}} x_{j_1} \cdots x_{j_{d_1}} y_{k_1} \cdots y_{k_{d_2}},
\end{equation*}
and given $\bm{\alpha} \in \mathbb{R}^R$, as in \cite{schindler_bihomogeneous}, we consider the multilinear forms
\begin{equation*}
	\Gamma_{\bm{\alpha} \cdot \bm{F}}(\widetilde{\bm{x}}, \widetilde{\bm{y}}) \coloneqq d_1! d_2! \sum_{i} \alpha_i \sum_{\bm{j}} \sum_{\bm{k}} F^{(i)}_{\bm{j},\bm{k}} x^{(1)}_{j_1} \cdots x^{(d_1)}_{j_{d_1}} y^{(1)}_{k_1} \cdots y^{(d_2)}_{k_{d_2}}.
\end{equation*}
Further we write $\widehat{\bm{x}} = (\bm{x}^{(1)}, \hdots, \bm{x}^{(d_1-1)})$ and similarly for $\widehat{\bm{y}}$. For any real number $\lambda$ we write $\nnorm{\lambda} = \min_{k \in \mathbb{Z}} \norm{\lambda-k}$. We now define $M_1(\bm{\alpha} \cdot \bm{F}; P_1,P_2,P^{-1})$ to be the number of integral $\widehat{\bm{x}} \in (-P_1,P_1)^{(d_1-1)n_1}$ and $\widetilde{\bm{y}} \in (-P_2,P_2)^{d_2n_2}$ such that for all $\ell = 1, \hdots, n_1$ we have
\begin{equation*}
	\nnorm{\Gamma_{\bm{\alpha} \cdot \bm{F}}(\widehat{\bm{x}}, \bm{e}_\ell, \widetilde{\bm{y}})} < P^{-1}.
\end{equation*}
Similarly, we define $M_2(\bm{\alpha} \cdot \bm{F}; P_1,P_2,P^{-1})$ to be the number of integral $\widetilde{\bm{x}} \in (-P_1,P_1)^{d_1n_1}$ and $\widehat{\bm{y}} \in (-P_2,P_2)^{(d_2-1)n_2}$ such that for all $\ell = 1, \hdots, n_2$ we have
\begin{equation*}
	\nnorm{\Gamma_{\bm{\alpha} \cdot \bm{F}}(\widetilde{\bm{x}}, \widehat{\bm{y}},\bm{e}_\ell,)} < P^{-1}.
\end{equation*}
For our purposes we will need a slight generalization of Lemma 2.1 in \cite{schindler_bihomogeneous} that deals with a polynomial $G(\bm{x},\bm{y})$, which is not necessarily bihomogeneous. If $G(\bm{x},\bm{y})$ has bidegree $(d_1,d_2)$ write 
\begin{equation*}
	G(\bm{x},\bm{y}) = \sum_{\substack{0 \leq r \leq d_1 \\ 0 \leq l \leq d_2}} G^{(r,l)}(\bm{x},\bm{y}),
\end{equation*}
where $G^{(r,l)}(\bm{x},\bm{y})$ is homogeneous of bidegree $(r,l)$. Using notation as above we first show the following preliminary Lemma, which is a version of Weyl's inequality for our context.

From now on we will often use the notation $\tilde{d} = d_1+d_2-2$.
\begin{lemma} \label{lem.weyl_differencing_general_poly}
	Let $\varepsilon > 0$. Let $G(\bm{x},\bm{y}) \in \mathbb{R}[x_1, \hdots, x_{n_1},y_1, \hdots, y_{n_2}]$ be a polynomial of bidegree $(d_1,d_2)$ with $d_1,d_2 \geq 1$.  For the exponential sum
	\begin{equation*}
		S_G(P_1,P_2) = \sum_{\bm{x} \in P_1 \mathcal{B}_1} \sum_{\bm{x} \in P_
		2 \mathcal{B}_2} e\left( G(\bm{x},\bm{y})\right)
	\end{equation*}
	we have the following bound
	\begin{equation*}
		\norm{S_G(P_1,P_2)}^{2^{\tilde{d}}} \ll P_1^{n_1(2^{\tilde{d}}-d_1+1) + \varepsilon} P_2^{n_2(2^{\tilde{d}}-d_2)} M_1 \left( G^{(d_1,d_2)}, P_1, P_2, P_1^{-1} \right).
	\end{equation*}
\end{lemma} 
\begin{proof}
	The proof is quite involved but follows closely the proof of Lemma 2.1 in \cite{schindler_bihomogeneous}, which in turn is based on idas of Schmidt \cite[Section 11]{schmidt85} and Davenport \cite[Section 3]{davenport_32_variables}. 
	
	Our first goal is to apply a Weyl differencing process $d_2-1$-times to the $\bm{y}$ part of $G$ and then $d_1-1$-times to the $\bm{x}$ part of the resulting polynomial. Clearly this is trivial if $d_2=1$ or $d_1=1$, respectively. Therefore assume for now that $d_2 \geq 2$. We start by applying the Cauchy-Schwarz inequality and the triangle inequality to find
	\begin{equation} \label{eq.S_G_cauchy_schwarz}
		\norm{S_G(P_1,P_2)}^{2^{d_2-1}} \ll P_1^{n_1(2^{d_2-1}-1)} \sum_{\bm{x} \in P_1 \mathcal{B}_1} \norm{S_{\bm{x}}(P_1,P_2)}^{2^{d_2-1}},
	\end{equation}
	where we define
	\begin{equation*}
		S_{\bm{x}}(P_1,P_2) = \sum_{\bm{y} \in P_2 \mathcal{B}_2} e(G(\bm{x},\bm{y})).
	\end{equation*}
	Now write $\mathcal{U} = P_2 \mathcal{B}_2$, write $\mathcal{U}^D = \mathcal{U}-\mathcal{U}$ for the difference set and define 
	\begin{equation*}
		\mathcal{U}(\bm{y}^{(1)}, \hdots, \bm{y}^{(t)}) = \bigcap_{\varepsilon_1 = 0,1} \cdots \bigcap_{\varepsilon_t = 0,1}\left( \mathcal{U}- \varepsilon_1 \bm{y}^{(1)} - \hdots - \varepsilon_t \bm{y}^{(t)} \right).
	\end{equation*}
	Write $\mathcal{F}(\bm{y}) = G(\bm{x},\bm{y})$ and set 
	\begin{equation*}
		\mathcal{F}_d(\bm{y}^{(1)}, \hdots, \bm{y}^{(d)}) = \sum_{\varepsilon_1=0,1} \cdots \sum_{ \varepsilon_d = 0,1} (-1)^{\varepsilon_1 + \hdots + \varepsilon_d} \mathcal{F}(\varepsilon_1 \bm{y}^{(1)} + \hdots + \varepsilon_d \bm{y}^{(d)}).
	\end{equation*}
	Equation (11.2) in \cite{schmidt85} applied to our situation gives
	\begin{multline*}
		\norm{S_{\bm{x}}(P_1,P_2)}^{2^{d_2-1}} \ll \norm{\mathcal{U}^D}^{2^{d_2-1}-d_2} \sum_{\bm{y}^{(1)} \in \mathcal{U}^D} \cdots \\
		\sum_{\bm{y}^{(d_2-2)} \in \mathcal{U}^D} \norm{\sum_{\bm{y}^{(d_2-1)} \in \mathcal{U}(\bm{y}^{(1)}, \hdots, \bm{y}^{(d_2-2)})} e \left( \mathcal{F}_{d_2-1} \left(\bm{y}^{(1)}, \hdots, \bm{y}^{(d_2-1)}\right)   \right) }^2,
	\end{multline*}
	and we note that this did not require $\mathcal{F}(\bm{y})$ to be homogeneous in Schmidt's work. It is not hard to see that for $\bm{z}, \bm{z}' \in \mathcal{U}(\bm{y}^{(1)}, \hdots, \bm{y}^{(d_2-2)}) $ we have
	\begin{multline*}
		\mathcal{F}_{d_2-1} (\bm{y}^{(1)}, \cdots, \bm{z}) - \mathcal{F}_{d_2-1} (\bm{y}^{(1)}, \cdots, \bm{z}') =  \\
		\mathcal{F}_{d_2} (\bm{y}^{(1)}, \cdots, \bm{y}^{(d_2-1)}, \bm{y}^{(d_2)} ) - \mathcal{F}_{d_2-1} (\bm{y}^{(1)}, \cdots, \bm{y}^{(d_2-1)}),
	\end{multline*}
	for some $\bm{y}^{(d_2-1)} \in \mathcal{U}(\bm{y}^{(1)}, \hdots, \bm{y}^{(d_2-2)})^D$ and $\bm{y}^{(d_2)} \in \mathcal{U}(\bm{y}^{(1)}, \hdots, \bm{y}^{(d_2-1)})$. Thus we find 
	\begin{multline} \label{eq.S_x_bound_long}
\norm{S_{\bm{x}}(P_1,P_2)}^{2^{d_2-1}} \ll \norm{\mathcal{U}^D}^{2^{d_2-1}-d_2} \sum_{\bm{y}^{(1)} \in \mathcal{U}^D} \cdots \sum_{\bm{y}^{(d_2-2)} \in \mathcal{U}^D} \sum_{\bm{y}^{(d_2-1)} \in \mathcal{U}(\bm{y}^{(1)}, \hdots, \bm{y}^{(d_2-2)})^D} \\
\sum_{\bm{y}^{(d_2)} \in \mathcal{U}(\bm{y}^{(1)}, \hdots, \bm{y}^{(d_2-1)}) } e \left( \mathcal{F}_{d_2} \left(\bm{y}^{(1)}, \hdots, \bm{y}^{(d_2)} \right)   - \mathcal{F}_{d_2-1} \left(\bm{y}^{(1)}, \hdots, \bm{y}^{(d_2-1)} \right)\right).
	\end{multline}
	We may write the polynomial $G(\bm{x},\bm{y})$ as follows
	\begin{equation*}
		G(\bm{x}, \bm{y}) = \sum_{\substack{0 \leq r \leq d_1 \\ 0 \leq l \leq d_2}} \sum_{\bm{j}_r, \bm{k}_l} G_{\bm{j}_r,\bm{k}_l}^{(r,l)} \bm{x}_{\bm{j}_r} \bm{y}_{\bm{k}_l},
	\end{equation*}
	for some real $G_{\bm{j}_r,\bm{k}_l}^{(r,l)}$. Further write $\mathcal{F}(\bm{y}) = \mathcal{F}^{(0)}(\bm{y}) + \hdots + \mathcal{F}^{(d_2)}(\bm{y})$, where $\mathcal{F}^{(d)}(\bm{y})$ denotes the degree $d$ homogeneous part of $\mathcal{F}(\bm{y})$. Lemma 11.4 (A) in \cite{schmidt85} states that $\mathcal{F}_{d_2}$ transpires to be the multilinear form associated to $\mathcal{F}^{(d_2)}(\bm{y})$. From this we see
	\begin{equation} \label{eq.F_squiggle_difference}
		\mathcal{F}_{d_2} - \mathcal{F}_{d_2-1} = \sum_{\substack{0 \leq r \leq d_1 \\ 0 \leq l \leq d_2}} \sum_{\bm{j}_r, \bm{k}_l} G_{\bm{j}_r,\bm{k}_l}^{(r,l)} x_{j_r(1)} \cdots x_{j_r(r)} h_{\bm{k}_l} \left(\bm{y}^{(1)}, \hdots, \bm{y}^{(d_2)} \right),
	\end{equation}
	where 
	\begin{equation*}
		h_{\bm{k}_{d_2}} \left(\bm{y}^{(1)}, \hdots, \bm{y}^{(d_2)} \right) = d_2! y_{k_{d_2}(1)}^{(1)} \cdots y_{k_{d_2}(d_2)}^{(d_2)} + \tilde{h}_{\bm{k}_{d_2}} \left(\bm{y}^{(1)}, \hdots, \bm{y}^{(d_2-1)} \right),
	\end{equation*}
	for some polynomials $\tilde{h}_{\bm{k}_{d_2}}$ of degree $d_2$ that are independent of $\bm{y}^{(d_2)}$ and further $h_{\bm{k}_{l}}$ are polynomials of degree $l$ that are always independent of $\bm{y}^{(d_2)}$ whenever $l \leq d_2-1$. Write $\widetilde{\bm{y}} = (\bm{y}^{(1)}, \hdots, \bm{y}^{(d_2)})$. Now set
	\begin{equation*}
		S_{\widetilde{\bm{y}}} = \sum_{\bm{x} \in P_1 \mathcal{B}_1} e \left( \sum_{\substack{0 \leq r \leq d_1 \\ 0 \leq l \leq d_2}} \sum_{\bm{j}_r, \bm{k}_l} G^{(r,l)}_{\bm{j}_r, \bm{k}_l} x_{j_r(1)} \cdots x_{j_r(r)} h_{\bm{k}_l} (\widetilde{\bm{y}}) \right).
	\end{equation*}
	Now we swap the order of summation of $\sum_{\bm{x}}$ in \eqref{eq.S_G_cauchy_schwarz} with the sums over $\bm{y}^{(i)}$ in \eqref{eq.S_x_bound_long}. Using the Cauchy-Schwarz inequality and \eqref{eq.F_squiggle_difference} we thus obtain
	\begin{equation*}
		\norm{S_G(P_1,P_2)}^{2^{\tilde{d}}} \ll P_1^{n_1(2^{\tilde{d}}-2^{d_1-1})} P_2^{n_2(2^{\tilde{d}}-d_2)} \sum_{\bm{y}^{(1)}} \cdots \sum_{\bm{y}^{(d_2)}} \norm{S_{\widetilde{\bm{y}}}}^{2^{d_1-1}}.
	\end{equation*}	
	The above still holds if $d_2=1$, which can be seen directly. Applying the same differencing process to $S_{\widetilde{\bm{y}}}$ gives
	\begin{equation} \label{eq.S_G_weyl_differencing}
		\norm{S_G(P_1,P_2)}^{2^{\tilde{d}}} \ll P_1^{n_1(2^{\tilde{d}}-d_1)} P_2^{n_2(2^{\tilde{d}}-d_2)} \sum_{\bm{y}^{(1)}} \cdots \sum_{\bm{y}^{(d_2)}} \sum_{\bm{x}^{(1)}} \cdots \norm{\sum_{\bm{x}^{(d_1)}} e \left( \gamma(\widetilde{\bm{x}}, \widetilde{\bm{y}}) \right)},
	\end{equation}
	where
	\begin{equation*}
		\gamma(\widetilde{\bm{x}}, \widetilde{\bm{y}}) = \sum_{\substack{0 \leq r \leq d_1 \\ 0 \leq l \leq d_2}} \sum_{\bm{j}_r, \bm{k}_l} G^{(r,l)}_{\bm{j}_r,\bm{k}_l} g_{\bm{j}_r} (\widetilde{\bm{x}}) h_{\bm{k}_l}(\widetilde{\bm{y}}),
	\end{equation*}
	and where similar to before we have
	\begin{equation*}
		g_{\bm{j}_{d_1}} (\widetilde{\bm{x}}) = d_1! x_{j_{d_1}(1)}^{(1)} \cdots  x_{j_{d_1}(d_1)}^{(d_1)} + \tilde{g}_{\bm{j}_{d_1}} (\bm{x}^{(1)}, \hdots, \bm{x}^{(d_1-1)}),
	\end{equation*}
	with $\tilde{g}_{\bm{j}_{d_1}}$ and $g_{\bm{j}_{r}}$ for $r < d_1$ not depending on $\bm{x}^{(d_1)}$. We note that \eqref{eq.S_G_weyl_differencing} holds for all $d_1,d_2 \geq 1$ and all the summations $\sum_{\bm{x}^{(i)}}$ and $\sum_{\bm{y}^{(j)}}$ in \eqref{eq.S_G_weyl_differencing} are over boxes contained in $[-P_1,P_1]^{n_1}$ and $[-P_2,P_2]^{n_2}$, respectively. Write $\widehat{\bm{x}} = (\bm{x}^{(1)}, \hdots, \bm{x}^{(d_1-1)})$ and $\widehat{\bm{y}} = (\bm{y}^{(1)}, \hdots, \bm{y}^{(d_2-1)})$. We now wish to estimate the quantity
	\begin{equation} \label{eq.sum_of_hats}
		\sum(\widehat{\bm{x}}, \widehat{\bm{y}}) \coloneqq \sum_{\bm{y}^{(d_2)}} \norm{ \sum_{\bm{x}^{(d_1)}} e \left( \gamma(\widetilde{\bm{x}}, \widetilde{\bm{y}}) \right)}.
	\end{equation}
	Viewing $\sum_{a < x \leq b} e(\beta x)$ for $b-a \geq 1$ as a geometric series we recall the following elementary estimate
	\begin{equation*}
		\norm{\sum_{a < x \leq b} e(\beta x)} \ll \min \{ b-a, \nnorm{\beta}^{-1} \}.
	\end{equation*}
	This yields
	\begin{equation*}
		\norm{\sum_{\bm{x}^{(d_1)}} e \left( \gamma(\widetilde{\bm{x}}, \widetilde{\bm{y}}) \right)} \ll \prod_{\ell = 1}^{n_1} \min \left\{ P_1, \nnorm{\widetilde{\gamma}(\widehat{\bm{x}}, \bm{e}_\ell, \widetilde{\bm{y}})}^{-1} \right\},
	\end{equation*}
	where $\bm{e}_\ell$ denotes the $\ell$-th unit vector and where
	\begin{equation*}
		\widetilde{\gamma}(\widetilde{\bm{x}}, \widetilde{\bm{y}}) = d_1! \sum_{0 \leq l \leq d_2} \sum_{\bm{j}_{d_1}, \bm{k}_l} G^{(d_1,l)}_{\bm{j}_{d_1}, \bm{k}_l} x_{j_{d_1}(1)}^{(1)} \cdots  x_{j_{d_1}(d_1)}^{(d_1)} h_{\bm{k}_l}(\widetilde{\bm{y}}).
	\end{equation*}
	We now apply a standard argument in order to estimate this product, as in Davenport~\cite[Chapter 13]{davenport_book}.
    For a real number $z$ write $\{z \}$ for its fractional part. Let $\bm{r} = (r_1, \hdots, r_{n_1}) \in \Z^{n_1}$ be such that $0 \leq r_\ell < P_1$ holds for $\ell = 1, \hdots, n_1$. Define $\mathcal{A}(\widehat{\bm{x}}, \widehat{\bm{y}}, \bm{r})$ to be the set of $\bm{y}^{(d_2)}$ in the sum in \eqref{eq.sum_of_hats} such that
    \begin{equation*}
    	r_\ell P_1^{-1} \leq \left\{ \widetilde{\gamma} \left(\widehat{\bm{x}}, \bm{e}_\ell, \widehat{\bm{y}}, \bm{y}^{(d_2)}\right) \right\} < (r_\ell+1)P_1^{-1},
    \end{equation*}
    holds for all $\ell = 1, \hdots, n_1$ and write $A(\widehat{\bm{x}}, \widehat{\bm{y}}, \bm{r})$ for its cardinality. We obtain the estimate
    \begin{equation*}
    	\sum(\widehat{\bm{x}}, \widehat{\bm{y}}) \ll \sum_{\bm{r}} A(\widehat{\bm{x}}, \widehat{\bm{y}}, \bm{r}) \prod_{\ell = 1}^{n_1} \min \left\{ P_1, \max \left\{ \frac{P_1}{r_\ell}, \frac{P_1}{P_1-r_{\ell}-1} \right\} \right\}, 
    \end{equation*}
    where the sum $\sum_{\bm{r}}$ is over integral $\bm{r}$ with $0 \leq r_\ell < P_1$ for all $\ell = 1, \hdots, n_1$. Our next aim is to find a bound for $A(\widehat{\bm{x}}, \widehat{\bm{y}}, \bm{r})$ that is independent of $\bm{r}$. Given $\bm{u}, \bm{v} \in \mathcal{A}(\widehat{\bm{x}}, \widehat{\bm{y}}, \bm{r})$ then
    \begin{equation*}
    	\nnorm{\widetilde{\gamma} \left(\widehat{\bm{x}}, \bm{e}_\ell, \widehat{\bm{y}}, \bm{u}\right)- \widetilde{\gamma} \left(\widehat{\bm{x}}, \bm{e}_\ell, \widehat{\bm{y}}, \bm{v}\right)} < P_1^{-1},
    \end{equation*}
    for $\ell = 1, \hdots, n_1$. Similar as before we now define the multilinear forms
    \begin{equation*}
    	\Gamma_G(\widetilde{\bm{x}}, \widetilde{\bm{y}}) \coloneqq d_1! d_2! \sum_{\bm{j}_{d_1}, \bm{k}_{d_2}} G^{(d_1,d_2)}_{\bm{j}_{d_1}, \bm{k}_{d_2}} x_{j_{d_1}(1)}^{(1)} \cdots x_{j_{d_1}(d_1)}^{(d_1)} y_{k_{d_2}(1)}^{(1)} \cdots y_{k_{d_2}(d_2)}^{(d_2)},
    \end{equation*}
    which only depend on the $(d_1,d_2)$-degree part of $G$. For fixed $\widehat{\bm{x}}, \widehat{\bm{y}}$ let $N(\widehat{\bm{x}}, \widehat{\bm{y}})$ be the number of $\bm{y} \in (-P_2,P_2)^{n_2}$
 such that
 \begin{equation*}
 	\nnorm{\Gamma_G( \widehat{\bm{x}}, \bm{e}_\ell, \widehat{\bm{y}}, \bm{y})} < P_1^{-1},
 \end{equation*}
 for al $\ell = 1, \hdots, n_1$. Observe now crucially
 \begin{equation*}
 	\widetilde{\gamma} \left(\widehat{\bm{x}}, \bm{e}_\ell, \widehat{\bm{y}}, \bm{u}\right)- \widetilde{\gamma} \left(\widehat{\bm{x}}, \bm{e}_\ell, \widehat{\bm{y}}, \bm{v}\right) = \Gamma_G( \widehat{\bm{x}}, \bm{e}_\ell, \widehat{\bm{y}}, \bm{u}- \bm{v}).
 \end{equation*}
 Thus we find $A(\widehat{\bm{x}}, \widehat{\bm{y}}, \bm{r}) \leq N(\widehat{\bm{x}}, \widehat{\bm{y}})$ for all $\bm{r}$ as specified above. Using this we get
 \begin{equation*}
 	\sum_{\bm{y}^{(d_2)}} \norm{ \sum_{\bm{x}^{(d_1)}} e \left( \gamma(\widetilde{\bm{x}}, \widetilde{\bm{y}}) \right)} \ll N(\widehat{\bm{x}}, \widehat{\bm{y}}) (P_1 \log P_1)^{n_1}.
 \end{equation*}
Finally, summing over $\widehat{\bm{x}}$ and $\widehat{\bm{y}}$ we obtain
\begin{equation*}
	\norm{S_G(P_1,P_2)}^{2^{\tilde{d}}} \ll P_1^{n_1(2^{\tilde{d}}-d_1+1) + \varepsilon} P_2^{n_2(2^{\tilde{d}}-d_2)} M_1 \left( G^{(d_1,d_2)}, P_1, P_2, P_1^{-1} \right). \qedhere
\end{equation*}  \end{proof}

Inspecting the proof of Lemma 4.1 in \cite{schindler_bihomogeneous} we find that for a polynomial $G(\bm{x}, \bm{y})$ as above given $\theta \in (0,1]$ the following holds
\begin{multline*}
	M_1(G^{(d_1,d_2)},P_1,P_2,P_1^{-1}) \ll P_1^{n_1(d_1-1)} P_2^{n_2d_2} P_2^{-\theta(n_1d_1 + n_2d_2)}  \\ 
	\times \max_{i=1,2} \left\{ P_2^{n_i \theta} M_i\left(G^{(d_1,d_2)}; P_2^\theta, P_2^\theta, P_1^{-d_1} P_2^{-d_2} P_2^{\theta(\tilde{d}+1)}\right) \right\}
\end{multline*}
Using this and Lemma \ref{lem.weyl_differencing_general_poly} we deduce the next Lemma.

\begin{lemma} \label{lem.schindler_counting_linearised}
	Let $P_1,P_2 > 1$, $\theta \in (0,1]$ and $\bm{\alpha} \in \mathbb{R}^R$. Write $S_G = S_G(P_1,P_2)$. Using the same notation as above for $i=1$ or $i=2$ we have
	\begin{equation*}
		\norm{S_G}^{2^{\tilde{d}}} \ll_{d_i,n_i,\varepsilon} P_1^{n_12^{\tilde{d}} + \varepsilon} P_2^{n_22^{\tilde{d}}} P_2^{\theta n_i- \theta(n_1d_1+n_2d_2) }  \times M_i\left(G^{(d_1,d_2)}; P_2^\theta, P_2^\theta, P_1^{-d_1} P_2^{-d_2}P_2^{\theta(\tilde{d}+1)}\right).
	\end{equation*}
\end{lemma}
Using the preceding Lemma and adapting the proof of \cite[Lemma 3.1]{myerson_quadratic} to our setting we can now show the following.
\begin{lemma} \label{lem.auxiliary_ineq_lemma}
	Let $\varepsilon > 0$, $\theta \in (0,1]$ and $\bm{\alpha}, \bm{\beta} \in \mathbb{R}^R$. Then for $i = 1$ or $i=2$ we have
	\begin{equation} \label{eq.almost_auxiliary}
		\min \left\{ \norm{\frac{S(\bm{\alpha})}{P_1^{n_1+\varepsilon} P_2^{n_2}}}, \norm{\frac{S(\bm{\alpha} + \bm{\beta})}{P_1^{n_1+\varepsilon} P_2^{n_2}}} \right\}^{2^{\tilde{d}+1}} \\ \ll_{d_i,n_i,\varepsilon} \frac{M_i\left(\bm{\beta} \cdot \bm{F}; P_2^\theta, P_2^\theta, P_1^{-d_1} P_2^{-d_2}P_2^{\theta(\tilde{d}+1)}\right)}{P_2^{\theta(n_1 d_1 + n_2d_2)-\theta n_i}} 
	\end{equation}
\end{lemma}
\begin{proof}
	Note first that for two real numbers $\lambda, \mu >0$ we have 
	\begin{equation*}
		\min\{\lambda, \mu\} \leq \sqrt{\lambda\mu}.
	\end{equation*}
	Therefore it suffices to show
	\begin{equation*}
		\norm{\frac{S(\bm{\alpha})S(\bm{\alpha} + \bm{\beta})}{P_1^{2n_1+2\varepsilon} P_2^{2n_2}} }^{2^{\tilde{d}}} \ll_{d_i,n_i,\varepsilon} \frac{M_i\left(\bm{\beta}; P_2^\theta, P_2^\theta, P_1^{-d_1} P_2^{-d_2}P_2^{\theta(\tilde{d}+1)}\right)}{P_2^{\theta(n_1 d_1 + n_2d_2)-\theta n_i}}.
	\end{equation*}
	holds for $i=1$ or $i=2$.
	Note first that 
	\begin{equation*}
		\norm{S(\bm{\alpha} + \bm{\beta}) \overbar{S}(\bm{\alpha})} = \norm{\sum_{\substack{\bm{x} \in P_1 \mathcal{B}_1 \\ \bm{y} \in  P_2 \mathcal{B}_2}} \sum_{\substack{\bm{x}+\bm{z} \in P_1 \mathcal{B}_1 \\ \bm{y}+\bm{w} \in  P_2 \mathcal{B}_2}} e\left( (\bm{\alpha} + \bm{\beta}) \cdot \bm{F}(\bm{x},\bm{y}) - \bm{\alpha} \cdot \bm{F} (\bm{x}+\bm{z}, \bm{y}+\bm{w}) \right)},
	\end{equation*}
	so by the triangle inequality we get
	\begin{equation*}
		\norm{S(\bm{\alpha} + \bm{\beta}) \overbar{S}(\bm{\alpha})} \leq \sum_{\substack{\nnorm{\bm{z}}_\infty \leq P_1 \\ \nnorm{\bm{w}}_\infty \leq P_2}} \norm{\sum_{\substack{\bm{x} \in P_1\mathcal{B}_{\bm{z}} \\ \bm{y} \in P_2\mathcal{B}_{\bm{w}}}} e\left( \bm{\beta} \cdot \bm{F}(\bm{x},\bm{y}) - g_{\bm{\alpha},\bm{\beta},\bm{z},\bm{w}}(\bm{x},\bm{y}) \right)},
	\end{equation*}
	where $g_{\bm{\alpha},\bm{\beta},\bm{z},\bm{w}}(\bm{x},\bm{y})$ is of degree at most $d_1+d_2-1$ in $(\bm{x},\bm{y})$ and we have some boxes $\mathcal{B}_{\bm{z}} \subset \mathcal{B}_1$ and $\mathcal{B}_{\bm{w}} \subset \mathcal{B}_2$. Applying Cauchy's inequality $\tilde{d}$-times we deduce 
	\begin{equation*}
		\norm{S(\bm{\alpha} + \bm{\beta}) \overbar{S}(\bm{\alpha})}^{2^{\tilde{d}}} \leq P_1^{n_1(2^{\tilde{d}}-1)}P_2^{n_2(2^{\tilde{d}}-1)} \sum_{\substack{\nnorm{\bm{z}}_\infty \leq P_1 \\ \nnorm{\bm{w}}_\infty \leq P_2}} \norm{\sum_{\substack{\bm{x} \in P_1\mathcal{B}_{\bm{z}} \\ \bm{y} \in P_2\mathcal{B}_{\bm{w}}}} e\left( \bm{\beta} \cdot \bm{F}(\bm{x},\bm{y}) - g_{\bm{\alpha},\bm{\beta},\bm{z},\bm{w}}(\bm{x},\bm{y}) \right)}^{2^{\tilde{d}}}.
	\end{equation*}
%Clearly we have $\#\{ \bm{x} \in \mathbb{Z}^n \mid \nnorm{\bm{x}}_\infty \leq A \} \ll A^n$. 
If we write $G(\bm{x}, \bm{y}) = \bm{\beta} \cdot \bm{F}(\bm{x},\bm{y}) - g_{\bm{\alpha},\bm{\beta},\bm{z},\bm{w}}(\bm{x},\bm{y})$ then note that $G^{(d_1,d_2)} = \bm{\beta} \cdot \bm{F}$. Using Lemma \ref{lem.schindler_counting_linearised} we therefore obtain
\begin{multline*}
	\norm{S(\bm{\alpha} + \bm{\beta}) \overbar{S}(\bm{\alpha})}^{2^{\tilde{d}}} \ll P_1^{2^{\tilde{d}+1}n_1+\varepsilon}P_2^{2^{\tilde{d}+1}n_2} P_2^{-\theta (n_1d_1 + n_2d_2) + \theta n_i}\\ \times M_i(\bm{\beta} \cdot \bm{F}, P_2^\theta, P_2^\theta, P_1^{-d_1}P_2^{-d_2}P_2^{\theta(\tilde{d}+1)}),
\end{multline*}
for $i=1$ or $i=2$, which readily delivers the result.
\end{proof}
As in the introduction, for $\bm{\beta} \in \mathbb{R}^R$ we define the auxiliary counting function $N_1^{\mathrm{aux}}(\bm{\beta}; B)$ to be the number of integer vectors $\widehat{\bm{x}} \in (-B,B)^{(d_1-1)n_1}$ and $\widetilde{\bm{y}} \in (-B,B)^{d_2n_2}$ such that 
\begin{equation*}
	\norm{\Gamma_{\bm{\beta} \cdot \bm{F}}(\widehat{\bm{x}}, \bm{e}_\ell, \widetilde{\bm{y}})} < \nnorm{\bm{\beta} \cdot \bm{F}}_\infty B^{\tilde{d}},
\end{equation*}
for $\ell = 1, \hdots, n_1$ where $\nnorm{f}_\infty \coloneqq \frac{1}{d_1!d_2!} \max_{\bm{j}, \bm{k}} \norm{\frac{\partial^{d_1+d_2}f}{\partial x_{j_1} \cdots \partial x_{j_{d_1}} \partial y_{k_1} \cdots \partial y_{k_{d_2}}  }}$. We also analogously define $N_2^{\mathrm{aux}}(\bm{\beta}; B)$.
We now formulate an analogue for \cite[Proposition 3.1]{myerson_quadratic}.
\begin{proposition} \label{prop.auxiliary_ineq_from_counting_function}
	Let $C_0 \geq 1$ and $\mathscr{C} > 0$ such that for all $\bm{\beta} \in \mathbb{R}^R$ and $B>0$ we have for $i=1,2$ that
	\begin{equation} \label{eq.auxiliary_ineq_assumption}
		N_i^{\mathrm{aux}}(\bm{\beta}; B) \leq C_0 B^{d_1n_1+d_2n_2 - n_i - 2^{\tilde{d}+1} \mathscr{C}}.
	\end{equation}
	Assume further that the forms $F_i$ are linearly independent, so that there exist $M>\mu>0$ such that
	\begin{equation} \label{eq.assumption_lin_ind}
		\mu \nnorm{\bm{\beta}}_\infty \leq \nnorm{\bm{\beta} \cdot \bm{F}}_\infty \leq M \nnorm{\bm{\beta}}_\infty. 
		\end{equation}
		Then there exists a constant $C>0$ depending on $C_0,d_i,n_i,\mu$ and $ M$ such that the following \emph{auxiliary inequality} 
		\begin{equation*} \label{eq.aux_ineq_section_2}
		\min \left\{ \norm{\frac{S(\bm{\alpha})}{P_1^{n_1+\varepsilon} P_2^{n_2}}}, \norm{\frac{S(\bm{\alpha} + \bm{\beta})}{P_1^{n_1+\varepsilon} P_2^{n_2}}} \right\} \leq
		 C\max\left\{P_2^{-1}, P_1^{-d_1}P_2^{-d_2} \nnorm{\bm{\beta}}_\infty^{-1}, \nnorm{\bm{\beta}}_\infty^{\frac{1}{\tilde{d}+1}} \right\}^{\mathscr{C}}
		\end{equation*}
		holds for all real numbers $P_1, P_2 > 1$.
\end{proposition}
\begin{proof}
	The strategy of this proof will closely follow the proof of \cite[Proposition 3.1]{myerson_quadratic}. By Lemma \ref{lem.auxiliary_ineq_lemma} we know that \eqref{eq.almost_auxiliary} holds for $i=1$ or $i=2$. Assume that there is some $\theta \in (0,1]$  such that for the same $i$ we have
	\begin{equation} \label{eq.counting_inequality}
		N_i^{\mathrm{aux}} (\bm{\beta}; P_2^\theta) < M_i(\bm{\beta} \cdot \bm{F}, P_2^\theta, P_2^\theta, P_1^{-d_1}P_2^{-d_2}P_2^{\theta(\tilde{d}+1)}),
	\end{equation}
 Going forward with the case $i=1$, noting that the case $i=2$ can be proven completely analogously, this means that there exists a $(d_1-1)$-tuple $\widehat{\bm{x}}$ and a $d_2$-tuple $\widetilde{\bm{y}}$ which is counted by $M_1(\bm{\beta} \cdot \bm{F}, P_2^\theta, P_2^\theta, P_1^{-d_1}P_2^{-d_2}P_2^{\theta(\tilde{d}+1)})$ but not by $N_1^{\mathrm{aux}} (\bm{\beta}; P_2^\theta)$. Therefore this pair of tuples satisfies
	\begin{equation} \label{eq.tuples_in_box}
		\nnorm{\widehat{\bm{x}}^{(i)}}_\infty, \nnorm{\widetilde{\bm{y}}^{(j)}}_\infty \leq P_2^\theta, \; \text{for} \; i = 1,\hdots, d_1-1 \; \text{and} \; j = 1, \hdots, d_2,
	\end{equation}
	and
	\begin{equation} \label{eq.gamma_upper_bound}
		\nnorm{\Gamma_{\bm{\beta} \cdot \bm{F}}(\widehat{\bm{x}}, \bm{e}_\ell, \widetilde{\bm{y}})} < P_1^{-d_1}P_2^{-d_2} P_2^{\theta(\tilde{d}+1)}, \; \text{for} \; \ell = 1, \hdots, n_1,
	\end{equation}
	since it is counted by $M_1(\bm{\beta} \cdot \bm{F}, P_2^\theta, P_2^\theta, P_1^{-d_1}P_2^{-d_2}P_2^{\theta(\tilde{d}+1)})$. On the other hand, since it is not counted by $N_1^{\mathrm{aux}} (\bm{\beta}; P_2^\theta)$ there exists $\ell_0 \in \{1, \hdots, n_1 \}$ such that
	\begin{equation} \label{eq.gamma_lower_bound}
		\norm{\Gamma_{\bm{\beta} \cdot \bm{F}} (\widehat{\bm{x}}, \bm{e}_{\ell_0}, \widetilde{\bm{y}})} \geq \nnorm{\bm{\beta} \cdot \bm{F}}_\infty P_2^{\tilde{d} \theta}.
	\end{equation}
	From \eqref{eq.gamma_upper_bound} we get that for $\ell_0$ we must have either
	\begin{equation} \label{eq.gamma_really_small}
		\norm{\Gamma_{\bm{\beta} \cdot \bm{F}}(\widehat{\bm{x}}, \bm{e}_{\ell_0}, \widetilde{\bm{y}})} < P_1^{-d_1} P_2^{-d_2} P_2^{\theta(\tilde{d}+1)}
	\end{equation}
	or
	\begin{equation} \label{eq.gamma_bigger_but_int_close}
		\norm{\Gamma_{\bm{\beta} \cdot \bm{F}}(\widehat{\bm{x}}, \bm{e}_{\ell_0}, \widetilde{\bm{y}})} \geq \frac{1}{2}.
	\end{equation}
	If \eqref{eq.gamma_really_small} holds then \eqref{eq.gamma_lower_bound} implies
	\begin{equation} \label{eq.beta_F_bound_1}
		\nnorm{\bm{\beta} \cdot \bm{F}}_\infty < \frac{P_1^{-d_1}P_2^{-d_2}P_2^{(\tilde{d}+1)\theta}}{P_2^{\tilde{d}\theta}} = P_2^\theta P_1^{-d_1} P_2^{-d_2}
	\end{equation}
	If on the other hand \eqref{eq.gamma_bigger_but_int_close} holds then \eqref{eq.tuples_in_box} gives
	\begin{equation} \label{eq.beta_F_bound_2}
		\frac{1}{2} \leq \norm{\Gamma_{\bm{\beta} \cdot \bm{F}}(\widehat{\bm{x}}, \bm{e}_{\ell_0}, \widetilde{\bm{y}})} \ll \nnorm{\bm{\beta} \cdot \bm{F}}_\infty P_2^{(\tilde{d}+1)\theta}.
	\end{equation}
	Since either \eqref{eq.beta_F_bound_1} or \eqref{eq.beta_F_bound_2} holds then via \eqref{eq.assumption_lin_ind} we deduce 
	\begin{equation} \label{eq.P_theta_max}
		P_2^{-\theta} \ll_{\mu,M} \max\left\{ P_1^{-d_1}P_2^{-d_2} \nnorm{\bm{\beta}}_\infty^{-1}, \nnorm{\bm{\beta}}_\infty^{\frac{1}{\tilde{d}+1}}  \right\}.
	\end{equation} 
	Since \eqref{eq.almost_auxiliary} holds for $i=1$ and due to the assumption \eqref{eq.auxiliary_ineq_assumption} we see that \eqref{eq.counting_inequality} holds if there exists some $C_1 > 0$ such that
	\begin{equation} \label{eq.aux_ineq_predecessor}
		P_2^{-\theta 2^{\tilde{d}+1}\mathscr{C}} \leq C_1 \min \left\{ \norm{\frac{S(\bm{\alpha})}{P_1^{n_1+\varepsilon} P_2^{n_2}}}, \norm{\frac{S(\bm{\alpha} + \bm{\beta})}{P_1^{n_1+\varepsilon} P_2^{n_2}}} \right\}^{2^{\tilde{d}+1}}.
	\end{equation}
	Now \emph{define} $\theta$ such that we have equality in the equation above, i.e. such that we have
	\begin{equation} \label{eq.theta_definition}
		P_2^\theta = C_1^{\frac{1}{2^{\tilde{d}+1} \mathscr{C}}} \min \left\{ \norm{\frac{S(\bm{\alpha})}{P_1^{n_1+\varepsilon} P_2^{n_2}}}, \norm{\frac{S(\bm{\alpha} + \bm{\beta})}{P_1^{n_1+\varepsilon} P_2^{n_2}}} \right\}^{-\frac{1}{\mathscr{C}}}.
	\end{equation}
	If $\theta \in (0,1]$ then~\eqref{eq.aux_ineq_predecessor} holds and so together with the assumption~\eqref{eq.auxiliary_ineq_assumption} as argued above this implies~\eqref{eq.P_theta_max} holds, which gives the result in this case. But $\theta$ will always be positive; for if $\theta \leq 0$ then~\eqref{eq.theta_definition} implies 
	\begin{equation*}
		\min \left\{ \norm{\frac{S(\bm{\alpha})}{P_1^{n_1+\varepsilon} P_2^{n_2}}}, \norm{\frac{S(\bm{\alpha} + \bm{\beta})}{P_1^{n_1+\varepsilon} P_2^{n_2}}} \right\} \geq C_1^{-\frac{1}{2^{\tilde{d}+1}}}.
	\end{equation*}
	However, note that clearly $\norm{S(\bm{\alpha})} \leq (P_1+1)^{n_1} (P_2+1)^{n_2}$. Without loss of generality we may take $P_i$ large enough, depending on $\varepsilon$, so that this clearly leads to a contradiction. Finally, if $\theta \geq 1$ then we find $P_2^{- \mathscr{C} \theta} \leq P_2^{-\mathscr{C}}$, and so from~\eqref{eq.theta_definition} we obtain.
	\[
	\min \left\{ \norm{\frac{S(\bm{\alpha})}{P_1^{n_1+\varepsilon} P_2^{n_2}}}, \norm{\frac{S(\bm{\alpha} + \bm{\beta})}{P_1^{n_1+\varepsilon} P_2^{n_2}}} \right\} \ll P_2^{-\mathscr{C}}.
	\]
	This gives the result.
\end{proof}
%We note that the results in this section are not symmetric in the indices $1$ and $2$. However, it is clear that the corresponding results 
%We note that by interchanging the indices $1$ and $2$ throughout this section, one would obtain the corresponding results under the assumption $P_2 > P_1$.
\section{The circle method} \label{sec.circle_method}
The aim of this section is to use the auxiliary inequality 
\begin{equation} \label{eq.aux_ineq}
		P_1^{-\varepsilon} \min \left\{ \norm{\frac{S(\bm{\alpha})}{P_1^{n_1} P_2^{n_2}}}, \norm{\frac{S(\bm{\alpha} + \bm{\beta})}{P_1^{n_1} P_2^{n_2}}} \right\} \leq \\
		 C\max\left\{P_2^{-1}, P_1^{-d_1}P_2^{-d_2} \nnorm{\bm{\beta}}_\infty^{-1}, \nnorm{\bm{\beta}}_\infty^{\frac{1}{\tilde{d}+1}} \right\}^{\mathscr{C}},
\end{equation}
where $C \geq 1$ 
%		\begin{equation} \label{eq.aux_ineq}
%		\min \left\{ \norm{\frac{S(\bm{\alpha})}{P_1^{n_1+\varepsilon} P_2^{n_2}}}, \norm{\frac{S(\bm{\alpha} + \bm{\beta})}{P_1^{n_1+\varepsilon} P_2^{n_2}}} \right\} \ll_{C_0,d_i,n_i,\mu,M} \max\left\{ P_1^{-d_1}P_2^{-d_2} \nnorm{\beta}_\infty^{-1}, \nnorm{\beta}_\infty^{\frac{1}{\tilde{d}+1}} \right\}^{\mathscr{C}},
%		\end{equation}
 and apply the circle method in order to deduce an estimate for $N(P_1,P_2)$. In this section we will use the notation $P = P_1^{d_1}P_2^{d_2}$. Write $b = \max \left\{1, \log P_1 / \log P_2\right\}$ and $u = \max \left\{1, \log P_2/\log P_1 \right\}$. If $P_1 \geq P_2$ then $b = \log P_1 / \log P_2$ and thus $P_2^{bd_1+d_2} = P$ holds. The main result will be the following. 
% Note that since we show the following result for general bidegree $(d_1,d_2)$ the assumption $P_1 \geq P_2$ can be made without any loss of generality.
\begin{proposition} \label{prop.main_prop}
	Let $\mathscr{C} > (bd_1+ud_2)R$, $C\geq 1$ and $\varepsilon >0$ such that the auxiliary inequality \eqref{eq.aux_ineq} holds for all $\bm{\alpha}, \bm{\beta} \in \mathbb{R}^R$, all $P_1, P_2 > 1$ and all boxes $\mathcal{B}_i \subset [-1,1]^{n_i}$ with side lengths at most $1$ and edges parallel to the coordinate axes. There exists some $\delta > 0$ depending on $b$, $u$, $R$, $d_i$ and $n_i$ such that
	\begin{equation*} %\label{eq.main_asymptotic}
		N(P_1,P_2) = \sigma P_1^{n_1-d_1R}P_2^{n_2-d_2R} + O \left(P_1^{n_1-d_1R}P_2^{n_2-d_2R} P^{-\delta} \right).
	\end{equation*}
	The factor $\sigma = \mathfrak{I} \mathfrak{S}$ is the product of the singular integral $\mathfrak{I}$ and the singular series $\mathfrak{S}$, as defined in~\eqref{eq.def_singular_integral} and~\eqref{eq.def_singular_series}, respectively. 
\end{proposition}
Note that this result holds for general bidegree, and therefore in the proof one may assume $P_1 \geq P_2$ throughout. For instance if one wishes to show the above proposition for bidegree $(2,1)$, the result follows from the asymmetric results of bidegree $(2,1)$ and bidegree $(1,2)$.

\subsection{The minor arcs}
First we will show that the contributions from the minor arcs do not affect the main term. For this we will prove a Lemma similar to Lemma 2.1 in~\cite{myerson_quadratic}.
\begin{lemma} \label{lem.general_integral_bound}
	Let $r_1, r_2 \colon (0, \infty) \rightarrow (0, \infty)$ be strictly decreasing and increasing bijections, respectively, and let $A >0$ be a real number. For any $\nu >0$ let $E_0 \subset \mathbb{R}^R$ be a hypercube of side lengths $\nu$ whose edges are parallel to the coordinate axes. Let $E \subseteq E_0$ be a measurable set and let $\varphi \colon E \rightarrow [0, \infty)$ be a measurable function.
	
	Assume that for all $\bm{\alpha}, \bm{\beta} \in \mathbb{R}^R$ such that $\bm{\alpha}, \bm{\alpha} + \bm{\beta} \in E$ we have
	\begin{equation} \label{eq.general_aux_ineq}
		\min \left\{ \varphi(\bm{\alpha}), \varphi(\bm{\alpha}+ \bm{\beta}) \right\} \leq \max\left\{A, r_1^{-1}\left( \nnorm{\bm{\beta}}_\infty \right),r_2^{-1}\left( \nnorm{\bm{\beta}}_\infty \right) \right\}.
	\end{equation}
	Then for all integers $k \leq \ell$ such that $A < 2^k$ we get
	\begin{equation} \label{eq.general_integral_estimate}
		\int_E \varphi(\bm{\alpha}) d\bm{\alpha} \ll_R \\
		\nu^R 2^k + \sum_{i = k}^{\ell-1} 2^i \left( \frac{\nu r_1 (2^i)}{\min\{r_2(2^i),\nu \}} \right)^R + \left( \frac{\nu r_1(2^\ell)}{\min\{ r_2(2^\ell), \nu \}} \right)^R \sup_{\bm{\alpha} \in E} \varphi(\bm{\alpha}).
	\end{equation}
\end{lemma} 
Note that if we take 
\begin{equation*}
	\varphi(\bm{\alpha}) = C^{-1} P_1^{-n_1-\varepsilon} P_2^{-n_2} \norm{S(\bm{\alpha})}, \quad r_1(t) = P_1^{-d_1}P_2^{-d_2}t^{-\frac{1}{\mathscr{C}}}, \quad r_2(t) = t^{\frac{\tilde{d}+1}{\mathscr{C}}}, \quad A = P_2^{-\mathscr{C}}
\end{equation*}
where $C$ is the constant in \eqref{eq.aux_ineq}, then the assumption \eqref{eq.general_aux_ineq}~is just the auxiliary inequality \eqref{eq.aux_ineq}. 
\begin{proof}
	Given $t \geq 0$ define the set
	\[
	D(t) = \left\{\bm{\alpha} \in E \colon \varphi(\bm{\alpha}) \geq t \right\}.
	\]
	If $\bm{\alpha}$ and $\bm{\alpha}+ \bm{\beta}$ are both contained in $D(t)$ then by~\eqref{eq.general_aux_ineq} one of the following must hold
	\[
	A \geq t, \quad \nnorm{\bm{\beta}}_\infty \leq r_1(t), \quad \text{or} \quad \nnorm{\bm{\beta}}_\infty \geq r_2(t).
	\]
	In particular, if $t > A$ then either $\nnorm{\bm{\beta}}_\infty \leq r_1(t)$ or $\nnorm{\bm{\beta}}_\infty \geq r_2(t)$. Assuming that $t > A$ is satisfied consider a box $\mathfrak{b} \subset \R^R$ with sidelengths $r_2(t)/2$  whose edges are parallel to the coordinate axes. Given $\bm{\alpha} \in \mathfrak{b} \cap D(t)$ set
	\[
	\mathfrak{B}(\bm{\alpha}) = \left\{ \bm{\alpha} + \bm{\beta} \colon \bm{\beta} \in \R^R, \nnorm{\bm{\beta}}_\infty \leq r_1(t) \right\}.
	\]
If $\bm{\alpha} + \bm{\beta} \in \mathfrak{b} \cap D(t)$ then by construction $\nnorm{\bm{\beta}}_\infty \leq r_2(t) /2 < r_2(t)$ whence $\nnorm{\bm{\beta}}_\infty \leq r_1(t)$. Therefore we have $\mathfrak{b} \cap D(t) \subset \mathfrak{B}(\bm{\alpha})$, which in turn implies that the measure of $\mathfrak{b} \cap D(t)$ is bounded by $(2r_1(t))^R$. Since $D(t)$ is contained in $E_0$ one can cover $D(t)$ with at most
\[
\ll_R \frac{\nu^R}{\min\{r_2(t),\nu\}^R}
\]
boxes $\mathfrak{b}$ whose sidelenghts are $r_2(t)/2$. Therefore we find
\[
\mu(D(t)) \ll_R \left( \frac{\nu r_1(t)}{\min \{r_2(t), \nu \}}\right)^R,
\]
where we write $\mu(D(t))$ for the Lebesgue measure of $D(t)$. If $k < \ell$ are two integers then 
\begin{equation*}
	\int_E \varphi(\bm{\alpha}) d \bm{\alpha} = \int_{E \setminus D(2^k)} \varphi(\bm{\alpha}) d \bm{\alpha} + \sum_{i = k}^\ell \int_{D(2^i) \setminus D(2^{i+1})} \varphi(\bm{\alpha}) d \bm{\alpha} + \int_{D(2^\ell)} \varphi(\bm{\alpha}) d \bm{\alpha}.
\end{equation*}
We can trivially bound $\int_{E \setminus D(2^k)}\varphi(\bm{\alpha}) d\bm{\alpha} \leq \nu^R 2^k$, and further we can bound 
\[
\int_{D(2^i) \setminus D(2^{i+1})} \varphi(\bm{\alpha}) d\bm{\alpha} \leq 2^{i+1} \mu(D(2^i)), \quad \text{and} \quad \int_{D(2^\ell)} \varphi(\bm{\alpha}) d \bm{\alpha} \leq \mu(D(2^\ell)) \sup_{\bm{\alpha} \in E} \varphi(\bm{\alpha}).
\]
If $2^k > A$ then for any $i \geq k$ by our discussion above we find \[\mu(D(2^i)) \ll_R \left( \frac{\nu r_1(2^i)}{\min \{r_2(2^i), \nu \}}\right)^R.\]
Therefore
the result follows.
\end{proof}

Recall the notation $P=P_1^{d_1}P_2^{d_2}$. From now on we will  assume $P_1 \geq P_2$. 
Note that the assumption in Proposition~\ref{prop.auxiliary_ineq_from_counting_function} that $\mathscr{C} > R(bd_1 + ud_2)$ holds, is equivalent to $\mathscr{C} > R(bd_1 + d_2)$ if $P_1 \geq P_2$.% and equivalent to $\mathscr{C} > R(d_1 + ud_2)$ if $P_2 > P_1$.
\begin{lemma} \label{lem.good_int_bound}
	Let $T \colon \mathbb{R}^R \rightarrow \mathbb{C}$ be a measurable function. With notation as in Lemma \ref{lem.general_integral_bound} assume that for all $\bm{\alpha}, \bm{\beta} \in \mathbb{R}^R$ and for all $P_1 \geq P_2 > 1$, and $\mathscr{C} >0$ we have
	\begin{equation} \label{eq.less_general_aux_ineq}
		\min \left\{ \norm{\frac{T(\bm{\alpha})}{P_1^{n_1}P_2^{n_2}}},\norm{\frac{T(\bm{\alpha}+\bm{\beta})}{P_1^{n_1}P_2^{n_2}}}  \right\} \leq \max \left\{P_2^{-1}, P_1^{-d_1}P_2^{-d_2} \nnorm{\bm{\beta}}_\infty^{-1}, \nnorm{\bm{\beta}}_\infty^{\frac{1}{\tilde{d}+1}} \right\}^\mathscr{C}.
	\end{equation}
	Write $P = P_1^{d_1} P_2^{d_2}$ and assume that that we have
	\begin{equation} \label{eq.general_sup_bound}
		\sup_{\bm{\alpha} \in E} \norm{T(\bm{\alpha})} \leq P_1^{n_1}P_2^{n_2} P^{-\delta},
	\end{equation}
	for some $\delta > 0$. Then we have
	\begin{multline} \label{eq.many_cases_estimate}
		\int_E \frac{T(\bm{\alpha})}{P_1^{n_1} P_2^{n_2}}d\bm{\alpha} \ll_{\mathscr{C},d_i,R} \\
		\begin{cases}
			\nu^R P^{-R} P_2^{(\tilde{d}+2)R-\mathscr{C}} + P_2^{-\mathscr{C}}  \quad &\text{if} \; \mathscr{C} < R \\
			\nu^R P^{-R} P_2^{(\tilde{d}+2)R-\mathscr{C}} + P^{-R} \log P_2 + P_2^{-\mathscr{C}} \quad  &\text{if }  \mathscr{C} = R \\
			\nu^R P^{-R} P_2^{(\tilde{d}+2)R-\mathscr{C}} + P^{-R-\delta(1-R/\mathscr{C})} + P_2^{-\mathscr{C}}\quad  &\text{if }  R < \mathscr{C} < (d_1+d_2)R \\
			\nu^R P^{-R}  \log P_2 + P^{-R-\delta(1-R/\mathscr{C})}+P_2^{-\mathscr{C}} \quad  &\text{if } \mathscr{C} = (d_1+d_2)R \\
			\nu^RP^{-R-\delta(1-(d_1+d_2)R/\mathscr{C})} + P^{-R-\delta(1-R/\mathscr{C})} +P_2^{-\mathscr{C}} \quad  &\text{if } \mathscr{C} > (d_1+d_2)R. \\
		\end{cases}
	\end{multline}
\end{lemma}
We expect the main term of $N(P_1,P_2)$ to be of order $P_1^{n_1-Rd_1}P_2^{n_2-Rd_2} = P_1^{n_1}P_2^{n_2} P^{-R}$. Thus the Lemma indicates why it is necessary for us to assume $\mathscr{C} > R(bd_1+d_2)$, using this method of proof at least.% because otherwise the contributions of the minor arcs would not be negligible.
\begin{proof}
	We apply Lemma \ref{lem.auxiliary_ineq_lemma} by taking
	\begin{equation} \label{eq.choices}
		\varphi(\bm{\alpha}) = \frac{\norm{T(\bm{\alpha})}}{P_1^{n_1}P_2^{n_2}}, \quad r_1(t) = P_1^{-d_1}P_2^{-d_2}t^{-\frac{1}{\mathscr{C}}}, \quad r_2(t) = t^{\frac{\tilde{d}+1}{\mathscr{C}}},
		 \text{ and } A = P_2^{-\mathscr{C}}.
	\end{equation}
	Then our assumption \eqref{eq.less_general_aux_ineq} is just \eqref{eq.general_aux_ineq}. We will choose our parameters $k$ and $\ell$ such that the $\sum_{i=k}^{\ell-1}$ term dominates the right hand side of \eqref{eq.general_integral_estimate}. Let
	\begin{equation} \label{eq.choosing_k_l}
		k = \left\lceil \log_2P_2^{-\mathscr{C}} \right\rceil, \quad \text{and} \quad \ell = \left\lceil \log_2P^{-\delta} \right\rceil,
	\end{equation}
	so that we have 
	\begin{equation*} \label{eq.k_ell_inequality}
		 P_2^{-\mathscr{C}} < 2^k \leq 2 P_2^{-\mathscr{C}},  \quad \text{and} \quad  P^{-\delta} \leq 2^\ell < 2P^{-\delta}.
	\end{equation*}
	Without loss of generality we assume $ k< \ell$ since otherwise the bound in the assumption \eqref{eq.general_sup_bound} would be sharper than any of those listed in \eqref{eq.many_cases_estimate}. Substituting our choices \eqref{eq.choices} into \eqref{eq.general_integral_estimate} we get
	\begin{multline}
 \label{eq.integral_bound_with_choices}
		\int_E \frac{\norm{T(\bm{\alpha})}}{P_1^{n_1}P_2^{n_2}} \ll_R \nu^R2^k + \sum_{i=k}^{\ell-1} 2^i \left( \frac{\nu P_1^{-d_1}P_2^{-d_2}2^{-i/\mathscr{C}}}{\min\left\{ \nu, 2^{i(\tilde{d}+1)/\mathscr{C}} \right\}} \right)^R + \\
		\left( \frac{\nu P_1^{-d_1}P_2^{-d_2} 2^{-\ell/\mathscr{C}}}{\min\left\{ \nu, 2^{\ell(\tilde{d}+1)/\mathscr{C}} \right\}} \right)^R \sup_{\bm{\alpha}\in E} \frac{\norm{T(\bm{\alpha})}}{P_1^{n_1}P_2^{n_2}}.
	\end{multline}
	From \eqref{eq.general_sup_bound} and \eqref{eq.choosing_k_l} we see that
	\begin{equation} \label{eq.estimate_1}
		\sup_{\bm{\alpha}\in E} \frac{\norm{T(\bm{\alpha})}}{P_1^{n_1}P_2^{n_2}} \leq P^{-\delta} \leq 2^\ell.
	\end{equation}
	Further, we clearly have
	\begin{equation} \label{eq.estimate_2}
		\frac{ P_1^{-d_1}P_2^{-d_2}2^{-i/\mathscr{C}}}{\min\left\{ \nu, 2^{i(\tilde{d}+1)/\mathscr{C}} \right\}} \leq \nu^{-1}P_1^{-d_1}P_2^{-d_2}2^{-i/\mathscr{C}} + 2^{-i(\tilde{d}+2)/\mathscr{C}}P_1^{-d_1}P_2^{-d_2}.
	\end{equation}
	Substituting the estimates \eqref{eq.estimate_1} and \eqref{eq.estimate_2} into \eqref{eq.integral_bound_with_choices} we obtain
	\begin{equation} \label{eq.another_estimate}
		\int_E \frac{\norm{T(\bm{\alpha})}}{P_1^{n_1}P_2^{n_2}} \ll_R \nu^R2^k + \sum_{i=k}^\ell \nu^R P_1^{-d_1R}P_2^{-d_2R} 2^{i(1-(\tilde{d}+2)R/\mathscr{C})} + \sum_{i=k}^\ell P_1^{-d_1R}P_2^{-d_2R} 2^{i(1-R/\mathscr{C})}.
	\end{equation}
	Note now that 
	\begin{equation} \label{eq.three_cases_estimate_1}
		\sum_{i=k}^\ell 2^{i(1-R(\tilde{d}+2)/\mathscr{C})} \ll_{\mathscr{C},d_i,R} \begin{cases}
			2^{k(1-R(\tilde{d}+2)/\mathscr{C})} \quad &\text{if $\mathscr{C} < (\tilde{d}+2)R$} \\
			\ell-k &\text{if $\mathscr{C} = (\tilde{d}+2)R$} \\
			2^{\ell(1-R(\tilde{d}+2)/\mathscr{C})} &\text{if $\mathscr{C} > (\tilde{d}+2)R$},
		\end{cases}
	\end{equation}
	where we used $k < \ell$ for the second alternative.
	Recall from \eqref{eq.choosing_k_l} that we have
	\begin{equation*}
		2^k \geq P_2^{-\mathscr{C}} \quad \text{and} \quad 2^\ell \leq 2P^{-\delta},
	\end{equation*}
	so using this in \eqref{eq.three_cases_estimate_1} we get
	\begin{equation} \label{eq.three_cases_estimate_2}
		\sum_{i=k}^\ell 2^{i(1-(\tilde{d}+2)/\mathscr{C})} \ll_{\mathscr{C},d_i,R} \begin{cases}
			P_2^{(\tilde{d}+2)R-\mathscr{C}} \quad &\text{if $\mathscr{C} < (\tilde{d}+2)R$} \\
			\log P_2 &\text{if $\mathscr{C} = (\tilde{d}+2)R$} \\
			P^{-\delta(1-(\tilde{d}+2)R/\mathscr{C})} &\text{if $\mathscr{C} > (\tilde{d}+2)R$}. \end{cases}
	\end{equation}
	Arguing similarly for $\sum_{i=k}^\ell 2^{i(1-R/\mathscr{C})}$ we find
	\begin{equation} \label{eq.three_cases_estimate_3}
		\sum_{i=k}^\ell 2^{i(1-R/\mathscr{C})} \ll_{\mathscr{C},d_i,R} \begin{cases}
			P_2^{R-\mathscr{C}} \quad &\text{if $\mathscr{C} < R$} \\
			\log P_2 &\text{if $\mathscr{C} = R$} \\
			P^{-\delta(1-R/\mathscr{C})} &\text{if $\mathscr{C} > R$}. \end{cases}.
	\end{equation}
	Finally we note that by our choice of $k$ we have $2^k \leq 2P_2^{-\mathscr{C}}$ and we recall that $\tilde{d}+2 = d_1 + d_2$. Using this, as well as \eqref{eq.three_cases_estimate_2} and \eqref{eq.three_cases_estimate_3} in \eqref{eq.another_estimate} we deduce the result.
\end{proof} 
%Before we can finish this section, we need to actually define the major and minor arcs first. That is, we will split the set $[0,1]^R$ into two disjoint subsets according to whether an $R$-tuple is well approximated by a rational $R$-tuple or not. 
We will finish this section by defining the major and minor arcs and showing that the minor arcs do not contribute to the main term.
For $\Delta > 0$ we define the \emph{major arcs } to be the set given by
\begin{equation*}
	\mathfrak{M}(\Delta) \coloneqq  \bigcup_{\substack{q \in \mathbb{N} \\ q \leq P^\Delta}} \bigcup_{\substack{0 \leq a_i \leq q \\ (a_1,\hdots, a_R,q) = 1}} \left\{ \bm{\alpha} \in [0,1]^R  \colon 2 \nnorm{q\bm{\alpha}- \bm{a}}_\infty < P_1^{-d_1}P_2^{-d_2} P^\Delta \right\},
\end{equation*}
and the \emph{minor arcs } to be the given by
\begin{equation*}
	\mathfrak{m}(\Delta) \coloneqq [0,1]^R \setminus \mathfrak{M}(\Delta).
\end{equation*}
Write further
\begin{equation} \label{eq.delta_0_definition}
	\delta_0 = \frac{\min_{i=1,2}\left\{n_1+n_2-\dim V_i^* \right\}}{(\tilde{d}+1)2^{\tilde{d}}R}.
\end{equation}
Note that if the forms $F_i$ are linearly independent, then $V_i^*$ are proper subvarieties of $\mathbb{A}_{\mathbb{C}}^{n_1+n_2}$ so that $\dim V_i^* \leq n_1+n_2-1$ whence $\delta_0 \geq \frac{1}{(\tilde{d}+1)2^{\tilde{d}}R}$. To see this for $V_1^*$ note that requiring
\begin{equation*}
	\mathrm{rank}\left( \frac{\partial F_i}{\partial x_j} \right)_{i,j} < R
\end{equation*}
is equivalent to requiring all the $R \times R$ minors of $\left( \frac{\partial F_i}{\partial x_j} \right)_{i,j}$ vanish. This defines a system of polynomials of degree $R(d_1+d_2-1)$ in $(\bm{x},\bm{y})$, which are not all zero unless there exists $\bm{\beta} \in \mathbb{R}^R \setminus \{ \bm{0} \}$ such that
\begin{equation*}
	\sum_{i = 1}^R \beta_i \left( \frac{\partial F_i}{\partial x_j} \right) = 0 \quad \text{for } j = 1, \hdots, n_1
\end{equation*}
holds identically in $(\bm{x},\bm{y})$.
This is the same as saying that 
\begin{equation*}
	\nabla_{\bm{x}} \left( \sum_{i= 1}^R \beta_i F_i \right) = 0
\end{equation*}
holds identically. From this we find that $\sum_{i= 1}^R \beta_i F_i$ must be a form entirely in the $\bm{y}$ variables. But this is a linear combination of homogeneous bidegree $(d_1,d_2)$ forms with $d_1 \geq 1$ and thus we must in fact have $\sum_{i= 1}^R \beta_i F_i= 0$ identically, contradicting linear independence. The argument works analogously for $V_2^*$.

The next Lemma shows that the assumption \eqref{eq.general_sup_bound} holds with $E = \mathfrak{m}(\Delta)$ and $T(\bm{\alpha}) = C^{-1}P_1^{-\varepsilon}S(\bm{\alpha})$.
\begin{lemma} \label{lem.sup_bound}
	Let $0 < \Delta \leq R(\tilde{d}+1)(bd_1+d_2)^{-1}$ and let $\varepsilon > 0$. Then we have the upper bound
	\begin{equation} \label{eq.sup_bound}
		\sup_{\bm{\alpha} \in \mathfrak{m}(\Delta)}\norm{S(\bm{\alpha})} \ll P_1^{n_1}P_2^{n_2} P^{-\Delta \delta_0 + \varepsilon}.
	\end{equation}
	\end{lemma}
\begin{proof}
	The result follows straightforward from \cite[Lemma 4.3]{schindler_bihomogeneous} by setting the parameter $\theta$ to be 
	\begin{equation*}
		\theta = \frac{\Delta}{(\tilde{d}+1)R}.
	\end{equation*}
	If we have $0 < \Delta \leq R(\tilde{d}+1)(bd_1+d_2)^{-1}$ this ensures that the assumption $0 < \theta \leq (bd_1+d_2)^{-1}$ in \cite[Lemma 4.3]{schindler_bihomogeneous} is satisfied.
\end{proof}
Before we state the next proposition, recall that we assume $P_1 \geq P_2$ throughout, as was mentioned at the beginning of this section.
\begin{proposition} \label{prop.minor_arcs_estimate}
	Let $\varepsilon > 0$ and let $0 < \Delta \leq R(\tilde{d}+1)(bd_1+d_2)^{-1}$. Under the assumptions of Proposition \ref{prop.main_prop} we have 
	\begin{equation*}
		\int_{\mathfrak{m}(\Delta)}S(\bm{\alpha}) d \bm{\alpha} \ll P_1^{n_1-d_1R}P_2^{n_2-d_2R} P^{-\Delta \delta_0(1-(d_1+d_2)R/\mathscr{C})+\varepsilon}.
	\end{equation*}
\end{proposition}
\begin{proof}
We apply Lemma \ref{lem.general_integral_bound}	with
\begin{equation*}
	T(\bm{\alpha}) = C^{-1}P^{-\varepsilon} S(\bm{\alpha}), \quad E_0 = [0,1]^R, \quad E = \mathfrak{m}(\Delta), \quad \text{and} \quad \delta = \Delta \delta_0,
\end{equation*}
where $C > 0$ is some real number.
With these choices \eqref{eq.less_general_aux_ineq} follows from the auxiliary inequality \eqref{eq.aux_ineq} since for any $\varepsilon > 0$ we have $P^{-\varepsilon} \leq P_1^{-\varepsilon}$. From Lemma \ref{lem.sup_bound} we have the bound
\begin{equation*}
	\sup_{\bm{\alpha} \in E}CT(\bm{\alpha}) \ll P_1^{n_1}P_2^{n_2}P^{-\delta}.
\end{equation*}
We may increase $C$ if necessary so that we recover \eqref{eq.general_sup_bound}. Therefore the hypotheses of Lemma \ref{lem.good_int_bound}. Since we assume $\mathscr{C} > (bd_1+d_2)R$, we also note
\[
P_2^{-\mathscr{C}} = P^{-R} P^{R-\mathscr{C} (bd_1+d_2)^{-1}} \ll_{\mathscr{C}} P^{-R-\tilde{\delta}},
\]
for some $\tilde{\delta} >0$. Therefore if we assume $\mathscr{C} > (bd_1+d_2)R$ then Lemma~\ref{lem.good_int_bound} gives
\begin{equation*}
	\int_{\mathfrak{m}(\Delta)}S(\bm{\alpha}) d \bm{\alpha} \ll P_1^{n_1-d_1R}P_2^{n_2-d_2R} P^{-\Delta \delta_0(1-(d_1+d_2)R/\mathscr{C})+\varepsilon},
\end{equation*}
as desired.
\end{proof}

\subsection{The major arcs}
The aim of this section is to identify the main term via integrating the exponential sum $S(\bm{\alpha})$ over the major arcs, and analyse the singular integral and singular series appropriately. For $\bm{a} \in \mathbb{Z}^R$ and $q \in \mathbb{N}$ consider the complete exponential sum
\begin{equation*}
	S_{\bm{a},q} \coloneqq q^{-n_1-n_2} \sum_{\bm{x},\bm{y}} e\left( \frac{\bm{a}}{q} \cdot \bm{F}(\bm{x},\bm{y}) \right),
\end{equation*}
where the sum $\sum_{\bm{x},\bm{y}}$ runs through a complete set of residues modulo $q$. Further, for $P\geq 1$ and $\Delta > 0$ we define the truncated singular series
\begin{equation*}
	\mathfrak{S}(P) \coloneqq \sum_{q \leq P^\Delta} \sum_{\bm{a}} S_{\bm{a},q},
\end{equation*}
where the sum $\sum_{\bm{a}}$ runs over $\bm{a} \in \mathbb{Z}^R$ such that $0 \leq a_i < q$ for $i = 1, \hdots, R$ and $(a_1, \hdots, a_R, q) = 1$. For $\bm{\gamma} \in \mathbb{R}^R$ we further define
\begin{equation*}
	S_\infty(\bm{\gamma}) \coloneqq \int_{\mathcal{B}_1 \times \mathcal{B}_2}e \left( \bm{\gamma} \cdot \bm{F}(\bm{u},\bm{v}) \right) d\bm{u} d\bm{v},
\end{equation*}
and we define the truncated singular integral for $P \geq 1$, $\Delta >0$ as follows
\begin{equation*}
	\mathfrak{I}(P) \coloneqq \int_{\nnorm{\bm{\gamma}}_\infty \leq P^\Delta} S_\infty(\bm{\gamma}) d \bm{\gamma}.
\end{equation*}
From now on we assume that our parameter $\Delta>0$ satisfies
\begin{equation} \label{eq.delta_assumption}
	(bd_1+d_2)^{-1}>\Delta(2R+3)+\delta
\end{equation}
for some $\delta > 0$. Since $\mathscr{C} > R(bd_1+d_2)$ we are always able to choose such $\Delta$ in terms of $\mathscr{C}$. %We stress at this point that it is crucial for us to assume some bound on $b$ that we fix beforehand since we require the parameter $\Delta$ to be small in terms of $b$. Otherwise the resulting asymptotic is essentially meaningless. 
Further as in \cite{schindler_bihomogeneous} we now define some slightly modified major arcs $\mathfrak{M}'(\Delta)$ as follows
\begin{equation*}
	\mathfrak{M}'(\Delta) \coloneqq \bigcup_{1 \leq q \leq P^\Delta} \bigcup_{\substack{0 \leq a_i < q \\ (a_1, \hdots, a_R,q) = 1}}\mathfrak{M}_{\bm{a},q}'(\Delta) ,
\end{equation*}
where $\mathfrak{M}_{\bm{a},q}'(\Delta) = \left\{ \bm{\alpha} \in [0,1]^R \colon  \nnorm{\bm{\alpha}- \frac{\bm{a}}{q}}_\infty < P_1^{-d_1}P_2^{-d_2} P^\Delta \right\}$. The sets $\mathfrak{M}_{\bm{a},q}'$ are disjoint for our choice of $\Delta$; for if there is some 
\begin{equation*}
	\bm{\alpha} \in \mathfrak{M}_{\bm{a},q}'(\Delta) \cap \mathfrak{M}_{\bm{\tilde{a}},\tilde{q}}'(\Delta),
\end{equation*}
where $\mathfrak{M}_{\bm{\tilde{a}},\tilde{q}}'(\Delta) \neq \mathfrak{M}_{\bm{a},q}'(\Delta)$ then there is some $i \in \{1, \hdots, R \}$ such that
\begin{equation*}
	P^{-2\Delta} \leq \frac{1}{q \tilde{q}} \leq \norm{\frac{a_i}{q}- \frac{\tilde{a}_i}{\tilde{q}}} \leq 2P^{\Delta-1},
\end{equation*}
which is impossible for large $P$, since by our assumption \eqref{eq.delta_assumption} we have $3 \Delta -1 < 0$. Further we note that clearly $\mathfrak{M}'(\Delta) \supseteq \mathfrak{M}(\Delta)$ whence $\mathfrak{m}'(\Delta) \subseteq \mathfrak{m}(\Delta)$ and so the conclusions of Proposition \ref{prop.minor_arcs_estimate} hold with $\mathfrak{m}(\Delta)$ replaced by $\mathfrak{m}'(\Delta)$.

The next result expands the exponential sum $S(\bm{\alpha})$ when $\bm{\alpha}$ can be well-approximated by a rational number. In particular for our applications it is important to obtain an error term in which the constant does not depend on $\bm{\beta}$, whence we cannot just use Lemma 5.3 in~\cite{schindler_bihomogeneous} as it is stated there.

\begin{lemma} \label{lem.approx_exp_sum_major_arc}
	Let $\Delta>0$ satisfy \eqref{eq.delta_assumption}, let $\bm{\alpha} \in \mathfrak{M}'_{\bm{a},q}(\Delta)$ where $q \leq P^\Delta$, and write $\bm{\alpha} = \bm{a}/q + \bm{\beta}$ such that $1 \leq a_i < q$ and $(a_1, \hdots, a_R,q) = 1$. If $P_1 \geq P_2 > 1$ then
	\begin{equation} \label{eq.approx_S(alpha)_major_arcs}
		S(\bm{\alpha}) = P_1^{n_1}P_2^{n_2} S_{\bm{a},q} S_\infty(P \bm{\beta}) + O\left(qP_1^{n_1}P_2^{n_2-1}\left(1+ P\nnorm{\bm{\beta}}_\infty \right) \right),
	\end{equation}
	where the implied constant in the error term does not depend on $q$ or on $\bm{\beta}$.
\end{lemma}
\begin{proof}
	In the sum for $S(\bm{\alpha})$ we begin by writing $\bm{x} = \bm{z}^{(1)} + q \bm{x}'$ and $\bm{y} = \bm{z}^{(2)} + q \bm{y}'$ where $0 \leq z_i^{(1)} < q$ and $0 \leq z_j^{(2)} < q$ for all $1 \leq i \leq n_1$ and for all $1 \leq j \leq n_2$. A simple calculation now shows
	\begin{align} 
		S(\bm{\alpha}) &= \sum_{\bm{x} \in P_1 \mathcal{B}_1} \sum_{\bm{y} \in P_2 \mathcal{B}_2} e \left( \bm{\alpha} \cdot \bm{F} (\bm{x}, \bm{y}) \right) \nonumber  \\
		&= \sum_{\bm{z}^{(1)}, \bm{z}^{(2)} \, \mathrm{mod} \, q} %\sum_{\bm{z}^{(2)}q} 
		e \left( \frac{\bm{a}}{q} \cdot \bm{F} (\bm{z}^{(1)}, \bm{z}^{(2)}) \right) \tilde{S}(\bm{z}^{(1)}, \bm{z}^{(2)}) \label{eq.summing_things}
	\end{align}
	where 
	\begin{equation*}
		\tilde{S}(\bm{z}^{(1)}, \bm{z}^{(2)}) = \sum_{\bm{x}', \bm{y}'} e \left( \bm{\beta} \cdot \bm{F}(q\bm{x}'+\bm{z}^{(1)},q\bm{y}'+\bm{z}^{(2)} ) \right),
	\end{equation*}
	where $\bm{x}', \bm{y}'$ in the sum runs through integer tuples such that $q\bm{x}'+\bm{z}^{(1)} \in P_1 \mathcal{B}_1$ and  $q\bm{y}'+\bm{z}^{(2)} \in P_2 \mathcal{B}_2$ is satisfied. Now consider $\bm{x}', \bm{x}''$ and $\bm{y}', \bm{y}''$ such that
	\begin{equation*}
		\nnorm{\bm{x}'-\bm{x}''}_\infty, \nnorm{\bm{y}'-\bm{y}''}_\infty \leq 2.
	\end{equation*}
	Then for all $i = 1, \hdots, R$ we have
	\begin{multline*}
		\norm{F_i(q \bm{x}' + \bm{z}^{(1)}, q\bm{y}' + \bm{z}^{(2)}) - F_i(q \bm{x}'' + \bm{z}^{(1)}, q\bm{y}'' + \bm{z}^{(2)}) } \\
		\ll qP_1^{d_1-1} P_2^{d_2} + qP_1^{d_1} P_2^{d_2-1} \ll qP_1^{d_1} P_2^{d_2-1},
	\end{multline*}
	where we used $P_1 \geq P_2 > 1$ for the last estimate. We note that the implied constant here does not depend on $q$. We now use this to replace the sum in $\tilde{S}$ by an integral to obtain
	\begin{multline*}
		\tilde{S}(\bm{z}^{(1)}, \bm{z}^{(2)}) = \int_{q \tilde{\bm{v}} \in P_1 \mathcal{B}_1} \int_{q \tilde{\bm{w}} \in P_2 \mathcal{B}_2} e \left( \sum_{i=1}^R \beta_i F_i(q \tilde{\bm{v}}, q\tilde{\bm{w}}) \right) d \tilde{\bm{v}} d\tilde{\bm{w}} \\
		+ O \left( \nnorm{\bm{\beta}}_\infty qP_1^{d_1} P_2^{d_2-1} \left( \frac{P_1}{q} \right)^{n_1} \left( \frac{P_2}{q} \right)^{n_2} + \left( \frac{P_1}{q} \right)^{n_1} \left( \frac{P_2}{q} \right)^{n_2-1}  \right),
	\end{multline*}
	where we used that $q \leq P_2$, which is implied by our assumptions, but we mention here for the convenience of the reader.
	In the integral above we perform a substitution $\bm{v} = qP_1^{-1}\tilde{\bm{v}}$ and $\bm{w} = qP_2^{-1}\tilde{\bm{w}}$ to get
	\begin{equation*}
		\tilde{S}(\bm{z}^{(1)}, \bm{z}^{(2)}) = P_1^{n_1} P_2^{n_2} q^{-n_1-n_2} \mathfrak{I} (P \bm{\beta})  + q^{-n_1-n_2} O \left( qP_1^{n_1}P_2^{n_2-1}\left(1+ P\nnorm{\bm{\beta}}_\infty \right) \right),
	\end{equation*}
	where the implied constant does not depend on $\bm{\beta}$ or $q$. Substituting this into \eqref{eq.summing_things} gives the result.
\end{proof}
From the Lemma und using that the sets $\mathfrak{M}_{\bm{a},q}'$ are disjoint we  deduce
\begin{multline} \label{eq.needed_once}
	\int_{\mathfrak{M}'(\Delta)} S(\bm{\alpha})d \bm{\alpha} = P_1^{n_1}P_2^{n_2} \sum_{1 \leq q \leq P^\Delta} \sum_{\bm{a}} S_{\bm{a},q} \int_{\norm{\bm{\beta}}} S_\infty (P \bm{\beta}) d \bm{\beta} \\
	+ O\left(P_1^{n_1}P_2^{n_2} P^{2 \Delta} P_2^{-1} \mathrm{meas}\left(\mathfrak{M}'(\Delta)\right) \right),
\end{multline}
where we used $q \leq P^\Delta$ and $P \nnorm{\bm{\beta}}_\infty \leq P^\Delta$ for the error term. Now we can bound the measure of the major arcs by
\begin{equation*}
	\mathrm{meas}(\mathfrak{M}'(\Delta)) \ll \sum_{q \leq P^\Delta} q^R P^{-R+\Delta R} \ll P^{-R+ \Delta(2R+1)}.
\end{equation*}
Using this and making the substitution $\bm{\gamma} = P \bm{\beta}$ in the integral in \eqref{eq.needed_once} we find
\begin{equation} \label{eq.integral_major_arcs}
	\int_{\mathfrak{M}'(\Delta)} S(\bm{\alpha})d \bm{\alpha} = P_1^{n_1}P_2^{n_2} P^{-R} \mathfrak{S}(P) \mathfrak{I}(P) \\
	+ O \left( P_1^{n_1} P_2^{n_2} P^{-R+\Delta(2R+3)-1/(bd_1+d_2)} \right).
\end{equation}
It becomes transparent why the assumption~\eqref{eq.delta_assumption} is in place, because then the error term in~\eqref{eq.integral_major_arcs} is bounded by $O(P_1^{n_1}P_2^{n_2}P^{-R-\delta})$ and thus is of smaller order than the main term.

We now focus on the singular series $\mathfrak{S}(P)$ and the singular integral $\mathfrak{I}(P)$ in the next two Lemmas.
\begin{lemma} \label{lem.singular_series}
	Let $\varepsilon > 0$ and assume that the bound \eqref{eq.aux_ineq} holds for some $C \geq 1$, $\mathscr{C} > 1+b\varepsilon$, for all $\bm{\alpha}, \bm{\beta} \in \mathbb{R}^R$ and all real $P_1 \geq P_2 > 1$. %Assume further that we take $\mathcal{B}_i = [0,1]^{n_i}$. 
	Then we have the following:
	\begin{itemize}
		\item[(i)] For all $\varepsilon' > 0$ such that $\varepsilon' = O_{\mathscr{C}}(\varepsilon)$ we have
		\begin{equation*} \label{eq.aux_ineq_consequence}
			\min\left\{ \norm{S_{\bm{a},q}}, \norm{S_{\bm{a}',q'}} \right\} \ll_C (q'+q)^\varepsilon \nnorm{\frac{\bm{a}}{q} - \frac{\bm{a}'}{q'}}_\infty^{\frac{\mathscr{C}-\varepsilon'}{\tilde{d}+1}}
		\end{equation*}
		for all $q,q' \in \mathbb{N}$ and all $\bm{a} \in \{1, \hdots, q \}^R$ and $\bm{a}' \in \{1, \hdots, q' \}^R$ with $\frac{\bm{a}}{q} \neq \frac{\bm{a}'}{q'}$.
		\item[(ii)] If $\mathscr{C} > \varepsilon'$ then for all $t \in \mathbb{R}_{>0}$ and $q_0 \in \mathbb{N}$ we have
		\begin{equation*}
			\# \left\{ \frac{\bm{a}}{q} \in [0,1]^R \cap \mathbb{Q}^R \colon q \leq q_0, \norm{S_{\bm{a},q}} \geq t \right\} \ll_C (q_0^{-\varepsilon} t)^{-\frac{(\tilde{d}+1)R}{\mathscr{C}-\varepsilon'}},
		\end{equation*}
		where the fractions in the set above are in lowest terms.
		\item[(iii)] Assume that the forms $F_i(\bm{x},\bm{y})$ are linearly independent. Then for all $q \in \mathbb{N}$ and $\bm{a} \in \mathbb{Z}^R$ with $(a_1, \hdots, a_R, q) = 1$ there exists some $\nu >0$ depending at most on $d_i$ and $R$  such that
		\begin{equation*}
			\norm{S_{\bm{a},q}} \ll q^{-\nu}.
		\end{equation*}
		\item[(iv)] Assume $\mathscr{C} >(\tilde{d}+1)R$ and assume the forms $F_i(\bm{x},\bm{y})$ are linearly independent. Then the \emph{singular series} \begin{equation} \label{eq.def_singular_series} \mathfrak{S} = \sum_{q=1}^\infty \sum_{\bm{a} \, \mathrm{ mod } \, q} S_{\bm{a},q}\end{equation}  exists and converges absolutely, with
		\begin{equation*}
			\norm{\mathfrak{S}(P) - \mathfrak{S}} \ll_{C, \mathscr{C}} P^{-\Delta \delta_1},
		\end{equation*}
		for some $\delta_1 > 0$ depending only on $\mathscr{C}, d_i$ and $R$.
	\end{itemize}
\end{lemma}
\begin{proof}[Proof of \emph{(i)}]
	Take $\mathcal{B}_i = [0,1]^{n_i}$ so that $S_\infty(\bm{0}) = 1$. Therefore \eqref{eq.approx_S(alpha)_major_arcs} implies that 
	\begin{equation*}
		\frac{S\left( \frac{\bm{a}}{q} \right)}{P_1^{n_1}P_2^{n_2}} = S_{\bm{a},q} + O \left( q P_2^{-1} \right) \quad \text{and} \quad \frac{S\left( \frac{\bm{a}'}{q'} \right)}{P_1^{n_1}P_2^{n_2}} = S_{\bm{a}',q'} + O \left( q' P_2^{-1} \right).
	\end{equation*}
	Using this and the bound \eqref{eq.aux_ineq} we obtain
	\begin{multline} \label{eq.inequality_1}
		\min \left\{ \norm{S_{\bm{a},q}}, \norm{S_{\bm{a}',q'}} \right\} \leq  CP_1^\varepsilon P^{-\mathscr{C}} \nnorm{\frac{\bm{a}}{q} - \frac{\bm{a}'}{q'}}_\infty^{-\mathscr{C}} + \\
		CP_1^\varepsilon \nnorm{\frac{\bm{a}}{q} - \frac{\bm{a}'}{q'}}_\infty^{\frac{\mathscr{C}}{\tilde{d}+1}} + O \left( (q'+q)P_2^{-1} \right),
	\end{multline}
	where we note that $P_1^\varepsilon P_2^{-\mathscr{C}} = O(P_2^{-1})$ due to our assumptions on $\mathscr{C}$.
	Now set 
	\begin{equation*}
		P_1 = P_2 = (q+q') \nnorm{\frac{\bm{a}}{q} - \frac{\bm{a}'}{q'}}_\infty^{-\frac{1+\mathscr{C}}{\tilde{d}+1}}.
	\end{equation*}
	Note $(q+q') \geq 1$ and $ \nnorm{\frac{\bm{a}}{q} - \frac{\bm{a}'}{q'}}_\infty \leq 1$ so that this gives $P_i \geq 1$. Substituting these choices into \eqref{eq.inequality_1} we get
	\begin{multline*}
		\min \left\{ \norm{S_{\bm{a},q}}, \norm{S_{\bm{a}',q'}} \right\} \leq
		  P_1^{ \varepsilon} (q+q')^{-\mathscr{C}(d_1+d_2)} \nnorm{\frac{\bm{a}}{q} - \frac{\bm{a}'}{q'}}_\infty^{\frac{\mathscr{C}^2+\mathscr{C}}{\tilde{d}+1}(d_1+d_2)-\mathscr{C}} + \\
		  CP_1^{\varepsilon} \nnorm{\frac{\bm{a}}{q} - \frac{\bm{a}'}{q'}}_\infty^{\frac{\mathscr{C}}{\tilde{d}+1}} + O \left(\nnorm{\frac{\bm{a}}{q} - \frac{\bm{a}'}{q'}}_\infty^{\frac{1+\mathscr{C}}{(\tilde{d}+1)}} \right).
	\end{multline*} 
	Noting again that $(q+q') \geq 1$, $ \nnorm{\frac{\bm{a}}{q} - \frac{\bm{a}'}{q'}}_\infty \leq 1$ and also that $\frac{\mathscr{C}^2+\mathscr{C}}{\tilde{d}+1}(d_1+d_2)-\mathscr{C} \geq \frac{\mathscr{C}}{\tilde{d}+1}$ we see that the second term on the right hand side above dominates the expression. Hence we finally obtain
	\begin{equation*}
		\min \left\{ \norm{S_{\bm{a},q}}, \norm{S_{\bm{a}',q'}} \right\} \ll_C P_1^{ \varepsilon} \nnorm{\frac{\bm{a}}{q} - \frac{\bm{a}'}{q'}}_\infty^{\frac{\mathscr{C}}{\tilde{d}+1}} = (q'+q)^\varepsilon \nnorm{\frac{\bm{a}}{q} - \frac{\bm{a}'}{q'}}_\infty^{\frac{\mathscr{C}-\varepsilon'}{\tilde{d}+1}},
	\end{equation*}
	for some $\varepsilon' = O_{\mathscr{C}}(\varepsilon)$.
	\end{proof}
	
	\begin{proof}[Proof of \emph{(ii)}]
	This now follows almost directly from (i). The points in the set	 
	\begin{equation*}
		\left\{ \frac{\bm{a}}{q} \in [0,1]^R \cap \mathbb{Q}^R \colon q \leq q_0, \norm{S_{\bm{a},q}} \geq t \right\}
	\end{equation*}
	are separated by gaps of size $\gg_C (q_0^{-\varepsilon} t )^{\frac{\tilde{d}+1}{\mathscr{C}-\varepsilon'}}$. Hence at most $O_{C}((q_0^{-\varepsilon} t )^{-\frac{\tilde{d}+1}{\mathscr{C}-\varepsilon'}})$ fit in the box $[0,1]^R$ so the result follows.
	\end{proof}
	
	\begin{proof}[Proof of \emph{(iii)}]
	Setting $P_1=P_2=q$ and $\bm{\alpha} = \bm{a}/ q$ we find $S_{\bm{a},q} = q^{-n_1-n_2} S(\bm{\alpha})$. Let $\delta_0$ be defined as in \eqref{eq.delta_0_definition}. We can define $\Delta$ by $(d_1+d_2) \Delta = 1- \varepsilon''$ for some $\varepsilon'' \in (0,1)$. We claim that $\bm{a}/q$ does not lie in the major arcs $\mathfrak{M}(\Delta)$ if $(a_1, \hdots, a_r, q) = 1$. For if, then there exist $q', \bm{a}'$ such that 
	\begin{equation*}
		1 \leq q' \leq q^{(d_1+d_2) \Delta},
	\end{equation*}	
	and
	\begin{equation*}
		2 \norm{q'a_i - q a_i'} \leq q^{1-d_1-d_2}q^{(d_1+d_2) \Delta} < 1,
	\end{equation*}
	which is clearly impossible. The bound \eqref{eq.sup_bound} applied to our situation gives
	\begin{equation*}
		\norm{S_{\bm{a},q}} \ll q^{-R \delta_0(1-\varepsilon'')+\varepsilon}.
	\end{equation*}
	As the forms $F_i$ are linearly independent we know that $\delta_0 \geq \frac{1}{(\tilde{d}+1)2^{\tilde{d}}R}$. Thus, choosing some small enough $\varepsilon$ delivers the result.
	\end{proof}
	
	\begin{proof}[Proof of \emph{(iv)}]
	For $Q > 0$ let
	\begin{equation*}
		s(Q) = \sum_{\substack{\bm{a}/q \in [0,1)^R \\ Q < q \leq 2Q}} \norm{S_{\bm{a},q}},
	\end{equation*}	
	where $\sum_{\bm{a}/q \in [0,1)^R}$ is shorthand for the sum  $\sum_{q=1}^\infty \sum_{\substack{\nnorm{\bm{a}}_\infty \leq q}}$ such that $(a_1, \hdots, a_R,q) = 1$. We claim that $s(Q) \ll_{C, \mathscr{C}} Q^{-\delta_1}$ for some $\delta_1 > 0$. To see this, let $\ell \in \mathbb{Z}$. Then
	\begin{multline} \label{eq.s(Q)_first_est}
		s(Q) = \sum_{\substack{\bm{a}/q \in [0,1)^R \\ Q < q \leq 2Q \\ \norm{S_{\bm{a},q}} \geq 2^{-\ell}}} \norm{S_{\bm{a},q}} + \sum_{i = \ell}^\infty \sum_{\substack{\bm{a}/q \in [0,1)^R \\ Q < q \leq 2Q \\ 2^{-i} > \norm{S_{\bm{a},q}} \geq 2^{-i-1}}} \norm{S_{\bm{a},q}} \\
		\leq \# \left\{ \frac{\bm{a}}{q} \in [0,1)^R \cap \mathbb{Q}^R \colon q \leq 2Q, \norm{S_{\bm{a},q}} \geq 2^{-\ell} \right\} \cdot \sup_{q > Q} \norm{S_{\bm{a},q}} \\
		+ \sum_{i = \ell}^\infty \# \left\{ \frac{\bm{a}}{q} \in [0,1)^R \cap \mathbb{Q}^R \colon q \leq 2Q, \norm{S_{\bm{a},q}} \geq 2^{-i-1} \right\} \cdot 2^{-i}.
	\end{multline}
	Now from (ii) we know 
	\begin{equation*}
		\# \left\{ \frac{\bm{a}}{q} \in [0,1)^R \cap \mathbb{Q}^R \colon q \leq 2Q, \norm{S_{\bm{a},q}} \geq t \right\} \ll_C (Q^{-\varepsilon}t)^{-\frac{(\tilde{d}+1)R}{\mathscr{C}-\varepsilon'}},
	\end{equation*}
	and from (iii) we know, since $F_i$ are linearly independent there is some $\nu >0$ such that
	\begin{equation*}
		\sup_{q > Q} \norm{S_{\bm{a},q}} \ll Q^{-\nu}.
	\end{equation*}
	Using these estimates in \eqref{eq.s(Q)_first_est} we  get
	\begin{equation*}
		s(Q) \ll_C Q^{O_{\mathscr{C}}(\varepsilon) - \nu} 2^{\ell \frac{(\tilde{d}+1)R}{\mathscr{C}-\varepsilon'}} + Q^{O_{\mathscr{C}}(\varepsilon)} \sum_{i = \ell}^\infty 2^{(i+1) \left( \frac{(\tilde{d}+1)R}{\mathscr{C}-\varepsilon''} \right) - i}.
	\end{equation*}
	Since we assumed $\mathscr{C} > (\tilde{d}+1)R$ and since $\varepsilon'$ is small in terms of $\mathscr{C}$ we may also assume $\mathscr{C} > (\tilde{d}+1)R + \varepsilon'$. Therefore, summing the geometric expression gives
	\begin{equation*}
		s(Q) \ll_{C,\mathscr{C}} Q^{O_{\mathscr{C}}(\varepsilon)} 2^{\ell \frac{(\tilde{d}+1)R}{\mathscr{C}-\varepsilon'}} \left( Q^{-\nu}+ 2^{-\ell} \right).
	\end{equation*}
	Now choose $\ell = \lfloor \log_2 Q^{\nu} \rfloor$ to get
	\begin{equation*}
		s(Q) \ll Q^{\nu \frac{(\tilde{d}+1)R-\mathscr{C}}{ \mathscr{C}} + O_{\mathscr{C}}(\varepsilon)}.
	\end{equation*}
Letting $\varepsilon$ be small enough in terms of $\mathscr{C}, d_i$, $R$ we get some $\delta_1 > 0$ depending on $\mathscr{C}, d_i$ and $R$ such that
	\begin{equation*}
		s(Q) \ll Q^{-\delta_1},
	\end{equation*}
	which proves the claim.
	Finally using this and splitting the sum into dyadic intervals we find
	\begin{equation*}
		\norm{\mathfrak{S}(P)-\mathfrak{S}} \leq \sum_{\substack{\bm{a}/q \in [0,1)^R \\ q > P^\Delta}} \norm{S_{\bm{a},q}} = \sum_{k=0}^\infty\sum_{\substack{Q = 2^k P^\Delta}} s(Q) \ll \sum_{k=0}^\infty \left( 2^k P^\Delta \right)^{-\delta_1},
	\end{equation*}
	which proves (iv).
	\end{proof}
	
	The next Lemma handles the singular integral.
	
\begin{lemma} \label{lem.singular_integral}
	Let $\varepsilon > 0$ and assume that the bound \eqref{eq.aux_ineq} holds for some $C \geq 1$, $\mathscr{C} > 1+b\varepsilon$ and for all $\bm{\alpha}, \bm{\beta} \in \mathbb{R}^R$ and all real $P_1 \geq P_2 > 1$. Then:
	\begin{itemize}
		\item[(i)] For all $\bm{\gamma} \in \mathbb{R}^R$ we have
		\begin{equation*}
			S_\infty(\bm{\gamma}) \ll_C \nnorm{\bm{\gamma}}_\infty^{-\mathscr{C} + \varepsilon'},
		\end{equation*} 
		for some $\varepsilon >0$ such that $\varepsilon' = O_{\mathscr{C}}(\varepsilon)$. 
		\item[(ii)] Assume that $\mathscr{C}-\varepsilon' > R$. Then for all $P_1, P_2 > 1$ we have
		\begin{equation*}
			\norm{\mathfrak{I}(P) - \mathfrak{I}} \ll_{\mathscr{C},C,\varepsilon'} P^{- \Delta (\mathscr{C}-\varepsilon'-R)},
		\end{equation*}
		where $\mathfrak{I}$ is the \emph{singular integral}
		\begin{equation} \label{eq.def_singular_integral}
			\mathfrak{I} = \int_{\bm{\gamma} \in \mathbb{R}^R} S_\infty(\bm{\gamma}) d \bm{\gamma}.
		\end{equation}
		In particular we see that $\mathfrak{I}$ exists and converges absolutely.
	\end{itemize}
\end{lemma}
\begin{proof}[Proof of \emph{(i)}]
	It is easy to see that for all $\bm{\beta} \in \mathbb{R}^R$ we have $\norm{S(\bm{\beta})} \leq \norm{S(\bm{0})}$. Thus applying \eqref{eq.aux_ineq} with $\bm{\alpha} = \bm{0}$ and $\bm{\beta} = P^{-1}\bm{\gamma}$ we get
	\begin{equation} \label{eq.used_1}
		\norm{S(P^{-1} \bm{\gamma})} \leq CP_1^{n_1}P_2^{n_2} P_1^\varepsilon \max \left\{P_2^{-1} \nnorm{\bm{\gamma}}_\infty^{-1}, P^{-\frac{1}{\tilde{d}+1}} \nnorm{\bm{\gamma}}^{\frac{1}{\tilde{d}+1}} \right\}^{\mathscr{C}}.
	\end{equation}
	Now from \eqref{eq.approx_S(alpha)_major_arcs} with $\bm{a} = \bm{0}$ and $\bm{\beta} = P^{-1} \bm{\gamma}$ we have
	\begin{equation} \label{eq.used_2}
		S(P^{-1} \bm{\gamma}) = P_1^{n_1}P_2^{n_2} S_\infty(\bm{\gamma}) + O \left( P_1^{n_1}P_2^{n_2-1}(1+\nnorm{\bm{\gamma}}_\infty) \right),
	\end{equation}
	where we used as in the proof of part (i) Lemma~\ref{lem.singular_series} that $P_1^\varepsilon P_2^{-\mathscr{C}} \leq P_2^{-1}$ due to our assumptions on $\mathscr{C}$.
	Combining \eqref{eq.used_1} and \eqref{eq.used_2} we obtain
	\begin{equation*}
		S_\infty(\bm{\gamma}) \ll_C P_1^\varepsilon \max \left\{ \nnorm{\bm{\gamma}}_\infty^{-1}, P^{-\frac{1}{\tilde{d}+1}} \nnorm{\bm{\gamma}}_\infty^{\frac{1}{\tilde{d}+1}} \right\}^{\mathscr{C}} + P_2^{-1} + \nnorm{\bm{\gamma}}_\infty P_2^{-1}.
	\end{equation*}
	Taking $P_1=P_2 = \max\{1, \nnorm{\bm{\gamma}}_\infty^{1+\mathscr{C}}\}$ gives the result.
\end{proof}
\begin{proof}[Proof of \emph{(ii)}]
For this simply note that by part (i) we get
\begin{equation*}
	\norm{\mathfrak{I}(P) - \mathfrak{I}} = \int_{\nnorm{\bm{\gamma}}_\infty \geq P^\Delta} S_\infty(\bm{\gamma}) d \bm{\gamma} \ll_{\mathscr{C},C,\varepsilon'} \int_{\nnorm{\bm{\gamma}}_\infty \geq P^\Delta} \nnorm{\bm{\gamma}}_\infty^{-\mathscr{C}-\varepsilon'} d\bm{\gamma} \ll P^{-\Delta(\mathscr{C}-\varepsilon'-R)},
\end{equation*}
where the last estimate follows since we assumed $\mathscr{C}-\varepsilon' > R$.
\end{proof}
Before we finish the proof of the main result we state two different expressions for the singular series and the singular integral that will be useful later on. If $\mathscr{C} > R(d_1+d_2)$ then $\mathfrak{I}$ and $\mathfrak{S}$ converge absolutely, as was shown in the previous two Lemmas. Therefore, as in \S7 of~\cite{birch_forms}, by regarding the bihomogeneous forms under investigation simply as homogeneous forms we may express the singular series as an absolutely convergent product
\begin{equation} \label{eq.sing_series_p_adic_expression}
\mathfrak{S} = \prod_p \mathfrak{S}_p,
\end{equation}
where 
\[
\mathfrak{S}_p = \lim_{k \rightarrow \infty} \frac{1}{p^{k(n_1+n_2-R)}} \# \left\{ (\bm{u}, \bm{v}) \in \{1, \hdots, p^k \}^{n_1+n_2} \colon F_i(\bm{u}, \bm{v}) \equiv 0 \; (\mathrm{mod} \, p), i = 1, \hdots, R \right\}.
\]
Lemma 2.6 in~\cite{myerson_quadratic} further shows that we can write the singular integral as
\begin{equation} \label{eq.sing_series_real_density}
\mathfrak{I} = \lim_{P \rightarrow \infty} \frac{1}{P^{n_1+n_2-(d_1+d_2)R}} \mu \big\{ (\bm{t}_1, \bm{t}_2)/P \in \mathcal{B}_1 \times \mathcal{B}_2 \colon \\
 \norm{F_i(\bm{t}_1, \bm{t}_2)} \leq 1/2, \; i = 1, \hdots, R \big\},
\end{equation}
where $\mu( \cdot )$ denotes the Lebesgue measure. We may therefore interpret the quantities $\mathfrak{I}$ and $\mathfrak{S}_p$ as the real and $p$-adic \emph{densities}, respectively, of the system of equations $F_1(\bm{x}, \bm{y}) = \cdots = F_R(\bm{x}, \bm{y}) = 0$. 
\subsection{Proofs of Proposition~\ref{prop.main_prop} and Theorem~\ref{thm.n_aux_imply_result}}
\begin{proof}[Proof of Proposition \ref{prop.main_prop}]
	From Proposition \ref{prop.minor_arcs_estimate}, the estimate \eqref{eq.integral_major_arcs}, Lemma \ref{lem.singular_series} and Lemma \ref{lem.singular_integral}, for any $\varepsilon > 0$ we find
	\begin{equation*}
		\frac{N(P_1,P_2)}{P_1^{n_1}P_2^{n_2}P^{-R}} - \mathfrak{S} \mathfrak{I}  \ll 
		 P^{-\Delta \delta_1} + P^{-\Delta \delta_0(1-(d_1+d_2)R/\mathscr{C} ) +  \varepsilon } +  P^{(2R+3) \Delta - 1/(bd_1+d_2)} +  P^{-\Delta(\mathscr{C}-\varepsilon'-R)}.
	\end{equation*}
	for some $\delta_1 > 0$ and some $1> \varepsilon' >0$.
	Recall we assumed $\mathscr{C} > (bd_1+d_2)R$ and assuming the forms $F_i$ are linearly independent we also have $\delta_0 \geq \frac{1}{(\tilde{d}+1)2^{\tilde{d}}R}$. Therefore choosing suitably small $\Delta > 0$ there exists some $\delta > 0$ such that 
	\begin{equation*}
		\frac{N(P_1,P_2)}{P_1^{n_1}P_2^{n_2}P^{-R}} - \mathfrak{S} \mathfrak{I}  \ll P^{-\delta}
	\end{equation*}
	as desired. Finally, since we assume that the equations $F_i$ define a complete intersection, it is a standard fact to see that $\mathfrak{S}$ is positive if there exists a non-singular $p$-adic zero for all primes $P$, and similarly $\mathfrak{I}$ is positive if there exists a non-singular real zero within $\mathcal{B}_1 \times \mathcal{B}_2$. A detailed argument of this fact using a version of Hensel's Lemma for $\mathfrak{S}$ and the implicit function theorem for $\mathfrak{I}$ can be found for example in \S4 of~\cite{myerson_quadratic}. 
	\end{proof}
We finish this section by deducing the technical main theorem, namely Theorem~\ref{thm.n_aux_imply_result}.
\begin{proof}[Proof of Theorem~\ref{thm.n_aux_imply_result}]
	Assume the estimate in~\eqref{eq.N_aux_general_condition} holds for some constant $C_0>0$. From Proposition~\ref{prop.auxiliary_ineq_from_counting_function} it thus follows that the auxiliary inequality~\eqref{eq.aux_ineq} holds with a constant $C > 0$ depending on $C_0$, $d_i$, $n_i$, $\mu$ and $M$, where all of these quantities follow the same notation as in Section~\ref{sec.weyl_differencing_etc}. Therefore the assumptions of Proposition~\ref{prop.main_prop} so we can apply it to obtain the desired conclusions.
\end{proof}

\section{Systems of bilinear forms} \label{sec.bilinear_proof}
In this section we assume $d_1=d_2=1$. Then we can write our system as
\begin{equation*}
	F_i(\bm{x},\bm{y}) = \bm{y}^T A_i \bm{x},
\end{equation*}
where $A_i$ are $n_2 \times n_1$-dimensional matrices with integer entries. For $\bm{\beta} \in \mathbb{R}^R$ we now have 
\begin{equation*}
	\bm{\beta} \cdot \bm{F} = \bm{y}^T A_{\bm{\beta}} \bm{x},
\end{equation*}
where $A_{\bm{\beta}} = \sum_i \beta_i A_i$. Recall that we put 
\begin{equation*}
	\sigma_{\mathbb{R}}^{(1)} = \max_{\bm{\beta} \in \mathbb{R}^R \setminus \{0 \} } \dim \ker ( A_{\bm{\beta}}) \quad \text{and} \quad \sigma_{\mathbb{R}}^{(2)} = \max_{\bm{\beta} \in \mathbb{R}^R \setminus \{0 \} } \dim \ker ( A_{\bm{\beta}}^T).
\end{equation*}
Since the row rank of a matrix is equal to its column rank we can also define
\begin{equation*}
	\rho_{\mathbb{R}} \coloneqq \min_{\bm{\beta} \in \mathbb{R}^R \setminus \{ 0\}} \mathrm{rank} (A_{\bm{\beta}}) = \min_{\bm{\beta} \in \mathbb{R}^R \setminus \{ 0\}} \mathrm{rank} (A_{\bm{\beta}}^T).
\end{equation*}
Due to the rank-nullity theorem the conditions
\[
n_i-\sigma_\R^{(i)} > (2b+2)R
\]
for $i=1,2$ are equivalent to 
\begin{equation*} \label{eq.condition_rho}
	\rho_{\mathbb{R}} > (2b+2)R.
\end{equation*}

\begin{lemma} \label{lem.bilinear_smooth_complete_assumptions}
	Assume that $\mathbb{V}(F_1, \hdots, F_R) \subset \mathbb{P}_{\mathbb{C}}^{n_1-1} \times \mathbb{P}_{\mathbb{C}}^{n_2-1}$ is a smooth complete intersection. Let $b \geq 1$ be a real number. Assume further
	\begin{equation} \label{eq.n_i_assumption_bilinear_lemma}
		\min\{n_1,n_2\} > (2b+2)R, \quad \text{and} \quad n_1+n_2 > (4b+5)R.
	\end{equation}
	 Then we have
	\begin{equation} \label{eq.n_i_assumption_lemma}
		n_i - \sigma_{\mathbb{R}}^{(i)} > (2b+2)R
	\end{equation} 
	for $i = 1,2$.
\end{lemma}
\begin{proof}
	Without loss of generality assume $n_1 \geq n_2$. Pick $\bm{\beta} \in \mathbb{R}^R\setminus \{\bm{0}\}$ such that $\mathrm{rank} (A_{\bm{\beta}}) = \rho_{\mathbb{R}}$. In particular then 
	\begin{equation*}
		\dim \ker (A_{\bm{\beta}}) = \sigma_{\mathbb{R}}^{(1)}, \quad \text{and} \quad \dim \ker (A_{\bm{\beta}}^T) = \sigma_{\mathbb{R}}^{(2)}.
	\end{equation*}
	We proceed in distinguishing two cases. Firstly, if $\sigma_{\mathbb{R}}^{(2)} = 0$ then~\eqref{eq.n_i_assumption_lemma} follows for $i = 2$ by the assumption~\eqref{eq.n_i_assumption_bilinear_lemma}. Further by comparing row rank and column rank of $A_{\bm{\beta}}$ in this case we must then have $\sigma_{\mathbb{R}}^{(1)} \leq n_1-n_2$, and therefore
	\begin{equation*}
		n_1 - \sigma_{\mathbb{R}}^{(1)} \geq n_2 > (2b+2)R,
	\end{equation*}
	so~\eqref{eq.n_i_assumption_lemma} follows for $i=1$.
	
	Now we turn to the case $\sigma_{\mathbb{R}}^{(2)} > 0$. Then also $\sigma_{\mathbb{R}}^{(1)} > 0$. The singular locus of the variety $ \mathbb{V}(\bm{\beta} \cdot \bm{F}) \subset \mathbb{P}_{\mathbb{C}}^{n_1-1} \times \mathbb{P}_{\mathbb{C}}^{n_2-1}$ is given by
	\begin{equation*}
		\mathrm{Sing} \mathbb{V}(\bm{\beta} \cdot \bm{F}) = \mathbb{V}(\bm{y}^T A_{\bm{\beta}}) \cap \mathbb{V} (A_{\bm{\beta}} \bm{x}).
	\end{equation*}
	Therefore we have 
	\begin{equation*}
		\dim \mathrm{Sing} \mathbb{V}(\bm{\beta} \cdot \bm{F}) = \sigma_{\mathbb{R}}^{(1)} + \sigma_{\mathbb{R}}^{(2)}-2.
	\end{equation*}
 	Since we assumed $\mathbb{V}(\bm{F})$ to be a smooth complete intersection we can apply Lemma~\ref{lem.sing_bf_small} to get $\dim \mathrm{Sing} \mathbb{V}(\bm{\beta} \cdot \bm{F}) \leq R-2$. Therefore we find
 	\begin{equation*}
 		\sigma_{\mathbb{R}}^{(1)} + \sigma_{\mathbb{R}}^{(2)} \leq R.
 	\end{equation*}
 	From our previous remarks we know that showing~\eqref{eq.n_i_assumption_lemma} is equivalent to showing $\rho_{\mathbb{R}} > (2b+2)R$. But now
 	\begin{equation*}
 		\rho_{\mathbb{R}} = \frac{1}{2} \left(n_1+n_2 - \sigma_{\mathbb{R}}^{(1)} - \sigma_{\mathbb{R}}^{(2)} \right) \geq \frac{1}{2} (n_1+n_2 - R) > (2b+2)R,
 	\end{equation*}
 	where the last inequality followed from the assumption~\eqref{eq.n_i_assumption_bilinear_lemma}. Therefore~\eqref{eq.n_i_assumption_lemma} follows as desired.	
\end{proof}

\begin{proof}[Proof of Theorem~\ref{thm.bilinear}]
Recall the notation $b = \frac{\log P_1}{\log P_2}$. By virtue of Theorem~\ref{thm.n_aux_imply_result} it suffices to show that assuming
\begin{equation*}
	n_i - \sigma_{\mathbb{R}}^{(i)} > (2b+2)R
\end{equation*}
for $i=1,2$ implies~\eqref{eq.N_aux_general_condition}. We will show~\eqref{eq.N_aux_general_condition} for $i=1$, the other case follows analogously. Let $\mathscr{C}  = \frac{n_2-\sigma_{\mathbb{R}}^{(2)}}{2}$ and we note that we have $\mathscr{C} > (bd_1+d_2)R = (b+1)R$ precisely when $n_2 - \sigma_{\mathbb{R}}^{(2)} > (2b+2)R$ holds. Therefore it suffices to show that
\begin{equation} \label{eq.N_1_small_sigma_2}
	N_1^{\mathrm{aux}}(\bm{\beta},B) \ll B^{\sigma_{\mathbb{R}}^{(2)}}.
\end{equation}
for all $\bm{\beta} \in \mathbb{R}^R \setminus \{ \bm{0}\}$ with the implied constant not depending on $\bm{\beta}$. In our case we have 
\begin{equation*}
\bm{\Gamma}(\bm{u}) = \bm{u}^T A(\bm{\beta}),
\end{equation*}
where $\bm{u} \in \mathbb{Z}^{n_2}$. Therefore $N_1^{\mathrm{aux}}(\bm{\beta},B)$ counts vectors $\bm{u} \in \mathbb{Z}^{n_2}$ such that
\begin{equation*}
	\nnorm{\bm{u}}_\infty \leq B \quad \text{and} \quad \nnorm{\bm{u}^T A(\bm{\beta})}_\infty \leq \nnorm{A(\bm{\beta})}_\infty = \nnorm{\bm{\beta} \cdot \bm{F}}_\infty.
\end{equation*}
In particular, all of the vectors $\bm{u} \in \mathbb{Z}^{n_2}$, which are counted by $N_1^{\mathrm{aux}}(\bm{\beta},B)$ are contained in the ellipsoid
\begin{equation*}
	E_{\bm{\beta}} \coloneqq \left\{ \bm{t} \in \mathbb{R}^{n_2} \colon \bm{t}^T A_{\bm{\beta}} A_{\bm{\beta}}^T \bm{t} < n_2\nnorm{\bm{\beta} \cdot \bm{F}}_\infty^2 \right\}.
\end{equation*}
The principal radii of $E_{\bm{\beta}}$ are given by $\norm{\lambda_i}^{-1} n_2^{1/2} \nnorm{\bm{\beta} \cdot \bm{F}}_\infty$ for $i = 1, \hdots, n_2$, where $\lambda_i$ run through the $n_2$ singular values of $A_{\bm{\beta}}$ and are listed in increasing order of absolute value. Thus we find
\begin{equation*}
	N_1^{\mathrm{aux}}(\bm{\beta},B) \ll \prod_{i = 1}^{n_2} \min \left\{ \norm{\lambda_i}^{-1} \nnorm{\bm{\beta} \cdot \bm{F}}_\infty +1,B \right\}.
\end{equation*}
If $\left\vert\lambda_{\sigma_{\mathbb{R}}^{(2)}+1} \right\vert \gg \nnorm{\bm{\beta} \cdot \bm{F}}_\infty$ holds then~\eqref{eq.N_1_small_sigma_2} would follow. So suppose for a contradiction that there exists a sequence $(\bm{\beta}^{(i)})$ such that $\norm{\lambda_{\sigma_{\mathbb{R}}^{(2)}+1}} = o\left(\nnorm{\bm{\beta}^{(i)} \cdot \bm{F}}_\infty\right)$. Let $\bm{\beta}$ be the limit of a subsequence of $\bm{\beta}^{(i)}/\nnorm{\bm{\beta}^{(i)}}$, which must exist by the Bolzano--Weierstrass theorem. For this $\bm{\beta}$ we must then have $\lambda_{\sigma_{\mathbb{R}}^{(2)}+1} = 0$. Since the singular values were listed in order of increasing absolute value it follows that
\begin{equation*}
	\lambda_1 = \cdots = \lambda_{\sigma_{\mathbb{R}}^{(2)}+1} = 0,
\end{equation*}
and so $\dim \ker A_{\bm{\beta}}^T = \sigma_{\mathbb{R}}^{(2)}+1$. This contradicts the maximality of $\sigma_{\mathbb{R}}^{(2)}+1$.

The second part of the theorem is now a direct consequence of Lemma~\ref{lem.bilinear_smooth_complete_assumptions}.
\end{proof}

\section{Systems of forms of bidegree $(2,1)$} \label{sec.proof.2-1}
We consider a system $\bm{F}(\bm{x},\bm{y})$ of homogeneous equations of bidegree $(2,1)$, where $\bm{x} = (x_1, \hdots, x_{n_1})$ and $\bm{y} = (y_1, \hdots, y_{n_2})$.  We will first assume $n_1=n_2 = n$, say, and then deduce Theorem~\ref{thm.2,1_different_dimensions} afterwards. Therefore the initial main goal is to establish the following.

\begin{proposition} \label{thm.2,1}
	Let $F_1(\bm{x},\bm{y}), \hdots, F_R(\bm{x},\bm{y})$ be bihomogeneous forms of bidegree $(2,1)$ such that the biprojective variety $\mathbb{V}(F_1, \hdots, F_R) \subset \mathbb{P}^{n-1}_{\mathbb{Q}} \times \mathbb{P}^{n-1}_{\mathbb{Q}}$ is a complete intersection.  Write $b = \max \{\log P_1/\log P_2, 1 \}$ and $u = \max \{\log P_2/\log P_1, 1 \}$
	Assume that 
	\begin{equation} \label{eq.assumption_n_i_(2,1)}
		n - s_{\mathbb{R}}^{(i)} > (8b+4u)R
	\end{equation}
	holds for $i=1,2$, where $s_{\mathbb{R}}^{(i)}$ are as defined in~\eqref{eq.defn_s_1} and~\eqref{eq.defn_s_2}. Then there exists some $\delta > 0$ depending at most on $\bm{F}$, $R$, $n$, $b$ and $u$ such that we have
	\begin{equation*}
		N(P_1,P_2) = \sigma P_1^{n-2R} P_2^{n-R} + O(P_1^{n-2R} P_2^{n-R} \min\{ P_1, P_2\}^{-\delta})
	\end{equation*}
 where $\sigma > 0$ if the system $\bm{F}(\bm{x},\bm{y}) = \bm{0}$ has a smooth $p$-adic zero for all primes $p$ and a smooth real zero in $\mathcal{B}_1 \times \mathcal{B}_2$.
	
	If we assume that $\mathbb{V}(F_1, \hdots, F_R) \subset \mathbb{P}^{n-1}_{\mathbb{Q}} \times \mathbb{P}_{\mathbb{Q}}^{n-1}$ is smooth, then the same conclusions hold if we assume
	\begin{equation*}
		n > (16b+8u+1)R
	\end{equation*}
	instead of \eqref{eq.assumption_n_i_(2,1)}.
\end{proposition}

For $r = 1, \hdots, R$ we can write each form $F_r(\bm{x}, \bm{y})$ as
\begin{equation*}
	F_r(\bm{x},\bm{y}) = \sum_{i,j,k} F^{(r)}_{ijk} x_i x_j y_k,
\end{equation*}
where the coefficients $F^{(r)}_{ijk}$ are symmetric in $i$ and $j$. In particular, for any $r = 1, \hdots, R$ we have an $n \times n$ matrix given by $H_r(\bm{y}) = (\sum_k F^{(r)}_{ijk} y_k)_{ij}$ whose entries are linear homogeneous polynomials in $\bm{y}$. We may thus also write each equation in the form
\[
F_r(\bm{x},\bm{y}) = \bm{x}^T H_r(\bm{y}) \bm{x}.
\]
The strategy of the proof of Proposition~\ref{thm.2,1} is the same as in the bilinear case, however this time more techincal arguments are required. We need to obtain a good upper bound for the counting functions $N_i^{\mathrm{aux}}(\bm{\beta}; B)$ so that we can apply Theorem~\ref{thm.n_aux_imply_result}. For $\bm{\beta} \in \mathbb{R}^R$ we consider $\bm{\beta} \cdot \bm{F}$, which we can rewrite in our case as
\begin{equation*}
	\bm{\beta} \cdot \bm{F}(\bm{x}, \bm{y}) = \bm{x}^T H_{\bm{\beta}}(\bm{y}) \bm{x}
\end{equation*}
where $H_{\bm{\beta}}(\bm{y}) = \sum_{i=1}^R \beta_i H_i(\bm{y})$ is a symmetric $n \times n$ matrix whose entries are linear and homogeneous in $\bm{y}$. The associated multilinear form $\Gamma_{\bm{\beta} \cdot \bm{F}}(\bm{x}^{(1)}, \bm{x}^{(2)}, \bm{y})$ is thus given by
\begin{equation*}
	\Gamma_{\bm{\beta} \cdot \bm{F}}(\bm{x}^{(1)}, \bm{x}^{(2)}, \bm{y}) =  2 \left(\bm{x}^{(1)}\right)^T H_{\bm{\beta}}(\bm{y}) \bm{x}^{(2)}.
\end{equation*} 
Recall $N_1^{\mathrm{aux}}(\bm{\beta}, B)$ counts integral tuples $\bm{x}, \bm{y} \in \Z^n$ satisfying $\nnorm{\bm{x}}_\infty, \nnorm{\bm{y}}_\infty \leq B$ and 
\begin{equation*}
	%\nnorm{\bm{\Gamma}_{\bm{\beta} \cdot \bm{F}}(\bm{x}, \bm{y})}_\infty = 
	\nnorm{\left(\Gamma_{\bm{\beta} \cdot \bm{F}}(\bm{x}, \bm{e}_1, \bm{y}), \hdots, \Gamma_{\bm{\beta} \cdot \bm{F}}(\bm{x}, \bm{e}_n, \bm{y}) \right)^T}_\infty = 2\nnorm{ H_{\bm{\beta}}( \bm{y}) \bm{x}}_\infty \leq \nnorm{\bm{\beta} \cdot \bm{F}}_\infty B.
\end{equation*}
Now $N_2^{\mathrm{aux}}(\bm{\beta}, B)$ counts integral tuples $\bm{x}^{(1)}$, $\bm{x}^{(2)}$ with $\nnorm{\bm{x}^{(1)}}_\infty, \nnorm{\bm{x}^{(2)}}_\infty \leq B$ and 
\begin{align*}
	%\nnorm{\bm{\Gamma}_{\bm{\beta} \cdot \bm{F}}(\bm{x}^{(1)}, \bm{x}^{(2)})}_\infty &= 
	\nnorm{\left(\Gamma_{\bm{\beta} \cdot \bm{F}}(\bm{x}^{(1)}, \bm{x}^{(2)}, \bm{e}_1), \hdots, \Gamma_{\bm{\beta} \cdot \bm{F}}(\bm{x}^{(1)}, \bm{x}^{(2)}, \bm{e}_n) \right)^T}_\infty 
	\leq \nnorm{\bm{\beta} \cdot \bm{F}}_\infty B.
\end{align*}
We may rewrite this as saying that
\[
\nnorm{\bm{x}^{(1)} H_{\bm{\beta}}(\bm{e}_\ell) \bm{x}^{(2)} } \leq \nnorm{\bm{\beta} \cdot \bm{F}}_\infty B
\]
is satisfied for $\ell = 1, \hdots, n$.
As in the proof of Theorem~\ref{thm.bilinear} using Proposition \ref{prop.auxiliary_ineq_from_counting_function} and Proposition \ref{prop.main_prop} we find that for the proof of Theorem~\ref{thm.2,1} it is enough to show that there exists a positive constant $C_0$ such that for all $B \geq 1$  and all $\bm{\beta} \in \mathbb{R}^r \setminus \{ 0 \}$ we have
\begin{equation*} \label{eq.aux_condition_2-1}
	N_i^{\mathrm{aux}}(\bm{\beta}; B) \leq C_0 B^{2n - 4 \mathscr{C}}
\end{equation*}
for $i=1,2$, where $\mathscr{C} > (2b+u)R$. 
The remainder of this section establishes these upper bounds.
\subsection{The first auxiliary counting function}
This is the easier case and the problem of finding a suitable upper bound for $N_1^{\mathrm{aux}}(\bm{\beta}; B)$ is essentially handled in~\cite{myerson_cubic}.

\begin{lemma}[Corollary 5.2 of \cite{myerson_cubic}] \label{lem.aux_small_or_bilinear_big}
	Let $H_{\bm{\beta}}(\bm{y})$ and $N_1^{\mathrm{aux}}(\bm{\beta}; B)$ be as above. Let $B,C \geq 1$, let $\bm{\beta} \in \mathbb{R}^R \setminus \{ 0 \}$ and let $\sigma \in \{0, \hdots, n-1 \}$. Then we either obtain the bound
	\begin{equation*} \label{eq.aux_counting_2-1_bounded}
		N_1^{\mathrm{aux}}(\bm{\beta}; B) \ll_{C,n} B^{n+\sigma} (\log B)^n
	\end{equation*}
	or there exist non-trivial linear subspaces $U,V \subseteq \mathbb{R}^n$ with $\dim U + \dim V = n+\sigma+1$ such that for all $\bm{v} \in V$ and $\bm{u}_1, \bm{u}_2 \in U$ we have
	\begin{equation*} \label{eq.bilinear_form_very_singular}
		\frac{\norm{\bm{u}_1^T H_{\bm{\beta}}(\bm{v}) \bm{u}_2 }}{\nnorm{\bm{\beta} \cdot \bm{F}}_\infty} \ll_n C^{-1} \nnorm{\bm{u}_1}_\infty \nnorm{\bm{v}}_\infty \nnorm{\bm{u}_2}_\infty. 
	\end{equation*}
\end{lemma}

	Recall the quantity
	\begin{equation*}
		s^{(1)}_{\mathbb{R}} \coloneqq 1 + \max_{\bm{\beta} \in \mathbb{R}^R \setminus \{ 0 \} } \dim \mathbb{V}(\bm{x}^T H_{\bm{\beta}} (\bm{e}_\ell) \bm{x})_{\ell = 1, \hdots, n_2},
	\end{equation*} 
	where we regard $\mathbb{V}(\bm{x}^T H_{\bm{\beta}} (\bm{e}_\ell) \bm{x})_{\ell = 1, \hdots, n_2} \subset \mathbb{P}_{\mathbb{C}}^{n_1-1}$ as a projective variety. Note that for this definition we do not necessarily require $n_1 = n_2$.

\begin{proposition} \label{prop.aux_counting_2-1_small}
	Let $\varepsilon > 0$. For all $B \geq 1$, $\bm{\beta} \in \mathbb{R}^R \setminus \{ 0 \}$ we have 
	\begin{equation} \label{eq.estimate_for_naux}
		N_1^{\mathrm{aux}}(\bm{\beta}; B) \ll_\varepsilon B^{n+s^{(1)}_{\mathbb{R}} + \varepsilon}.
	\end{equation}
\end{proposition}

\begin{proof}
	Assume for a contradiction that the estimate in \eqref{eq.estimate_for_naux} does not hold. In this case Lemma \ref{lem.aux_small_or_bilinear_big} gives that for each $N \in \mathbb{N}$ there exist $\bm{\beta}_N \in \mathbb{R}^R$  and there are non-trivial linear subspaces $U_N, V_N \subseteq \mathbb{R}^n$ with $\dim U_N + \dim V_N = n+s^{(1)}_{\mathbb{R}}+1$ such that for all $\bm{v} \in V_N$ and $\bm{u}_1, \bm{u}_2 \in U_N$ we have
	\begin{equation*}
		\frac{\norm{\bm{u}_1^T H_{\bm{\beta}_N}(\bm{v}) \bm{u}_2 }}{\nnorm{\bm{\beta}_N \cdot \bm{F}}_\infty} \ll_n N^{-1} \nnorm{\bm{u}_1}_\infty \nnorm{\bm{v}}_\infty \nnorm{\bm{u}_2}_\infty. 
	\end{equation*}
If we change $\bm{\beta}_N$ by a scalar then $2\frac{\norm{ H_{\bm{\beta}_N}(\bm{y})}}{\nnorm{\bm{\beta}_N \cdot \bm{F}}_\infty}$ remains unchanged for any $\bm{y} \in \mathbb{R}^n$. Therefore we may without loss of generality assume $\nnorm{\bm{\beta}_N}_\infty = 1$. Thus there exists a convergent subsequence of $(\bm{\beta}_N)$ whose limit we will denote by $\bm{\beta}$. Hence we find subspaces $U,V \subseteq \mathbb{R}^n$ with $\dim U+ \dim V = n+s^{(1)}_{\mathbb{R}}+1$ such that for all $\bm{v} \in V$ and $\bm{u}_1, \bm{u}_2 \in U$ we have
	\begin{equation*} 
		\bm{u}_1^T H_{\bm{\beta}}(\bm{v}) \bm{u}_2 =0. 
	\end{equation*}
	Let $k$ denote the nonnegative integer such that
	\begin{equation*}
		\dim V = n-k, \quad \text{and} \quad \dim U = s^{(1)}_{\mathbb{R}} + k +1
	\end{equation*}
	holds.
	Consider now a basis $\bm{v}_{k+1}, \hdots, \bm{v}_n$ of $V$ that we extend to a basis $\bm{v}_{1}, \hdots, \bm{v}_n$ of $\mathbb{R}^n$. Write also $[U] \subseteq \mathbb{P}_{\mathbb{C}}^{n-1}$ for the projectivisation of $U$. Define $W \subseteq [U]$ to be the projective variety defined by the equations
	\begin{equation*} \label{eq.proj_var_equation}
		\bm{u}^T H_{\bm{\beta}}(\bm{v}_i) \bm{u} = 0, \quad \text{for $i = 1, \hdots, k$}
	\end{equation*} 
	We find $\dim W \geq \dim [U] - k = s^{(1)}_{\mathbb{R}}$. Since $W \subseteq [U]$ and by the definition of $W$, noting that the entries of $H_{\bm{\beta}}(\bm{y})$ are linear in $\bm{y}$ we get that if $\bm{u} \in W$ then
	\begin{equation*}
		\bm{u}^T H_{\bm{\beta}}(\bm{y}) \bm{u} = 0 \quad \text{for all $\bm{y} \in \mathbb{R}^n$}.
	\end{equation*}
	In particular it follows that $W \subseteq \mathbb{V}(\bm{x}^T H_{\bm{\beta}} (\bm{e}_\ell) \bm{x})_{\ell = 1, \hdots, n} \subset \mathbb{P}_{\mathbb{C}}^{n-1}$ and thus
	\begin{equation*}
		s^{(1)}_{\mathbb{R}} -1 \geq \dim W \geq s^{(1)}_{\mathbb{R}},
	\end{equation*}
	which is clearly a contradiction.
\end{proof}

Now that we found an upper bound in terms of the geometry of $\mathbb{V}(\bm{F})$ the next Lemma shows that if $\bm{F}$ defines a non-singular variety then $s_\R^{(1)}$ is not too large. For the next Lemma we will not assume $n_1 = n_2$ as we will require it later in the slightly more general context when this assumption is not necessarily satisfied.
\begin{lemma} \label{lem.sigma_2-1_small}
	Let $s^{(1)}_{\mathbb{R}}$ be defined as above and assume that $\bm{F}$ is a system of bihomogenous equations of bidegree $(2,1)$ that defines a smooth complete intersection $ \mathbb{V}(\bm{F}) \subset \mathbb{P}^{n_1-1}_{\mathbb{C}} \times \mathbb{P}^{n_2-1}_{\mathbb{C}}$. Then 
	\begin{equation*}
		s^{(1)}_{\mathbb{R}} \leq \max\{ 0,R + n_1-n_2\}.
	\end{equation*}
\end{lemma}
\begin{proof}
Consider $\bm{\beta} \in \mathbb{R}^R \setminus \{ 0 \}$ such that $\dim \mathbb{V}(\bm{x}^T H_{\bm{\beta}} (\bm{e}_\ell) \bm{x})_{\ell = 1, \hdots, n_2} = s^{(1)}_{\mathbb{R}} -1$. In the case when $\mathbb{V}(\bm{x}^T H_{\bm{\beta}} (\bm{e}_\ell) \bm{x})_{\ell = 1, \hdots, n_2} = \emptyset$ then the statement in the lemma is trivially true. Hence we may assume that this is not the case. The singular locus of $\mathbb{V}(\bm{\beta} \cdot \bm{F}) \subseteq \mathbb{P}_{\mathbb{C}}^{n_1-1} \times \mathbb{P}_{\mathbb{C}}^{n_2-1}$ is given by
	\begin{equation*}
		\mathrm{Sing} \mathbb{V} (\bm{\beta} \cdot \bm{F} ) = \left( \mathbb{V}(\bm{x}^T H_{\bm{\beta}} (\bm{e}_\ell) \bm{x})_{\ell = 1, \hdots, n_2} \times \mathbb{P}_{\mathbb{C}}^{n_2-1} \right) \cap \mathbb{V}( H_{\bm{\beta}} (\bm{y}) \bm{x}).
	\end{equation*}
	From Lemma~\ref{lem.sing_bf_small} we obtain 
	\begin{equation*}
		\dim \mathrm{Sing} \mathbb{V}(\bm{\beta} \cdot \bm{F} ) \leq R-2.
	\end{equation*}
	Further, since $\mathbb{V}( H_{\bm{\beta}} (\bm{y}) \bm{x})$ is a system of $n_1$ bilinear equations, Lemma~\ref{lem.geometry_intersections} gives
	\[
	\dim \mathrm{Sing} \mathbb{V}(\bm{\beta} \cdot \bm{F} ) \geq s_{\mathbb{R}}^{(1)} - 1 + n_2-1 -n_1.
	\]
	Combining the previous two inequalities yields
	\[
	s_{\mathbb{R}}^{(1)} \leq R + n_1-n_2,
	\]
	as desired. 
\end{proof}
We remark here that the proof of Lemma~\ref{lem.sigma_2-1_small} shows that if $\mathbb{V}(\bm{F})$ defines a smooth complete intersection and if $s^{(1)}_{\mathbb{R}} >0$ then $n_2 < n_1+R$.

\subsection{The second auxiliary counting function}

Define $\widetilde{H}_{\bm{\beta}}(\bm{x}^{(1)})$ to be the $n \times n$ matrix with the rows given by $(\bm{x}^{(1)})^T H_{\bm{\beta}}(\bm{e}_\ell)/\nnorm{\bm{\beta} \cdot \bm{F}}_\infty$ for $\ell = 1, \hdots, n$. Using this notation $N_2^{\mathrm{aux}}(\bm{\beta}, B)$ counts the number of integer tuples $\bm{x}^{(1)}$, $\bm{x}^{(2)}$ such that $\nnorm{\bm{x}^{(1)}}_\infty, \nnorm{\bm{x}^{(2)}}_\infty \leq B$ and
\begin{equation*}
	\nnorm{\widetilde{H}_{\bm{\beta}}(\bm{x}^{(1)}) \bm{x}^{(2)}}_\infty \leq B,
\end{equation*}
is satisfied.
 The entries of $\widetilde{H}_{\bm{\beta}}(\bm{x}^{(1)})$ are homogeneous linear polynomials in $\bm{x}^{(1)}$ whose coefficients do not exceed absolute value $1$. 

 Let $A$ be a real $m \times n$ matrix. Then $A^TA$ is a symmetric and positive definite  $n \times n$ matrix, with eigenvalues $\lambda_1^2, \hdots, \lambda_n^2$. The nonnegative real numbers $\{ \lambda_i \}$ are the \emph{singular values} of $A$.
 
 \noindent \textbf{Notation.} 
 Given a matrix $M = (m_{ij})$ we define $\nnorm{M}_\infty \coloneqq \max_{i,j} |m_{ij}|$.
 For simplicity we will from now on write $\bm{x}$ instead of $\bm{x}^{(1)}$ and $\bm{y}$ instead of $\bm{x}^{(2)}$. For $\bm{x} \in \mathbb{R}^n$ let $\lambda_{\bm{\beta}, 1}(\bm{x}), \hdots, \lambda_{\bm{\beta},n}(\bm{x})$ denote the singular values of the real $n \times n$ matrix $\widetilde{H}_{\bm{\beta}}(\bm{x})$ in descending order, counted with multiplicity. Note that $\lambda_{\bm{\beta}, i}(\bm{x})$ are real and nonnegative. Also note
 \begin{equation*}
 	\lambda_{\bm{\beta},1}^2(\bm{x}) \leq n \nnorm{\widetilde{H}_{\bm{\beta}}(\bm{x})^T \widetilde{H}_{\bm{\beta}}(\bm{x})}_\infty \leq n^2 \nnorm{\widetilde{H}_{\bm{\beta}}(\bm{x})}_\infty^2 \leq n^4 \nnorm{\bm{x}}_\infty^2.
 \end{equation*}
 Taking square roots we find the following useful estimates
 \begin{equation} \label{eq.useful_estimate_sing_values}
 	\lambda_{\bm{\beta},1} (\bm{x})  \leq n \nnorm{\widetilde{H}_{\bm{\beta}}(\bm{x})}_\infty \leq n^2 \nnorm{\bm{x}}_\infty
 \end{equation}
 Let $i \in \{1, \hdots, n\}$ and write $\bm{D}^{(\bm{\beta}, i)}(\bm{x})$ for the vector with $\binom{n}{i}^2$ entries being the $i \times i$ minors of $\widetilde{H}_{\bm{\beta}}(\bm{x})$. Note that the entries are homogeneous polynomials in $\bm{x}$ of degree $i$.
 
 Finally write $J_{\bm{D}^{(\bm{\beta}, i)}}(\bm{x})$ for the Jacobian matrix of $\bm{D}^{(\bm{\beta}, i)}(\bm{x})$. That is, $J_{\bm{D}^{(\bm{\beta}, i)}}(\bm{x})$ is the $\binom{n}{i}^2 \times n$ matrix given by
 \begin{equation*}
 	(J_{\bm{D}^{(\bm{\beta}, i)}}(\bm{x}))_{jk} =  \frac{\partial D_j^{(\bm{\beta}, i)}}{\partial x_k}.
 \end{equation*}
 
 \begin{definition} \label{def.k_k}
 	Let $k \in \{0, \hdots, n\}$ and let $E_1, \hdots, E_{k+1} \in \mathbb{R}$ be such that $E_1 \geq \hdots \geq E_{k+1} \geq 1$ holds. We define $K_k(E_1, \hdots, E_{k+1}) \subseteq \mathbb{R}^n$ to be the set containing $\bm{x} \in \mathbb{R}^n$ such that the following three conditions are satisfied:
 	\begin{enumerate}[(i)]
 		\item \label{cond.K_k_1} $\nnorm{\bm{x}}_\infty \leq B$,
 		\item \label{cond.K_k_2} $\frac{1}{2}E_i < \lambda_{\bm{\beta},i}(\bm{x}) \leq E_i$ if $1 \leq i \leq k$, and
 		\item \label{cond.K_k_3} $\lambda_{\bm{\beta},i}(\bm{x}) \leq E_{k+1}$ if $k +1 \leq i \leq n$.
 	\end{enumerate}
 \end{definition}
 
\begin{lemma} \label{lem.fixed_matrix_eigenvalues_counting}
	Let $\widetilde{H}$  be an $n \times n $ matrix with real entries, and denote its singular values in descending order by $\lambda_1, \hdots, \lambda_n$. Let $C,B \geq 1$ and assume $\lambda_1 \leq CB$. Write $N_{\widetilde{H}}(B)$ for the number of integral vectors $\bm{y} \in \mathbb{Z}^n$ such that
	\begin{equation*}
		\nnorm{\bm{y}}_\infty \leq B, \quad \text{and} \quad \nnorm{\widetilde{H}\bm{y}}_\infty \leq B
	\end{equation*}
	holds. Then
	\begin{equation*}
		N_{\widetilde{H}}(B) \ll_{C,n} \min_{1 \leq i \leq n} \frac{B^n}{1+ \lambda_1 \cdots \lambda_i}.
	\end{equation*}
\end{lemma} 
\begin{proof}
	Consider the ellipsoid
	\begin{equation*}
		\mathcal{E} \coloneqq \{ \bm{t} \in \mathbb{R}^n \colon \bm{t}^T \widetilde{H}^T \widetilde{H} \bm{t} \leq nB^2 \}.
	\end{equation*}
	Note that any $\bm{y} \in \mathbb{Z}^n$ counted by $N_{\widetilde{H}}(B)$ is contained in $\mathcal{E}\cap [-B,B]^n$. Now recall that $\widetilde{H}^T \widetilde{H}$ is a symmetric matrix with eigenvalues $\lambda_1^2, \hdots, \lambda_n^2$. Therefore the principal radii of the ellipsoid $\mathcal{E}$ are given by $\lambda_i^{-1} \sqrt{n} B$. Hence we find
	\begin{equation} \label{eq.N_H(B)_first_estimate}
		N_{\widetilde{H}}(B) \ll_n \prod_{i=1}^n \min \{1+ \lambda_i^{-1} \sqrt{n}B, B\}
	\end{equation}
	By assumption we have $\lambda_i \leq CB$ and so the quantity on the right hand side of \eqref{eq.N_H(B)_first_estimate} is bounded above by
	\begin{equation*}
		\prod_{i=1}^n \min \{2C \lambda_i^{-1} \sqrt{n} B, B\},
	\end{equation*}
	and thus
	\begin{equation*}
		N_{\widetilde{H}}(B) \ll_{C,n} B^n \prod_{i=1}^n \min \{ \lambda_i^{-1}, 1\}.
	\end{equation*}
	Since $\lambda_1 \geq \cdots \geq \lambda_n$ the result now follows.
\end{proof}

\begin{lemma} \label{lem.N_2_aux_in_terms_K_k}
	Given $B \geq 1$ one of the following three possibilities must be true. Either we have
		 \begin{equation} \label{eq.alt_1}
		 \frac{N_2^{\mathrm{aux}}(\bm{\beta}, B)}{B^n (\log B)^n} \ll_n \# (\mathbb{Z}^n \cap K_0(1)),
		 \end{equation}
		 or  there exist nonnegative integers $e_1, \hdots, e_k$ for some $k \in \{1, \hdots, n-1\}$ such that $\log B \gg_n e_1 \geq \hdots \geq e_k$ and
		 \begin{equation} \label{eq.alt_2}
		 \frac{2^{e_1 + \cdots + e_k} N_2^{\mathrm{aux}}(\bm{\beta}, B)}{B^n (\log B)^n} \ll_n \# (\mathbb{Z}^n \cap K_k (2^{e_1}, \hdots, 2^{e_k}, 1) ),
		 \end{equation}
		or there exist nonnegative integers $e_1, \hdots, e_n$ such that $\log B \gg_n e_1 \geq \hdots \geq e_n$ and
		 \begin{equation} \label{eq.alt_3}
		 \frac{2^{e_1 + \cdots + e_n} N_2^{\mathrm{aux}}(\bm{\beta}, B)}{B^n (\log B)^n} \ll_n \# (\mathbb{Z}^n \cap K_{n-1} (2^{e_1}, \hdots, 2^{e_n}) ).
		 \end{equation}
\end{lemma}
\begin{proof}
	If $k=n$ then condition \eqref{cond.K_k_3} in Definition~\ref{def.k_k} is always trivially satisfied and thus
	\begin{equation*}
		K_n(2^{e_1}, \hdots, 2^{e_n}, 1) \subseteq K_{n-1} (2^{e_1}, \hdots, 2^{e_n}).
	\end{equation*}
	In particular, \eqref{eq.alt_3} follows from \eqref{eq.alt_2} with $k=n$. We are left showing that either \eqref{eq.alt_1} holds or  there exist nonnegative integers $e_1, \hdots, e_k$ for some $k \in \{1, \hdots, n\}$ such that $\log B \gg_n e_1 \geq \hdots \geq e_k$ and \eqref{eq.alt_2} holds. 
	
	Note that the box $[-B,B]^n$ is the disjoint union of $K_0(1)$ and $K_k(2^{e_1}, \hdots, 2^{e_k},1)$ where $k$ runs over $1, \hdots, n$ and $e_i$ run over  integers $ \log B \gg_n e_1 \geq \hdots \geq e_k$. Given $\bm{x} \in \Z^n$ write
	\begin{equation*}
		N_{\bm{x}}(B) = \# \left\{ \bm{y} \in \mathbb{Z}^n \colon \nnorm{\bm{y}}_\infty \leq B, \; \nnorm{\widetilde{H}_{\bm{\beta}}(\bm{x}) \bm{y}  }_\infty \leq B \right\}.
	\end{equation*} 
	 We thus obtain
	\begin{equation} \label{eq.N2aux_long_sum}
		N_2^{\mathrm{aux}}(\bm{\beta},B) = \sum_{\substack{\bm{x} \in \mathbb{Z}^n \\ \bm{x} \in K_0(1)}}N_{\bm{x}}(B) + \sum_{\substack{1 \leq k \leq n \\1 \leq e_k \leq \hdots \leq e_1 \\ e_1 \ll_n \log B}} \; \sum_{\substack{\bm{x} \in \mathbb{Z}^n \\ \bm{x} \in K_k(2^{e_1}, \hdots, 2^{e_k},1 )}} N_{\bm{x}}(B).
	\end{equation}
	Note that the number of terms of the outer sum of the second term of the right hand side of \eqref{eq.N2aux_long_sum} is bounded by $\ll_n (\log B)^n$. From this it follows that we either have
	\begin{equation} \label{eq.Nx_bound_alt_1}
		\sum_{\substack{\bm{x} \in \mathbb{Z}^n \\ \bm{x} \in K_0(1)}}N_{\bm{x}}(B) \gg_n \frac{N_2^{\mathrm{aux}}(\bm{\beta}, B) }{(\log B)^n}
	\end{equation} 
	or there exists an integer $k \in \{1, \hdots, n \}$ and integers $e_1 \geq \hdots \geq e_k \geq 1$ such that
	\begin{equation} \label{eq.Nx_bound_alt_2}
		\sum_{\substack{\bm{x} \in \mathbb{Z}^n \\ \bm{x} \in K_k(2^{e_1}, \hdots, 2^{e_k},1 )}} N_{\bm{x}}(B) \gg_n \frac{N_2^{\mathrm{aux}}(\bm{\beta}, B) }{(\log B)^n}.
	\end{equation}
	If \eqref{eq.Nx_bound_alt_1} holds then \eqref{eq.alt_1} follows from the trivial bound $N_{\bm{x}}(B) \ll_n B^n$. Assume now \eqref{eq.Nx_bound_alt_2} holds. From \eqref{eq.useful_estimate_sing_values}, for each $\bm{x}$ appearing in the sum of \eqref{eq.Nx_bound_alt_2}  we have the bound
	\begin{equation*}
		\lambda_{\bm{\beta},1}(\bm{x}) \leq n^2B.
	\end{equation*}
	Applying Lemma \ref{lem.fixed_matrix_eigenvalues_counting} with $C = n^2$ and $\widetilde{H} = \widetilde{H}_{\bm{\beta}}(\bm{x})$ we find 
	\begin{equation} \label{eq.estimating_Nx_last_time}
		N_{\bm{x}}(B) \ll_n \frac{B^n}{2^{e_1 + \hdots + e_k}}.
	\end{equation}
	Substituting \eqref{eq.estimating_Nx_last_time} into \eqref{eq.Nx_bound_alt_2} delivers \eqref{eq.alt_2}.
\end{proof}
We now recall two Lemmas from~\cite{myerson_cubic} that are conveniently stated in a form so that they apply to our setting.

\begin{lemma}[Lemma 3.2 in \cite{myerson_cubic}] \label{lem.minors_and_sing_values}
	Let $M$ be a real $m \times n$ matrix with singular values $\lambda_1, \hdots, \lambda_n$ listed with multiplicity in descending order. For $k \leq \min \{m,n \}$ denote by $\bm{D}^{(k)}$ the vector of $k \times k$ minors of $M$. Given such $k$, the following statements are true:
	\begin{enumerate}[(i)]
		\item \label{lem.singular_1} We have \begin{equation*} \label{eq.minors_asymp_sing_values}
			\nnorm{\bm{D}^{(k)}}_\infty \asymp \lambda_1 \cdots \lambda_k
		\end{equation*}
		\item There is a $k$-dimensional subspace $V \subset \mathbb{R}^n$, which can be taken to be a span of standard basis vectors $\bm{e}_i$, such that for all $\bm{v} \in V$ the following holds \begin{equation*} \label{eq.lower_bound_matrix_mult_sing_values}
			\nnorm{M \bm{v}}_\infty \gg_{m,n} \nnorm{\bm{v}}_\infty \lambda_k
		\end{equation*}
		\item Given $C \geq 1$ one of the following alternatives holds. Either there exists a $(n-k+1)$-dimensional subspace $X \subset \mathbb{R}^n$ such that
		\begin{equation*}
			\nnorm{M \bm{X}}_\infty \leq C^{-1} \nnorm{\bm{X}}_\infty \quad \text{for all $\bm{X} \in X$},	
		\end{equation*}
		or there is a $k$-dimensional subspace $V \subset \mathbb{R}^n$ spanned by standard basis vectors such that 
		\begin{equation*}
			\nnorm{M \bm{v}}_\infty \gg_{m,n} C^{-1} \nnorm{\bm{v}}_\infty \quad \text{for all $\bm{v} \in V$}.
		\end{equation*}
	\end{enumerate}
\end{lemma}

Next, we are interested in counting the number of integer tuples contained in the sets $K_k(E_1, \hdots, E_{k+1})$. The next Lemma is taken from \cite{myerson_cubic}.

\begin{lemma}[Lemma 4.1 in \cite{myerson_cubic}] \label{lem.counting_K_k}
	Let $B,C \geq 1$, $\sigma \in \{0, \hdots, n-1 \}$ and $k \in \{0, \hdots, n-\sigma-1 \}$. Assume further $CB \geq E_1 \geq \hdots \geq E_{k+1} \geq 1$. Then one of the following alternatives must hold.
	\begin{enumerate}
		\item[(I)\textsubscript{$k$}] \label{alt_I_k} We have the estimate  
		\begin{equation*}
			\# (\mathbb{Z}^n \cap K_k(E_1, \hdots, E_{k+1})) \ll_{C,n} B^\sigma (E_1 \cdots E_{k+1}) E_{k+1}^{n-\sigma-k-1}.
		\end{equation*} 
		\item[(II)\textsubscript{$k$}] \label{alt_II_k} For some integer $b \in \{1, \hdots, k\}$ there exists a $(\sigma + b +1)$-dimensional subspace $X \subset \mathbb{R}^n$ and there exists $\bm{x}^{(0)} \in K_b(E_1, \hdots, E_{b+1})$ such that $E_{b+1} < C^{-1} E_b$ and
		\begin{equation*}
			\nnorm{J_{\bm{D}^{(\bm{\beta}, b+1)}}(\bm{x}^{(0)}) \bm{X} }_\infty \leq C^{-1}\nnorm{ \bm{D}^{(\bm{\beta}, b)}(\bm{x}^{(0)})}_\infty \nnorm{\bm{X}}_\infty \quad \text{for all $\bm{X} \in X$.}
		\end{equation*}
		\item[(III)]  \label{alt_III} There exists a $(\sigma + 1)$-dimensional subspace $X \subset \mathbb{R}^n$ such that
		\begin{equation} \label{eq.H_tilde_small_on_subspace}
			\nnorm{\widetilde{H}_{\bm{\beta}}(\bm{X})}_\infty \leq C^{-1} \nnorm{\bm{X}}_\infty \quad \text{for all $\bm{X} \in X$.}
		\end{equation}
	\end{enumerate}
\end{lemma}
\begin{remark}
	In~\cite{myerson_cubic}, Lemma \ref{lem.counting_K_k} was stated for $\widetilde{H}_{\bm{\beta}}(\bm{x})$ being a symmetric matrix, and $\lambda_{\bm{\beta},i}(\bm{x})$ were taken to be the eigenvalues of $\widetilde{H}_{\bm{\beta}}(\bm{x})$ whose absolute values coincide with its singular values. However, an inspection of the proof shows that only the estimates in Lemma \ref{lem.minors_and_sing_values} as well as \eqref{eq.useful_estimate_sing_values} were used, which are valid for singular values as well as the (absolute values) of the eigenvalues. Therefore the proof remains valid in our setting.
\end{remark}
The next Lemma is similar to Lemma 5.1 in \cite{myerson_cubic}, however we need to account for the fact that $\widetilde{H}_{\bm{\beta}}(\bm{x})$ is not necessarily a symmetric matrix.
\begin{lemma} \label{lem.bilinear_small}
	Let $b \in \{1, \hdots, n-1\}$ and $\bm{x}^{(0)} \in \mathbb{R}^n$ be such that $\bm{D}^{(\bm{\beta},b)}(\bm{x}^{(0)}) \neq 0$. Then there exist subspaces $Y_1, Y_2 \subseteq \mathbb{R}^n$ with $\dim Y_1 = \dim Y_2 = n-b$ such that for all $\bm{Y}_1 \in Y_1$, $\bm{Y}_2 \in Y_2$ and $\bm{t} \in \mathbb{R}^n$ we have
	\begin{equation} \label{eq.estimate_bilinear_form}
		\bm{Y}_1^T \widetilde{H}_{\bm{\beta}}(\bm{t}) \bm{Y}_2 \ll_n \left( \frac{\nnorm{J_{\bm{D}^{(\bm{\beta},b+1)}}(\bm{x}^{(0)}) \bm{t}}_\infty}{\nnorm{\bm{D}^{(\bm{\beta},b)}(\bm{x}^{(0)})}_\infty} + \frac{\lambda_{\bm{\beta}, b+1}(\bm{x}^{(0)}) \cdot \nnorm{\bm{t}}_\infty }{\lambda_{\bm{\beta},b} (\bm{x}^{(0)})}  \right) \nnorm{\bm{Y}_1}_\infty \nnorm{\bm{Y}_2}_\infty
	\end{equation}
	where the implied constant only depends on $n$ but is otherwise independent from $\widetilde{H}_{\bm{\beta}}(\bm{t})$
\end{lemma}

\begin{proof}
%We will write down two sets of linearly independent vectors $\{\bm{y}^{(i)}_1\}_{i = 1, \hdots, n-b}$ and $\{ \bm{y}^{(i)}_2\}_{i = 1, \hdots, n-b}$ for which \eqref{eq.estimate_bilinear_form} holds. The span of these 
	Given $\bm{x} \in \mathbb{R}^n$ define $\bm{y}^{(1)}_1(\bm{x}), \hdots, \bm{y}^{(n-b)}_1(\bm{x})$ in the following way. The $j$-th entries are given by
	\begin{equation} \label{eq.def_y_1}
		(y_1^{(i)}(\bm{x}))_j= \begin{cases}
			(-1)^{n-b} \det \left( (\widetilde{H}_{\bm{\beta}}(\bm{x})_{k \ell} )_{\substack{k = n-b+1, \hdots, n \\ \ell = n-b+1, \hdots, n }} \right) \quad &\text{if $j=i$,} \\
			(-1)^j \det \left( (\widetilde{H}_{\bm{\beta}}(\bm{x})_{k \ell} )_{\substack{k = i, n-b+1, \hdots, n; \; k \neq j \\ \ell = n-b+1, \hdots, n }} \right) &\text{if $j > n-b$,} \\
			0 &\text{otherwise},
		\end{cases}
	\end{equation}
	where $k = i, n-b+1, \hdots, n; \; k \neq j$ denotes that we let the index $k$ run over the values $i,n-b+1, \hdots, n$ with $k = j$ omitted. Similarly we define $\bm{y}^{(1)}_2(\bm{x}), \hdots, \bm{y}^{(n-b)}_2(\bm{x})$ by
	\begin{equation*} \label{eq.def_y_2}
		(y_2^{(i)}(\bm{x}))_j = \begin{cases}
			(-1)^{n-b} \det \left( (\widetilde{H}_{\bm{\beta}}(\bm{x})_{k \ell} )_{\substack{k = n-b+1, \hdots, n \\ \ell = n-b+1, \hdots, n }} \right) \quad &\text{if $j=i$,} \\
			(-1)^j \det \left( (\widetilde{H}_{\bm{\beta}}(\bm{x})_{k \ell} )_{\substack{k = n-b+1, \hdots, n \\ \ell = i, n-b+1, \hdots, n; \; \ell \neq j }} \right) &\text{if $j > n-b$,} \\
			0 &\text{otherwise}.
		\end{cases}
	\end{equation*} \label{eq.determinants_laplace_y_1}
	Using the Laplace expansion of a determinant along columns and rows we thus obtain
	\begin{equation}
		(\bm{y}_1^{(i)}(\bm{x})^T  \widetilde{H}_{\bm{\beta}}(\bm{x}))_j = \begin{cases}
			(-1)^{n-b} \det \left( (\widetilde{H}_{\bm{\beta}}(\bm{x})_{k \ell} )_{\substack{k = i,n-b+1, \hdots, n \\ \ell =j, n-b+1, \hdots, n }} \right) \quad &\text{if $j \leq n-b$,} \\
			0 &\text{otherwise,}
		\end{cases}
	\end{equation}
	and 
	\begin{equation} \label{eq.determinants_laplace_y_2}
		(  \widetilde{H}_{\bm{\beta}}(\bm{x})\bm{y}_2^{(i)}(\bm{x}))_j = \begin{cases}
			(-1)^{n-b} \det \left( (\widetilde{H}_{\bm{\beta}}(\bm{x})_{k \ell} )_{\substack{k = j,n-b+1, \hdots, n \\ \ell =i, n-b+1, \hdots, n }} \right) \quad &\text{if $j \leq n-b$,} \\
			0 &\text{otherwise,}
		\end{cases}
	\end{equation}
	respectively. It follows from \eqref{eq.def_y_1} --- \eqref{eq.determinants_laplace_y_2} that there exist matrices $L_1^{(i)}$, $L_2^{(i)}$, $M_1^{(i)}$ and $M_2^{(i)}$ for $i = 1, \hdots, n-b$ with entries only in $\{0, \pm 1\}$ such that we obtain
	\begin{align}
		\bm{y}_1^{(i)}(\bm{x}) &= L_1^{(i)} \bm{D}^{(\bm{\beta}, b)}(\bm{x}), \\
		\bm{y}_2^{(i)}(\bm{x}) &= L_2^{(i)} \bm{D}^{(\bm{\beta}, b)}(\bm{x}), \label{eq.y_2_in_minors} \\
		(\bm{y}_1^{(i)}(\bm{x}))^T \widetilde{H}_{\bm{\beta}}(\bm{x}) &= [ M_1^{(i)}   \bm{D}^{(\bm{\beta}, b+1)}(\bm{x}) ]^T, \quad \text{and} \label{eq.transpose_y_1_H_tilde} \\
		 \widetilde{H}_{\bm{\beta}}(\bm{x}) \bm{y}_2^{(i)}(\bm{x}) &= M_2^{(i)}   \bm{D}^{(\bm{\beta}, b+1)}(\bm{x}). \label{eq.H_y_linear_span_minors}
	\end{align}
	Given $\bm{t} \in \mathbb{R}^n$ we write $\partial_{\bm{t}}$ for the directional derivative given by $\sum t_i \frac{\partial}{\partial x_i}$. Applying $\partial_{\bm{t}}$ to both sides of \eqref{eq.H_y_linear_span_minors} we obtain
	\begin{equation} \label{eq.taking_derivative}
		[\partial_{\bm{t}} \widetilde{H}_{\bm{\beta}}(\bm{x}) ] \bm{y}_2^{(i)}(\bm{x})  + \widetilde{H}_{\bm{\beta}}(\bm{x}) [\partial_{\bm{t}} \bm{y}_2^{(i)}(\bm{x})] = M_2^{(i)} [\partial_{\bm{t}}   \bm{D}^{(\bm{\beta}, b+1)}(\bm{x})]. 
	\end{equation}
	Now note
	\begin{equation} \label{eq.directional_deriv_easy_identity}
		\partial_{\bm{t}} \bm{D}^{(\bm{\beta}, b+1)}(\bm{x}) = J_{\bm{D}^{(\bm{\beta},b+1 )}}(\bm{x}) \bm{t}, \quad \text{and} \quad \partial_{\bm{t}} \widetilde{H}_{\bm{\beta}}(\bm{x}) = \widetilde{H}_{\bm{\beta}}(\bm{t}).
	\end{equation}
	Substituting \eqref{eq.directional_deriv_easy_identity} and \eqref{eq.y_2_in_minors} into \eqref{eq.taking_derivative} yields
	\begin{equation*}
		\widetilde{H}_{\bm{\beta}}(\bm{t}) \bm{y}_2^{(i)}(\bm{x}) = M_2^{(i)} J_{\bm{D}^{(\bm{\beta},b+1 )}}(\bm{x}) \bm{t} - \widetilde{H}_{\bm{\beta}}(\bm{x}) L_2^{(i)} \partial_{\bm{t}} \bm{D}^{(\bm{\beta}, b)}(\bm{x}).
	\end{equation*}
	If we premultiply this by $\bm{y}_1^{(j)}(\bm{x})^T$ and use \eqref{eq.transpose_y_1_H_tilde} then we obtain
	\begin{multline} \label{eq.bilinear_general_x}
		\bm{y}_1^{(j)}(\bm{x})^T \widetilde{H}_{\bm{\beta}}(\bm{t}) \bm{y}_2^{(i)}(\bm{x}) = \bm{y}_1^{(j)}(\bm{x})^T M_2^{(i)} J_{\bm{D}^{(\bm{\beta},b+1 )}}(\bm{x}) \bm{t} \\ - [ M_1^{(j)} \bm{D}^{(\bm{\beta}, b+1)}(\bm{x}) ]^T  [L_2^{(i)} \partial_{\bm{t}} \bm{D}^{(\bm{\beta}, b)}(\bm{x})].
	\end{multline}
	 Lemma~\ref{lem.minors_and_sing_values}~\eqref{lem.singular_1} yields the bounds
	 \begin{equation} \label{eq.bounds_minors_sing_values_1}
	 	\frac{\nnorm{\bm{D}^{(\bm{\beta},b+1)}(\bm{x})}_\infty}{\nnorm{\bm{D}^{(\bm{\beta},b)}(\bm{x})}_\infty} \ll_n \lambda_{\bm{\beta},b+1}(\bm{x}),
	 \end{equation}
	 and
	 \begin{equation} \label{eq.bounds_minors_sing_values_2}
	 	\frac{\nnorm{\partial_{\bm{t}} \bm{D}^{(\bm{\beta},b)}(\bm{x})}_\infty}{\nnorm{\bm{D}^{(\bm{\beta},b)}(\bm{x})}_\infty} \ll_n \frac{\nnorm{\bm{t}}_\infty}{\lambda_{\bm{\beta},b}(\bm{x})}.
	 \end{equation}
	 Now we specify $\bm{x} = \bm{x}^{(0)}$ so by assumption we have $\nnorm{\bm{D}^{(\bm{\beta},b)}(\bm{x}^{(0)})}_\infty > 0$. Thus define
	\begin{equation} \label{eq.Y_k_def}
		\bm{Y}_k^{(i)} = \frac{\bm{y}_k^{(i)}(\bm{x}^{(0)})}{\nnorm{\bm{D}^{(\bm{\beta},b)}(\bm{x}^{(0)})}_\infty}, \quad \text{for $i = 1, \hdots, n-b$ and $k = 1,2$.}
	\end{equation}
	Dividing~\eqref{eq.bilinear_general_x} by $1/\nnorm{\bm{D}^{(\bm{\beta},b)}(\bm{x}^{(0)})}_\infty^2$ and using~\eqref{eq.Y_k_def} as well as the bounds~\eqref{eq.bounds_minors_sing_values_1} and~\eqref{eq.bounds_minors_sing_values_2} gives
	\begin{equation*} \label{eq.blinear_specified_x_0}
		\norm{\bm{Y}_1^{(j)} \widetilde{H}_{\bm{\beta}}(\bm{t}) \bm{Y}_2^{(i)}} \ll_n \frac{\nnorm{J_{\bm{D}^{(\bm{\beta},b+1 )}}(\bm{x}^{(0)}) \bm{t}}_\infty}{\nnorm{\bm{D}^{(\bm{\beta},b)}(\bm{x}^{(0)})}_\infty} + \frac{\lambda_{\bm{\beta},b+1}(\bm{x}^{(0)}) \nnorm{\bm{t}}_\infty}{\lambda_{\bm{\beta},b}(\bm{x}^{(0)})}.
	\end{equation*}
	We claim now that we can take the subspaces $Y_k \subseteq \mathbb{R}^n$ to be defined as the span of $\bm{Y}_k^{(1)}, \hdots, \bm{Y}_k^{(n-b)}$ for $k = 1,2$ respectively, so that the Lemma holds. For this we need to show that~\eqref{eq.estimate_bilinear_form} holds, and also that $\dim Y_1 = \dim Y_2 = n-b$. Therefore it suffices to show the following claim: Given $\bm{\gamma} \in \mathbb{R}^{n-b}$ if we take $\bm{Y}_k = \sum \gamma_i Y_k^{(i)}$ then $\nnorm{\bm{\gamma}}_\infty \ll_n \nnorm{\bm{Y}_k}_\infty$, for $k = 1,2$ respectively.
	
	Assume that the $b \times b$ minor of $\widetilde{H}_{\bm{\beta}}(\bm{x}^{(0)})$ of largest absolute value lies in the bottom right corner of $\widetilde{H}_{\bm{\beta}}(\bm{x}^{(0)})$. In other words, we assume
	\begin{equation} \label{eq.assumption_that_is_wlog}
		\nnorm{\bm{D}^{(\bm{\beta},b)}(\bm{x}^{(0)})}_\infty = \norm{\det \left( (\widetilde{H}_{\bm{\beta}}(\bm{x}^{(0)})_{k \ell} )_{\substack{k = n-b+1, \hdots, n \\ \ell = n-b+1, \hdots, n }} \right)}.
	\end{equation}
	After permuting the rows and columns of $\widetilde{H}_{\bm{\beta}}(\bm{x}^{(0)})$ the identity \eqref{eq.assumption_that_is_wlog} will always be true. The vectors $\bm{Y}_k^{(i)}$ depend on minors of $\widetilde{H}_{\bm{\beta}}(\bm{x}^{(0)})$. Thus we can apply the same permutations to $\widetilde{H}_{\bm{\beta}}(\bm{x}^{(0)})$ that ensure that \eqref{eq.assumption_that_is_wlog} holds to the definition of these vectors. From this we see that we can always reduce the general case to the case where \eqref{eq.assumption_that_is_wlog} holds. 
	
	Now for $k= 1,2$ we define matrices 
	\begin{equation*}
		Q_k = \left( \bm{Y}_k^{(1)} \Big\vert \cdots \Big\vert \bm{Y}_k^{(n-b)} \Big\vert \bm{e}_{n-b+1} \Big\vert \cdots \Big\vert \bm{e}_{n} \right).
	\end{equation*}
	By the definition of $\bm{Y}_k^{(i)}$ we see that $Q_k$ must be of the following form
	\begin{equation*}
		Q_k = \begin{pmatrix}
			I_{n-b} & 0 \\
			\widetilde{Q}_k & I_{b}
		\end{pmatrix},
	\end{equation*}
	for some matrix $\widetilde{Q}_k$. In particular we find $\det Q_k = 1$ and so $\nnorm{Q_k^{-1}}_\infty \ll_n 1$. Given $\bm{Y}_k = \sum \gamma_i Y_k^{(i)}$ we thus find
	\begin{equation*}
		\nnorm{\bm{\gamma}}_\infty = \nnorm{Q_k^{-1} \bm{Y}_k}_\infty \ll_n \nnorm{\bm{Y}_k}_\infty,
	\end{equation*}
	and so the Lemma follows.
	\end{proof}
		The next Corollary is the main technical result from this section, which will allow us to deduce that either $N_2^{\mathrm{aux}}(\bm{\beta},B)$ is small or a suitable singular locus is large.

\begin{corollary} \label{cor.N_2_small_or_bilinear_vanishes}
	Let $B, C \geq 1$ and let $\sigma \in \{ 0, \hdots, n-1\}$. Then one of the following alternatives is true. Either we have the bound
	\begin{equation} \label{eq.corollary_alt_1}
		N_2^{\mathrm{aux}}(\bm{\beta},B) \ll_{C,n} B^{n+\sigma} (\log B)^n,
	\end{equation}
	or there exist subspaces $X, Y_1, Y_2 \subseteq \mathbb{R}^n$ with $\dim X + \dim Y_1 = \dim X + \dim Y_2 = n+\sigma +1$, such that
	\begin{equation} \label{eq.corollary_alt_2}
		\norm{\bm{Y}_1^T \widetilde{H}_{\bm{\beta}}(\bm{X}) \bm{Y}_2} \ll_n C^{-1} \nnorm{\bm{Y}_1}_\infty \nnorm{\bm{X}}_\infty \nnorm{\bm{Y}_2}_\infty 
	\end{equation}
	holds for all $\bm{X} \in X, \bm{Y}_1 \in Y_1, \bm{Y}_2 \in Y_2$.
\end{corollary}
\begin{proof}
	Let $k \in \{0, \hdots, n-\sigma-1\}$ and $E_1, \hdots, E_{k+1} \in \mathbb{R}$ be such that 
	\begin{equation*}
		CB \geq E_1 \geq \hdots \geq E_{k+1} \geq 1.
	\end{equation*}
	We know that one of the alternatives (I)\textsubscript{$k$}, (II)\textsubscript{$k$} or (III) in Lemma \ref{lem.counting_K_k} holds. Assume first that~(I)\textsubscript{$k$} always holds so that the estimate
	\begin{equation} \label{eq.K_k_bound}
			\# (\mathbb{Z}^n \cap K_k(E_1, \hdots, E_{k+1})) \ll_{C,n} B^\sigma (E_1 \cdots E_{k+1}) E_{k+1}^{n-\sigma-k-1}.
	\end{equation} 
	holds for every $k \in \{0, \hdots, n-\sigma-1\}$ and $E_1, \hdots, E_{k+1} \in \mathbb{R}$ such that $CB \geq E_1 \geq \hdots \geq E_{k+1} \geq 1$. From Lemma~\ref{lem.N_2_aux_in_terms_K_k} we find that either we have
		 \begin{equation} \label{eq.alt_1_in_proof}
		 \frac{N_2^{\mathrm{aux}}(\bm{\beta}, B)}{B^n (\log B)^n} \ll_n \# (\mathbb{Z}^n \cap K_0(1)),
		 \end{equation}
		 or  there exist nonnegative integers $e_1, \hdots, e_k$ for some $k \in \{1, \hdots, n-1\}$ such that $\log B \gg_n e_1 \geq \hdots \geq e_k$ and
		 \begin{equation} \label{eq.alt_2_in_proof}
		 \frac{2^{e_1 + \cdots + e_k} N_2^{\mathrm{aux}}(\bm{\beta}, B)}{B^n (\log B)^n} \ll_n \# (\mathbb{Z}^n \cap K_k (2^{e_1}, \hdots, 2^{e_k}, 1) ),
		 \end{equation}
		or there exist nonnegative integers $e_1, \hdots, e_n$ such that $\log B \gg_n e_1 \geq \hdots \geq e_n$ and
		 \begin{equation} \label{eq.alt_3_in_proof}
		 \frac{2^{e_1 + \cdots + e_n} N_2^{\mathrm{aux}}(\bm{\beta}, B)}{B^n (\log B)^n} \ll_n \# (\mathbb{Z}^n \cap K_{n-1} (2^{e_1}, \hdots, 2^{e_n}) ).
		 \end{equation}
		 We may take $C$ to be large enough depending on $n$ such that $CB \geq 2^{e_1}$ is satisfied. Then upon sbustituting the bound~\eqref{eq.K_k_bound} into either of~\eqref{eq.alt_1_in_proof}, \eqref{eq.alt_2_in_proof} or~\eqref{eq.alt_3_in_proof} gives~\eqref{eq.corollary_alt_1}.
		 
		 If~(III) holds in Lemma~\ref{lem.counting_K_k} we can take $Y_1 = Y_2 = \mathbb{R}^n$ so that~\eqref{eq.corollary_alt_2} follows from~\eqref{eq.H_tilde_small_on_subspace}.
		 
		 Finally, assume there exist $k \in \{0, \hdots, n-\sigma-1\}$ and $E_1, \hdots, E_{k+1} \in \mathbb{R}$ with $CB \geq E_1 \geq \hdots \geq E_{k+1} \geq 1$ such that~(II)\textsubscript{$k$} in Lemma \ref{lem.counting_K_k} holds. Recall this means there exists some integer $b \in \{1, \hdots, k\}$, a $(\sigma + b +1)$-dimensional subspace $X \subset \mathbb{R}^n$ and $\bm{x}^{(0)} \in K_b(E_1, \hdots, E_{b+1})$ such that $E_{b+1} < C^{-1} E_b$ and
		\begin{equation} \label{eq.jacobian_bound}
			\nnorm{J_{\bm{D}^{(\bm{\beta}, b+1)}}(\bm{x}^{(0)}) \bm{X} }_\infty \leq C^{-1}\nnorm{ \bm{D}^{(\bm{\beta}, b)}(\bm{x}^{(0)})}_\infty \nnorm{\bm{X}}_\infty \quad \text{for all $\bm{X} \in X$.}
		\end{equation}
		As $\bm{x}^{(0)} \in K_b(E_1, \hdots, E_{b+1})$ we have $E_i/2 < \lambda_{\bm{\beta},i}(\bm{x}^{(0)}) \leq E_i$ for $i = 1, \hdots, k$ and $\lambda_{\bm{\beta},b+1}(\bm{x}^{(0)}) \leq E_{b+1}$.  This, together with the fact that $E_{b+1} < C^{-1} E_b$ implies
		\begin{equation} \label{eq.sing_value_bound}
			\lambda_{\bm{\beta},b+1}(\bm{x}^{(0)}) < 2C^{-1}\lambda_{\bm{\beta},b}(\bm{x}^{(0)}).
		\end{equation}
		Also we find  $\lambda_{\bm{\beta},b}(\bm{x}^{(0)}) \neq 0$, from which it follows from Lemma~\ref{lem.minors_and_sing_values}~\eqref{lem.singular_1} that $\bm{D}^{(\bm{\beta},b)}(\bm{x}^{(0)}) \neq 0$. Thus we may apply Lemma~\ref{lem.bilinear_small} to obtain spaces $Y_1,Y_2 \subseteq \mathbb{R}^n$ with $\dim Y_1 = \dim Y_2 = n-b$  such that the estimate \eqref{eq.estimate_bilinear_form} holds. Now taking $\bm{t} = \bm{X}$ in~\eqref{eq.estimate_bilinear_form} and using~\eqref{eq.jacobian_bound} and~\eqref{eq.sing_value_bound} then~\eqref{eq.corollary_alt_2} follows. Since $\dim X = \sigma +b+1$ we also have $\dim X + \dim Y_1 = \dim X + \dim Y_2 = n+\sigma +1$ as desired.
\end{proof}
Recall the definition of the quantity
\begin{equation*}
	s_{\mathbb{R}}^{(2)} \coloneqq \left\lfloor \frac{\max_{\bm{\beta} \in \mathbb{R}^R \setminus \{ 0\}} \dim \mathbb{V}(H_{\bm{\beta}}(\bm{y}) \bm{x})}{2} \right\rfloor +1, 
\end{equation*}
where $\lfloor x \rfloor$ denotes the largest integer $m$ such that $m \leq x$.  Although we have been assuming $n_1=n_2$ throughout the definition of this quantity remains valid if $n_1 \neq n_2$. Note that we have  $\mathbb{V}(H_{\bm{\beta}}(\bm{y}) \bm{x}) \subsetneq \mathbb{P}^{n_1-1}_{\mathbb{C}} \times \mathbb{P}^{n_2-1}_{\mathbb{C}}$ for all $\bm{\beta} \in \mathbb{R}^R \setminus \{0\}$. For if not, then the matrix $H_{\bm{\beta}}(\bm{y})$ is identically zero for some $\bm{\beta} \in \mathbb{R}^R \setminus \{0\}$ contradicting the fact that $\mathbb{V}(\bm{F})$ is a complete intersection. In particular this yields $s_{\mathbb{R}}^{(2)} \leq \frac{n_1+n_2}{2}-1$.

Before we prove the main result of this section we require another small Lemma.

\begin{lemma} \label{lem.H_tilde_or_normal_H}
	Let $\bm{\beta} \in \mathbb{R} \setminus \{0\}$. The system of equations
	\begin{equation*}
		\bm{y}^T \widetilde{H}_{\bm{\beta}}(\bm{e}_\ell) \bm{x} = 0, \; \text{for $\ell = 1, \hdots, n$}  \quad \text{and} \quad H_{\bm{\beta}}(\bm{y})\bm{x} = \bm{0}
	\end{equation*}
	define the same variety in $\mathbb{P}_{\mathbb{C}}^{n-1} \times \mathbb{P}_{\mathbb{C}}^{n-1}$.
\end{lemma}
\begin{proof}
	Recall that by definition we have
	\begin{equation*}
		\widetilde{H}_{\bm{\beta}}(\bm{z}) = \begin{pmatrix}
			\bm{z}^T H_{\bm{\beta}}(\bm{e}_1) \\
			\vdots \\
			\bm{z}^T H_{\bm{\beta}}(\bm{e}_n).
		\end{pmatrix}
	\end{equation*}
	For $\ell \in \{1, \hdots, n\}$ we get
	\begin{equation*}
		\bm{y}^T \widetilde{H}_{\bm{\beta}}(\bm{e}_\ell) \bm{x} = \bm{y}^T \begin{pmatrix}
			\bm{e}_\ell^T H_{\bm{\beta}}(\bm{e}_1) \bm{x} \\
			\vdots \\
			\bm{e}_\ell^T H_{\bm{\beta}}(\bm{e}_n) \bm{x}
		\end{pmatrix} = \sum_{i=1}^n y_i \bm{e}_\ell^T H_{\bm{\beta}}(\bm{e}_i) \bm{x} = \bm{e}_\ell^T H_{\bm{\beta}}(\bm{y}) \bm{x},
	\end{equation*}
	where the last line follows since the entries of $H_{\bm{\beta}}(\bm{y})$ are linear homogeneous in $\bm{y}$. The result is now immediate.
\end{proof}

\begin{proposition} \label{prop.N_2aux_small}
	Let $s_{\mathbb{R}}^{(2)}$ be defined as above and let $B \geq 1$. Then for all $\bm{\beta} \in \mathbb{R}^{R} \setminus \{0\}$ the following holds
	\begin{equation*} \label{eq.N_2_estimate_that_we_want}
				N_2^{\mathrm{aux}}(\bm{\beta},B) \ll_{n} B^{n+s_{\mathbb{R}}^{(2)}} (\log B)^n.
	\end{equation*}
\end{proposition}
\begin{proof}
	Suppose for a contradiction the result were false. Then for each positive integer $N$ there exists some $\bm{\beta}_N$ such that
	\begin{equation*}
				N_2^{\mathrm{aux}}(\bm{\beta}_N,B) \geq N B^{n+s_{\mathbb{R}}^{(2)}} (\log B)^n.
	\end{equation*}
	From Corollary~\ref{cor.N_2_small_or_bilinear_vanishes} it follows that there are linear subspaces $X^{(N)}, Y_1^{(N)}, Y_2^{(N)} \subset \mathbb{R}^n$ with 
	\begin{equation*}
		\dim X^{(N)} + \dim Y_i^{(N)} = n+ s_{\mathbb{R}}^{(2)} + 1, \quad i=1,2,
	\end{equation*}
	such that for all $\bm{X} \in X^{(N)}$, $\bm{Y}_i \in Y_i^{(N)}$ we get
	\begin{equation*}
		\norm{\bm{Y}_1^T \widetilde{H}_{\bm{\beta}_N}(\bm{X}) \bm{Y}_2} \leq N^{-1} \nnorm{\bm{Y}_1}_\infty \nnorm{\bm{X}}_\infty \nnorm{\bm{Y}_2}_\infty.
	\end{equation*}
	Note that $\widetilde{H}_{\bm{\beta}_N}(\bm{\beta})$ is unchanged when $\bm{\beta}_N$ is multiplied by a constant. Thus we may assume $\nnorm{\bm{\beta}_N}_\infty = 1$ and consider a converging subsequence of $\bm{\beta}_{N_r}$ converging to $\bm{\beta}$, say, as $N \rightarrow \infty$. This delivers subspaces $X,Y_1,Y_2 \subset \mathbb{R}^n$ with $\dim X + \dim Y_i = n + s_{\mathbb{R}}^{(2)} + 1$ for $i = 1,2$ such that 
	\begin{equation*}
		\bm{Y}_1^T \widetilde{H}_{\bm{\beta}}(\bm{X}) \bm{Y}_2 = 0 \quad \text{for all $\bm{X} \in X, \bm{Y}_1 \in Y_1, \bm{Y}_2 \in Y_2$.}
	\end{equation*}
	There exists some $b \in \{0, \hdots, n- s_{\mathbb{R}}^{(2)}-1 \}$ such that $\dim X = n-b$ and $\dim Y_i = s_{\mathbb{R}}^{(2)} + b+1$. Now let $\bm{x}^{(1)}, \hdots, \bm{x}^{(n)}$ be a basis for $\mathbb{R}^n$ such that $\bm{x}^{(b+1)}, \hdots, \bm{x}^{(n)}$ is a basis for $X$. Write $[Y_i] \subset \mathbb{P}_{\mathbb{C}}^{n-1}$ for the linear subspace of $\mathbb{P}_{\mathbb{C}}^{n-1}$ associated to $Y_i$ for $i = 1,2$.
	
	Define the biprojective variety $W \subset [Y_1] \times [Y_2]$ in the variables $(\bm{y}_1, \bm{y}_2)$ by 
	\begin{equation*}
		W = \mathbb{V}(\bm{y}_1 \widetilde{H}_{\bm{\beta}}(\bm{x}^{(i)}) \bm{y}_2 )_{i = 1, \hdots, b}.
	\end{equation*}
	Since the non-trivial equations defining $W$ have bidegree $(1,1)$ we can apply Corollary~\ref{cor.geometry_intersections} to find
	\begin{equation} \label{eq.dim_w_lower_bound}
		\dim W \geq \dim [Y_1] \times [Y_2] - b = 2 s_{\mathbb{R}}^{(2)} + b.
	\end{equation}
	Given $(\bm{y}_1, \bm{y}_2) \in W$ we have in particular $(\bm{y}_1, \bm{y}_2) \in [Y_1] \times [Y_2]$ and so 
	\begin{equation*}
		\bm{y}_1 \widetilde{H}_{\bm{\beta}}(\bm{x}^{(i)}) \bm{y}_2 = 0, \quad \text{for $i = b+1, \hdots, n$,}
	\end{equation*}
	and hence  $\bm{y}_1 \widetilde{H}_{\bm{\beta}}(\bm{z}) \bm{y}_2 = 0$ for all $\bm{z} \in \mathbb{R}^n$. From Lemma~\ref{lem.H_tilde_or_normal_H} we thus see $H_{\bm{\beta}}(\bm{y}_1) \bm{y}_2 = 0$ for all $(\bm{y}_1, \bm{y}_2) \in W$. Hence in particular
	\begin{equation*} \label{eq.dim_W_upper_bound}
		\dim W \leq \dim \mathbb{V}(H_{\bm{\beta}}(\bm{y})\bm{x}) \leq 2 s_{\mathbb{R}}^{(2)} -1,
	\end{equation*}
	where we regard $\mathbb{V}(H_{\bm{\beta}}(\bm{y})\bm{x})$ as a variety in $\mathbb{P}_{\mathbb{C}}^{n-1} \times \mathbb{P}_{\mathbb{C}}^{n-1}$ in the variables $(\bm{x}, \bm{y})$. This together with \eqref{eq.dim_w_lower_bound} implies $b \leq -1$, which is clearly a contradiction.
	\end{proof}
	
In the next Lemma we show that $s_{\mathbb{R}}^{(2)}$ is small if $\mathbb{V}(\bm{F})$ defines a smooth complete intersection. For this we no longer assume $n_1 = n_2$.

\begin{lemma} \label{lem.sigma_2_bounds}
		Let $s_{\mathbb{R}}^{(2)}$ be defined as above. If $\mathbb{V}(\bm{F})$ is a smooth complete intersection in $\mathbb{P}_{\mathbb{C}}^{n_1-1} \times \mathbb{P}_{\mathbb{C}}^{n_2-1}$ then we have the bound
		\begin{equation} \label{eq.sigma_2_bounds}
			\frac{n_2-1}{2} \leq s_{\mathbb{R}}^{(2)} \leq \frac{n_2+R}{2}.
		\end{equation}
	\end{lemma}
	\begin{proof}
		Let $\bm{\beta} \in \mathbb{R}^R \setminus \{ \bm{0} \}$ be such that
		\begin{equation*}
			s_{\mathbb{R}}^{(2)} = \left\lfloor \frac{\dim \mathbb{V}(H_{\bm{\beta}}(\bm{y} )\bm{x}) }{2} \right\rfloor +1.
		\end{equation*}
		Note that then
		\begin{equation} \label{eq.sigma_dimensions_bounds}
			2s_{\mathbb{R}}^{(2)} -2 \leq \dim \mathbb{V}(H_{\bm{\beta}}(\bm{y} )\bm{x}) \leq 2 s_{\mathbb{R}}^{(2)} -1.
		\end{equation}
		The variety $\mathbb{V}(H_{\bm{\beta}}(\bm{y} )\bm{x}) \subset \mathbb{P}_{\mathbb{C}}^{n_1-1} \times \mathbb{P}_{\mathbb{C}}^{n_2-1}$ is defined by $n_1$ bilinear polynomials. Using Corollary~\ref{cor.geometry_intersections} we thus find
		\begin{equation*}
			\dim \mathbb{V}(H_{\bm{\beta}}(\bm{y} )\bm{x}) \geq n_2-2
		\end{equation*}
		so the lower bound in~\eqref{eq.sigma_2_bounds} follows.
		We proceed by considering two cases. 
		
	\noindent \textbf{Case 1: $\mathbb{V}(\bm{x}^T H_{\bm{\beta}}(\bm{e}_\ell) \bm{x})_{\ell = 1, \hdots, n_2} = \emptyset$.} Note that this can only happen if $n_2 \geq n_1$. We can therefore apply Lemma~\ref{lem.dimensions_varieties_not_too_big} with $V_1 = \mathbb{V}(\bm{x}^T H_{\bm{\beta}}(\bm{e}_\ell) \bm{x})_{\ell = 1, \hdots, n_2}$, $V_2 = \mathbb{V}(H_{\bm{\beta}}(\bm{y} )\bm{x})$ and $A_i = H_{\bm{\beta}}(\bm{e}_i)$ to find
	\[
	\dim \mathbb{V}(H_{\bm{\beta}}(\bm{y} )\bm{x}) \leq n_2-1 + \dim \mathbb{V}(\bm{x}^T H_{\bm{\beta}}(\bm{e}_\ell) \bm{x})_{\ell = 1, \hdots, n_2} = n_2-2.
	\]
	From this and~\eqref{eq.sigma_dimensions_bounds} the upper bound in~\eqref{eq.sigma_2_bounds} follows for this case.
	
\noindent \textbf{Case 2: $\mathbb{V}(\bm{x}^T H_{\bm{\beta}}(\bm{e}_\ell) \bm{x})_{\ell = 1, \hdots, n_2} \neq \emptyset$:} By assumption there exists $\bm{x} \in \mathbb{C}^{n_1} \setminus \{ \bm{0}\}$ such that
	\begin{equation*}
		\bm{x}^T H_{\bm{\beta}}(\bm{e}_\ell) \bm{x} = 0, \quad \text{for all $\ell = 1, \hdots, n_2$}.
	\end{equation*} 
	We claim that there exists $\bm{y} \in \mathbb{C}^{n_2} \setminus \{ \bm{0}\}$ such that $H_{\bm{\beta}}(\bm{y})\bm{x} = \bm{0}$. For this define the vectors
	\begin{equation*}
		\bm{u}_\ell = H_{\bm{\beta}}(\bm{e}_\ell) \bm{x}, \quad \ell = 1, \hdots, n_2.
	\end{equation*}
	Note that $\bm{x} \in \langle \bm{u}_1, \hdots, \bm{u}_{n_2} \rangle^\perp$ so these vectors must be linearly dependent. 	Thus there exist $y_1, \hdots, y_{n_2} \in \mathbb{C}$ not all zero, such that
	\begin{equation*}
		H_{\bm{\beta}}(\bm{y})\bm{x} = \sum_{\ell = 1}^{n_2} y_\ell H_{\bm{\beta}}(\bm{e}_\ell)\bm{x} = \bm{0},
	\end{equation*} 
	where the first equality followed since the entries of $H_{\bm{\beta}}(\bm{y})$ are linear homogeneous in $\bm{y}$. The claim follows. In particular it follows from this that
	\begin{equation*}
\left( \mathbb{V}(\bm{x}^T H_{\bm{\beta}}(\bm{e}_\ell) \bm{x})_{\ell = 1, \hdots, n_2} \times \mathbb{P}^{n_2-1} \right)\cap \mathbb{V}(H_{\bm{\beta}}(\bm{y})\bm{x}) \neq \emptyset.
	\end{equation*}	
	Using Lemma~\ref{lem.geometry_intersections} and~\eqref{eq.sigma_dimensions_bounds} we therefore find
	\begin{multline} \label{eq.dimensions_sigma_small}
		\dim \left[\left( \mathbb{V}(\bm{x}^T H_{\bm{\beta}}(\bm{e}_\ell) \bm{x})_{\ell = 1, \hdots, n_2} \times \mathbb{P}^{n_2-1} \right)\cap \mathbb{V}(H_{\bm{\beta}}(\bm{y})\bm{x})\right] \geq  \\
		\dim \mathbb{V}(H_{\bm{\beta}}(\bm{y})\bm{x}) - n_2 \geq 2 s_{\mathbb{R}}^{(2)} -n_2-2.
	\end{multline}
	Recall $\bm{\beta} \cdot \bm{F} = \bm{x}^T H_{\bm{\beta}}(\bm{y}) \bm{x}$ so that
	\begin{equation*}
		\mathrm{Sing} \mathbb{V} (\bm{\beta} \cdot \bm{F}) = \left( \mathbb{V}(\bm{x}^T H_{\bm{\beta}}(\bm{e}_\ell) \bm{x})_{\ell = 1, \hdots, n_2} \times \mathbb{P}^{n_2-1} \right) \cap \mathbb{V}(H_{\bm{\beta}}(\bm{y})\bm{x}).
	\end{equation*}
 	Under our assumptions we can apply Lemma~\ref{lem.sing_bf_small} to find $\dim \mathrm{Sing} \mathbb{V} (\bm{\beta} \cdot \bm{F}) \leq R-2$. The result follows from this and~\eqref{eq.dimensions_sigma_small}.
	\end{proof}
	
	\begin{proof}[Proof of Theorem~\ref{thm.2,1}]
	Applying Theorem~\ref{thm.n_aux_imply_result} it suffices to show
	\begin{equation} \label{eq.aux_suff}
	N_i^{\mathrm{aux}}(\bm{\beta}; B) \leq C_0 B^{2n - 4 \mathscr{C}},
	\end{equation}
	holds for all $\bm{\beta} \in \mathbb{R}^R \setminus \{ 0 \}$ and $i = 1,2$, where $\mathscr{C} > (2b+u)R$. Let 
	\begin{equation*}
		s = \max \{s_{\mathbb{R}}^{(1)}, s_{\mathbb{R}}^{(2)} \},
	\end{equation*}
	where $s_{\mathbb{R}}^{(1)}$ and $s_{\mathbb{R}}^{(2)}$ are as defined in~\eqref{eq.defn_s_1} and~\eqref{eq.defn_s_2}, respectively. From Proposition~\ref{prop.aux_counting_2-1_small} and Proposition~\ref{prop.N_2aux_small} for any $\varepsilon >0$ we get
	\begin{equation*}
		N_i^{\mathrm{aux}}(\bm{\beta}; B) \ll_\varepsilon B^{n+s+\varepsilon},
	\end{equation*} 
	with the implied constant not depending on $\bm{\beta}$. Choose $\varepsilon = \frac{n-s-(8b+4u)R}{2}$, which is a positive real number by our assumption~\eqref{eq.assumption_n_i_(2,1)}. Taking
	\begin{equation*}
		\mathscr{C} = \frac{n-s-\varepsilon}{4},
	\end{equation*}
	we see that from the assumption $n-s_{\mathbb{R}}^{(i)} > (8b+4u)R$ for $i = 1,2$ we must have $\mathscr{C} > (2b+u)R$ for this choice. Therefore~\eqref{eq.aux_suff} holds and the first part of the theorem follows upon applying Theorem~\ref{thm.n_aux_imply_result}.

For the second part recall we assume $n >(16b+8u+1)R$ and that the forms $F_i(\bm{x},\bm{y})$ define a smooth complete intersection in $\mathbb{P}_{\mathbb{C}}^{n-1} \times \mathbb{P}_{\mathbb{C}}^{n-1}$. By Lemma~\ref{lem.sigma_2-1_small} in this case we obtain
\[
s_{\mathbb{R}}^{(1)} \leq R,
\] 
and from Lemma~\ref{lem.sigma_2_bounds} we find
\[
s_{\mathbb{R}}^{(2)} \leq \frac{n+R}{2}.
\]
Therefore it is easily seen that assuming $n > (16b+8u+1)R$ implies that
\[
n-s_{\mathbb{R}}^{(i)} > (8b+4u)R
\]
holds for $i=1,2$, which is what we wanted to show.
\end{proof}

	\subsection{Proof of Theorem~\ref{thm.2,1_different_dimensions}}
	\begin{proof}[Proof of Theorem~\ref{thm.2,1_different_dimensions}]
	If $n_1 = n_2$ then the result follows immediately from Proposition~\ref{thm.2,1}. We have two cases to consider and although their strategies are very similar they are not entirely symmetric. Therefore it is necessary to consider them individually.
	
	\noindent \textbf{Case 1:} $n_1 > n_2$. We consider a new system of equations $\widetilde{F}_i({\bm{x}}, \tilde{\bm{y}})$ in the variables ${\bm{x}} = (x_1, \hdots, x_n)$ and $\tilde{\bm{y}} = (y_1, \hdots, y_{n_2}, y_{n_2+1}, \hdots, y_{n_1})$ where the forms $\widetilde{F}_i({\bm{x}}, \tilde{\bm{y}})$ satisfy \[ \widetilde{F}_i({\bm{x}}, \tilde{\bm{y}}) = F(\bm{x},\bm{y}),
	\]
	where $\bm{y} = (y_1, \hdots, y_{n_2})$. 
	Write $\widetilde{N}(P_1,P_2)$ for the counting function associated to the system $\widetilde{\bm{F}} = \bm{0}$ and the boxes $\mathcal{B}_1 \times (\mathcal{B}_2 \times [0,1]^{n_1-n_2})$. Note in particular, that if we replace $F$ by $\widetilde{F}$ in~\eqref{eq.sing_series_real_density} and~\eqref{eq.sing_series_p_adic_expression} then the expressions for the singular series and the singular integral remain unchanged.
	Further denote by $\tilde{s}_{\mathbb{R}}^{(i)}$ the quantities defined in~\eqref{eq.defn_s_1} and~\eqref{eq.defn_s_2} but with $F$ replaced by $\widetilde{F}$.
	Note that we have $\tilde{s}_{\mathbb{R}}^{(1)} = s_{\mathbb{R}}^{(1)}$ and $\tilde{s}_{\mathbb{R}}^{(2)} \leq s_{\mathbb{R}}^{(2)} + \frac{n_1-n_2}{2}$. Therefore the assumptions~\eqref{eq.assumption_n_i_(2,1)_introduction} imply
	\[
	n_1 - \tilde{s}_{\mathbb{R}}^{(i)} > (8b+4u)R
	\]
	for $i=1,2$. Hence we may apply Proposition~\ref{thm.2,1} in order to obtain
	\[
	\widetilde{N}(P_1,P_2) = {\mathfrak{I}} {\mathfrak{S}}P_1^{n_1-2R} P_2^{n_1-R} + O(P_1^{n_1-2R} P_2^{n_1-R} \min\{P_1,P_2\}^{-\delta}),
	\]
	for some $\delta >0$. Finally it is easy to see that 
	\begin{align*}
	\widetilde{N}(P_1,P_2) &= N(P_1,P_2) \# \left\{ \bm{t} \in \Z^{n_1-n_2} \cap [0,P_2]^{n_1-n_2} \right\} \\
	&= N(P_1,P_2) (P_2^{n_1-n_2} + O(P_2^{n_1-n_2-1})),
	\end{align*}
	and so~\eqref{eq.asymptotic_bidegree_2_1} follows.
	
	\noindent \textbf{Case 2:} $n_2 > n_1$ We deal with this very similarly as in the first case; we define a new system of forms $\widetilde{F}_i(\tilde{\bm{x}}, \bm{y})$ in the variables $\tilde{\bm{x}} = (x_1, \hdots, x_{n_2})$ and $\bm{y} = (y_1, \hdots, y_{n_2})$ such that
	\[
	\widetilde{F}_i({\bm{x}}, \tilde{\bm{y}}) = F_i(\bm{x},\bm{y})
	\]
	holds. As before we define a new counting function $\widetilde{N}(P_1, P_2)$ with respect to the new product of boxes  $\left(\mathcal{B}_1 \times [0,1]^{n_2-n_1}\right) \times \mathcal{B}_2$, and we define $\tilde{s}_{\mathbb{R}}^{(i)}$ similarly to the previous case. Note that $\tilde{s}_{\mathbb{R}}^{(1)} = s_{\mathbb{R}}^{(1)}+n_2-n_1$ and $\tilde{s}_{\mathbb{R}}^{(2)} \leq s_{\mathbb{R}}^{(2)} + \frac{n_2-n_1}{2}$ so that~\eqref{eq.assumption_n_i_(2,1)_introduction} gives
	\[
	n_2 - \tilde{s}_{\mathbb{R}}^{(i)} > (8b+4u)R,
	\]
	for $i=1,2$.
	Therefore Proposition~\ref{thm.2,1} applies and we deduce again that~\eqref{eq.asymptotic_bidegree_2_1} holds as desired.
	
	Finally we turn to the case when $\mathbb{V}(\bm{F})$ defines a smooth complete intersection. Note first that by Lemma~\ref{lem.sigma_2_bounds} we have
	\[
	s_{\mathbb{R}}^{(2)} \leq \frac{n_2+R}{2},
	\]
	and therefore the condition
	\[
	\frac{n_1+n_2}{2} - s_{\mathbb{R}}^{(2)} > (8b+4u)R
	\]
	is satisfied if we assume $n_1 > (16b+8u+1)R$. Further, by Lemma~\ref{lem.sigma_2-1_small} we have
	\[
	s_{\mathbb{R}}^{(1)} \leq \max\{0, n_1+R-n_2 \},
	\]
	and so we may replace the condition $n_1 - s_{\mathbb{R}}^{(1)} > (8b+4u)R$ by
	\[
	n_1- \max\{0, n_1+R-n_2 \} > (8b+4u)R.
	\]
	If $n_2 \geq n_1+R$ then this reduces to assuming $n_1 > (8b+4u+1)R$, which follows immediately since we assumed $n_1 > (16b+8u+1)R$. If $n_2 \leq n_1+R$ on the other hand, then this is equivalent to assuming
	\[
	n_2 > (8b+4u+1)R.
	\]
	In any case, the assumptions~\eqref{eq.assumptions_n_i_(2,1)_smooth_case} imply the assumptions~\eqref{eq.assumption_n_i_(2,1)_introduction} as desired.
	\end{proof}

\bibliography{refs.bib}
\bibliographystyle{plain}
\end{document}